\documentclass{article}

\usepackage{amsthm}
\usepackage{amsmath}
\usepackage{amssymb}
\usepackage{graphicx}
\usepackage{color}
\usepackage{bbm}
\usepackage[top=2.5cm, bottom=2.4cm, left=2.5cm, right=2.5cm]{geometry}
\usepackage{enumerate}

\newcommand{\bigO}{\mathcal{O}}
\newcommand{\I}{\mathrm{Im}}
\newcommand{\Ima}{\mathfrak{Im}}
\newcommand{\N}{\mathcal{N}}
\newcommand{\R}{\mathbb{R}}
\newcommand{\id}{\mathrm{id}}
\newcommand{\C}{\mathbb{C}}
\newcommand{\T}{\mathbb{T}}

\newcommand{\D}{\mathcal{D}}
\newcommand{\tr}{\mathrm{tr}}
\newcommand{\vol}{\mathrm{vol}}
\newcommand{\st}{\Sigma_\tau}
\newcommand{\so}{\Sigma_0}
\newcommand{\nt}{n_{\Sigma_\tau}}
\newcommand{\no}{n_{\Sigma_0}}

\newcommand{\utl}{\tilde{u}_\lambda}
\newcommand{\vl}{v_\lambda}
\newcommand{\supp}{\mathrm{supp}}

\newcommand{\vgt}{\mathrm{vol}_{\bar{g}_\tau}}
\newcommand{\vgo}{\mathrm{vol}_{\bar{g}_0}}
\newcommand{\vg}{\mathrm{vol}_g}
\newcommand{\Rot}{R_{[0,\tau]}}
\newcommand{\grad}{\mathrm{grad}}
\newcommand{\pk}{\partial_\kappa}
\newcommand{\pr}{\partial_\rho}
\newcommand{\pmu}{\partial_\mu}
\newcommand{\pn}{\partial_\nu}
\newcommand{\Vt}{\tilde{V}}
\newcommand{\Jt}{\tilde{J}}
\newcommand{\Real}{\mathfrak{Re}}
\newcommand{\uln}{u_{\lambda,\N}}
\newcommand{\elp}{e^{-2\lambda \phi_2}}
\newcommand{\End}{\mathrm{End}}
\newcommand{\Pb}{\mathbb{P}}
\newcommand{\Db}{\mathbb{D}}

\makeatletter
\newcommand{\LeftEqNo}{\let\veqno\@@leqno}
\makeatother

\begin{document}

\numberwithin{equation}{section}
\newtheorem{theorem}[equation]{Theorem}
\newtheorem{remark}[equation]{Remark}
\newtheorem{claim}[equation]{Claim}
\newtheorem{lemma}[equation]{Lemma}
\newtheorem{definition}[equation]{Definition}
\newtheorem{conjecture}[equation]{Conjecture}
\newtheorem{proposition}[equation]{Proposition}

\title{Characterisation of the Energy of Gaussian Beams on Lorentzian Manifolds - with Applications to Black Hole Spacetimes}
\author{Jan Sbierski\thanks{Department of Applied Mathematics and Theoretical Physics, University of Cambridge,
Wilberforce Road,
Cambridge,
CB3 0WA,
United Kingdom}}
\date{\today}

\maketitle

\begin{abstract}

It is known that using the Gaussian beam approximation one can show that there exist solutions of the wave equation on a general globally hyperbolic Lorentzian manifold whose energy is localised along a given null geodesic for a finite, but arbitrarily long time.
In this paper, we show that the energy of such a localised solution is determined by the energy of the underlying null geodesic. This result opens the door to various applications of Gaussian beams on Lorentzian manifolds that \emph{do not admit a globally timelike Killing vector field}. In particular we show that trapping in the exterior of Kerr or at the horizon of an extremal Reissner-Nordstr\"om black hole necessarily leads to a `loss of derivative' in a local energy decay statement.  We also demonstrate the obstruction formed by the red-shift effect at the event horizon of a Schwarzschild black hole to scattering constructions from the future (where the red-shift  turns into a blue-shift): we construct solutions to the backwards problem whose energies grow exponentially for a finite, but arbitrarily long time.  Finally, we give a simple mathematical realisation of the heuristics for the blue-shift effect near the Cauchy horizon of sub-extremal and extremal black holes: we construct a sequence of solutions to the wave equation whose initial energies are uniformly bounded, whereas the energy near the Cauchy horizon goes to infinity.
\end{abstract}

\tableofcontents

\section{Introduction}

Part I of this paper is concerned with the study of the temporal behaviour of Gaussian beams on \emph{general} globally hyperbolic Lorentzian manifolds. Here, a Gaussian beam is a highly oscillatory wave packet of the form
\begin{equation*}
\tilde{u}_\lambda =\frac{1}{\sqrt{E(\lambda,a,\phi)}} \cdot a \cdot e^{i\lambda \phi}\;,
\end{equation*}
where \(E(\lambda, a, \phi)\) is a renormalisation factor keeping the initial energy of \(\tilde{u}_\lambda\) independent of \(\lambda \in \R^+\), and the complex valued functions \(a\) and \(\phi\) are chosen in such a way that for \(\lambda \gg 0\) the Gaussian beam \(\tilde{u}_\lambda\) is an approximate solution to the wave equation on the underlying Lorentzian manifold \((M,g)\).
The failure of \(\tilde{u}_\lambda\) being an actual solution to the wave equation
\begin{equation}
\label{waveint}
\Box_g u =0
\end{equation}
is measured in terms of an energy norm - and this error can be made arbitrarily small up to a finite, but arbitrarily long time by choosing \(\lambda\) large enough. The construction of the functions \(a\) and \(\phi\) allows for restricting the support of \(a\) to a small neighbourhood of a given null geodesic. Thus, one can infer from \(\tilde{u}_\lambda\) being an approximate solution with respect to some energy norm, that\footnote{Cf. Theorem \ref{localised}.}
\vspace{2pt} 
\begin{equation}
\label{FirstStatement}
\begin{split}
&\textnormal{there exist \emph{actual} solutions of the wave equation \eqref{waveint} whose `energy' is localised} \\[-2pt] &\textnormal{along a given null geodesic up to some finite, but arbitrarily long time.}
\end{split}
\end{equation}
\vspace{2pt}
This is, roughly, the state of the art knowledge of Gaussian beams (see for instance \cite{Ral83}).

The main new result of Part I of this paper is to provide a geometric characterisation of the temporal behaviour of the localised energy of a Gaussian beam. More precisely, given a timelike vector field \(N\) (with respect to which we measure the energy) and a Gaussian beam \(\tilde{u}_\lambda\) supported in a small neighbourhood of an affinely parametrised null geodesic \(\gamma\), we show in Theorem \ref{main} that
\begin{equation}
\label{expo}
\int_{\st} J^N(\tilde{u}_\lambda) \cdot \nt \approx -g(N,\dot{\gamma})\big|_{\I(\gamma) \cap \st}
\end{equation}
holds up to some finite time \(T\). Here, we consider a foliation of the Lorentzian manifold \((M,g)\) by spacelike slices \(\st\), \(J^N(\tilde{u}_\lambda)\) denotes the contraction of the stress-energy tensor\footnote{We refer the reader to \eqref{seten} in Section \ref{notation} for the definition of the stress-energy tensor.} of \(\tilde{u}_\lambda\) with \(N\), and \(\nt\) is the normal of \(\st\). The left hand side of \eqref{expo} is called the \emph{\(N\)-energy} of the Gaussian beam \(\tilde{u}_\lambda\). The approximation in \eqref{expo} can be made arbitrarily good and the time \(T\) arbitrarily large if we only take \(\lambda >0\) to be big enough.
This characterisation of the energy allows then for a refinement of \eqref{FirstStatement}:\footnote{Cf. Theorem \ref{symbiosis}.}
\vspace{2pt} 
\begin{equation}
\label{SecondStatement}
\begin{split}
&\textnormal{There exist (actual) solutions of the wave equation \eqref{waveint} whose \(N\)-energy is localised along}  \\[-2pt] &\textnormal{a given null geodesic \(\gamma\) and behaves approximately like \(-g(N,\dot{\gamma})\big|_{\I(\gamma) \cap \st}\) up to some finite,} \\[-2pt] &\textnormal{but arbitrarily large time \(T\). Here, \(\dot{\gamma}\) is with respect to some affine parametrisation of \(\gamma\).} 
\end{split}
\end{equation}
It is worth emphasising that the need for an understanding of the temporal behaviour of the energy only arises for Gaussian beams on Lorentzian manifolds that \emph{do not admit a globally timelike Killing vector field\footnote{One could add here `\emph{uniformly}' timelike, meaning that the timelike Killing vector field does not `degenerate' when approaching the `boundary' of the manifold. Let us just state here that one can give precise meaning to `degenerating at the boundary'.}} - otherwise there is a canonical energy which is conserved for solutions to the wave equation \eqref{waveint}. Thus, for the majority of problems which so far found applications of Gaussian beams, for example the obstacle problem or the wave equation in time-independent inhomogeneous media, the question of the temporal behaviour of the energy did not arise (since it is trivial). However, understanding this behaviour on general Lorentzian manifolds is crucial for widening the application of Gaussian beams to problems arising, in particular, from general relativity.

In Part II of this paper, by applying \eqref{SecondStatement}, we derive some new results on the study of the wave equation on the familiar Schwarzschild, Reissner-Nordstr\"om, and Kerr black hole backgrounds (see \cite{HawkEllis} for an introduction to these spacetimes):
\begin{enumerate}
\item It is well-known folklore that the trapping\footnote{We do not intend to give a precise definition in this paper of what we mean by `trapping'. However, loosely speaking `trapping' refers here to the presence of null geodesics that stay for all time in a compact region of `space'.} at the \emph{photon sphere} in Reissner-Nordstr\"om and in Kerr necessarily leads to a `loss of derivative' in a local energy decay (LED) statement. We give a rigorous proof of this fact.
\item We also show that the trapping at the \emph{horizon} of an extremal Reissner-Nordstr\"om (and Kerr) black hole necessarily leads to a loss of derivative in an LED statement. 
\item When solving the wave equation \eqref{waveint} on the exterior of a Schwarzschild black hole backwards in time, the red-shift effect at the event horizon turns into a blue-shift: we construct solutions to the backwards problem whose energies grow exponentially for a finite, but arbitrarily long time. This demonstrates the obstruction formed by the red-shift effect at the event horizon to scattering constructions from the future.
\item Finally, we give a simple mathematical realisation of the heuristics for the blue-shift effect near the Cauchy horizon of (sub)-extremal Reissner-Nordstr\"om and Kerr black holes: we construct a sequence of solutions to the wave equation whose initial energy is uniformly bounded whereas the energy near the Cauchy horizon goes to infinity.
\end{enumerate}

\subsubsection*{Outline of the paper}

We start by giving a short historical review of Gaussian beams in Section \ref{review}. Thereafter we briefly explain how the notion of `energy' arises in the study of the wave equation and why it is important. We also discuss how the results obtained in this paper allow to disprove certain uniform statements about the temporal behaviour of the energy of waves. Section \ref{parsimonious} elaborates on the wide applicability of the Gaussian beam approximation and explains its advantage over the geometric optics approximation. In the physics literature a similar `characterisation of the energy of high frequency waves' is folklore - we discuss its origin in Section \ref{HF} and put it into context with the work presented in this paper. Section \ref{notation} lays down the notation we use. 

Part I of this paper discusses the theory of Gaussian beams on Lorentzian manifolds. Sections \ref{underlying} and \ref{GaussianBeams} establish Theorem \ref{localised} which basically says \eqref{FirstStatement} and is more or less well-known. Although the proof of Theorem \ref{localised} can be reconstructed from the literature (cf.\ especially \cite{Ral83}), we could not find a complete and self-contained proof of this statement. Moreover, there are some important subtleties (cf.\ footnote \ref{Subtlety}) which are not discussed elsewhere. For these reasons, and moreover for making the paper self-contained, we have included a full proof of Theorem \ref{localised}. In Section \ref{CharacterisationEnergy} we characterise the energy of a Gaussian beam, which is the main result of Part I of this paper. This result is then incorporated into Theorem \ref{localised}, which yields Theorem \ref{symbiosis} (or \eqref{SecondStatement}). Moreover, Section \ref{generalthms} contains some general theorems which are tailored to the needs of many applications.

In Part II of this paper, we prove the above mentioned new results on the behaviour of waves on various black hole backgrounds. The important ideas are first introduced in Section \ref{RNSec} by the example of the Schwarzschild and Reissner-Nordstr\"om family, whose simple form of the metric allows for an uncomplicated presentation. Thereafter, in Section \ref{KerrSec}, we proceed to the Kerr family.

The main purpose of the appendix is to contrast the Gaussian beam approximation with the much older geometric optics approximation. In Appendix \ref{geomop}, we recall the basics of the geometric optics approximation. Appendix \ref{Comparison} discusses Ralston's work \cite{Ral69} from 1969, where he showed that trapping in the obstacle problem necessarily leads to a loss of differentiability in an LED statement.
We conclude in Appendix \ref{BD} with giving a sufficient criterion for the formation of caustics, i.e., a breakdown criterion for solutions of the eikonal equation.

\subsection{A brief historical review of Gaussian beams}
\label{review}

The ansatz
\begin{equation}
\label{geomopansatz}
u_\lambda =  e^{i\lambda \phi} \big(a_0 + \frac{1}{\lambda} a_1 + \ldots + \frac{1}{\lambda^N} a_N \big)
\end{equation}
for either an highly oscillatory approximate solution to some PDE or for an highly oscillatory approximate eigenfunction to some partial differential operator, is known as the \emph{geometric optics ansatz}.
Here, \(N \in \mathbb{N}\), \(\phi\) is a real function (called the eikonal), the \(a_k\)'s are complex valued functions, and \(\lambda\) is a positive parameter determining how quickly the function \(u_\lambda\) oscillates. In the widest sense, we understand under a Gaussian beam a function of the form \eqref{geomopansatz} with a \emph{complex} valued eikonal \(\phi\) that is real valued along a bicharacteristic and has growing imaginary part off this bicharacteristic. This then leads to an exponential fall off in \(\lambda\) away from the bicharacteristic. 

The use of a complex eikonal, although in a slightly different context, appears already in work of Keller from 1956, see \cite{Kell56}. It was, however, only in the 1960's that the method of Gaussian beams was systematically applied and explored - mainly from a physics perspective. For more on these early developments we refer the reader to \cite{Arn73}, Chapter 4, and references therein. A general, mathematical theory of Gaussian beams, or what he called the complex WKB method, was developed by Maslov, see the book \cite{Mas94} for an overview and also for references therein. Several of the later papers on Gaussian beams have their roots in this work. 

The earliest application of the Gaussian beam method was to the construction of \emph{quasimodes}, see for example the paper \cite{Ral76} by Ralston from 1976. Quasimodes approximately satisfy some type of Helmholtz equation, and thus they give rise to time harmonic, approximate solutions to a wave equation. In this way quasimodes can be interpreted as standing waves. Later, various people used the Gaussian beam method for the construction of Gaussian \emph{wave packets} (but also called `Gaussian beams') which form approximate solutions to a hyperbolic PDE\footnote{It is this sort of `Gaussian beam' that is the subject of this paper for the case of the wave equation on Lorentzian manifolds. More appropriately, one could name them `Gaussian wave packets' or `Gaussian pulses' to distinguish them from the standing waves - which are actually \emph{beams}. However, we stick to the standard terminology.}. Those wave packets, in contrast to quasimodes, are not stationary waves, but they  move through space, the trajectory in spacetime being a bicharacteristic of the partial differential operator. A detailed reference for this construction is the work \cite{Ral83} by Ralston, which goes back to 1977. Another presentation of this construction scheme was given in 1981 by Babich and Ulin, see \cite{BabUl81}. 

Since then, there have been a lot of papers applying Gaussian beams to various problems\footnote{We refer the reader to \cite{Mas94} for a list of references.}. For instance, in quantum mechanics Gaussian beams correspond to semiclassical approximate solutions to the Schr\"odinger equation and thus help understand the classical limit; or in geophysics, one models seismic waves using the Gaussian beam approximation for solutions to a wave equation in an inhomogeneous (time-independent) medium.

\subsection{Gaussian beams and the energy method}
\label{energymethodGB}

\subsubsection{The energy method as a versatile method for studying the wave equation}

The study of the wave equation on various geometries has a long history in mathematics and physics.  A very successful and widely applicable method for obtaining quantitative results on the long-time behaviour of waves is the energy method. It was pioneered by Morawetz in the papers \cite{Mora61} and \cite{Mora62}, where she proved pointwise decay results in the context of the obstacle problem. In \cite{Mora68} she established what is now known as \emph{integrated local energy decay} (ILED) for solutions of the Klein-Gordon equation (and thus inferring decay). 
In the past ten years her methods were adapted and extended by many people in order to prove boundedness and decay of waves on various (black hole) spacetimes - a study which is mainly motivated by the black hole stability conjecture (cf. the introduction of \cite{DafRod08}). A small selection of examples is \cite{Klain85}, \cite{DafRod09a}, \cite{DafRod11}, \cite{DafRod10a}, \cite{DafRod10b}, \cite{AndBlue09}, \cite{TatToh08}, \cite{Luk10}, \cite{Schl10}, \cite{Are11a}, \cite{HolzSmu11}, \cite{Civ14} and \cite{Dya11}.

The philosophy of the energy method is first to derive estimates on a suitable energy (and higher order energies\footnote{A \emph{first order} energy controls the \emph{first} derivatives of the wave and is referred to in the following just as `energy'. \emph{Higher order} energies control \emph{higher} derivatives of the wave. A special case of the energy method is the so-called \emph{vector field method}. Higher order energies arise there naturally by commutation with suitable vector fields, see \cite{Klain85}.}) and then to establish pointwise estimates using Sobolev embeddings. Thus, given a spacetime on which one intends to study the wave equation using the energy method, one first has to set up such a suitable energy (and higher order energies - but in this paper we focus on the first order energy). A general procedure is to construct an energy from a foliation of the spacetime by spacelike slices \(\st\) together with a timelike vector field \(N\), see \eqref{N-energy} in Section \ref{notation}. We refrain from discussing here what choices of foliation and timelike vector field lead to a `suitable' notion of energy\footnote{However, see Section \ref{applications} for some examples and footnote \ref{suitable} for some further comments.}. Let us just mention here that in the presence of a globally timelike Killing vector field \(T\) one obtains a particularly well behaved energy by choosing \(N = T\) and a foliation that is invariant under the flow of \(T\).\footnote{Such a choice corresponds to what we denoted in the introduction as a  `canonical energy'.} We invite the reader to convince him- or herself that the familiar notions of energy for the wave equation on the Minkowski spacetime or in time independent inhomogeneous media arise as special cases of this more general scheme. 

\subsubsection{Gaussian beams as obstructions to certain uniform behaviour of the energy of waves}

The approximation with Gaussian beams allows to construct solutions to the wave equation whose energy is localised for an arbitrarily long, but finite time along a null geodesic. Such solutions form naturally an obstruction to certain uniform statements about the temporal behaviour of the energy of waves. A classical example is the case in which one has a null geodesic that does not leave a compact region in `space' and which has constant energy\footnote{We refer to the right hand side of \eqref{expo} as the \(N\)-energy of the null geodesic.}. Such null geodesics form obstructions to certain formulations of local energy decay being true\footnote{A classic regarding such a result is the work \cite{Ral69} by Ralston. However, he does not use the Gaussian beam approximation in this work, but the geometric optics approximation. We discuss his work in some detail in Appendix \ref{Comparison}.}. However, it is very important to be aware of the fact, that in general none of the solutions from \eqref{SecondStatement} has localised energy for \emph{all} time. Thus, in order to contradict, for instance, an LED statement, it is in general inevitable to resort to a \emph{sequence} of solutions of the form \eqref{SecondStatement} which exhibit the contradictory behaviour in the limit. For this scheme to work, however, it is clearly crucial that the LED statement in question is \emph{uniform} with respect to some energy which is left constant by the sequence of Gaussian beam solutions. Note here that \eqref{SecondStatement} in particular states that the time \(T\), up to which one has good control over the wave, can be made arbitrarily large without changing the initial energy! Higher order initial energies, however, will blow up when \(T\) is taken bigger and bigger.  In this paper we restrict our consideration to disproving statements that are uniform with respect to the first order energy.
In Sections \ref{photonsphere}, \ref{extremal} and \ref{SecKerrTrap} we demonstrate this important application of Gaussian beams: we show that certain (I)LED statements derived by various people in the presence of `trapping' are sharp in the sense that \emph{some} loss of derivative is necessary (however, one does not necessarily need to lose a whole derivative, cf.\ the discussion at the end of Section \ref{photonsphere}).

We conclude this section with the remark that in the presence of a globally timelike Killing vector field one can already infer such obstructions from \eqref{FirstStatement}, since the (canonical) energy of solutions to the wave equation is then constant. In this way one can easily infer from \eqref{FirstStatement} alone that an LED statement in Schwarzschild has to lose differentiability due to the trapping at the photon sphere.  But already for trapping in Kerr one needs to know how the `trapped' energy of the solutions referred to in \eqref{FirstStatement} behaves in order to infer the analogous result. This knowledge is provided by \eqref{expo} and/or \eqref{SecondStatement}.

\subsection{Gaussian beams are parsimonious}
\label{parsimonious}

The approximation by Gaussian beams can be carried out on a Lorentzian manifold \((M,g)\) under minimal assumptions:
\begin{enumerate}
\item 
One needs a well-posed initial value problem. This is ensured by requiring that \((M,g)\) is globally hyperbolic\footnote{The assumption of global hyperbolicity has another simplifying, but not essential, feature; cf. the discussion after Definition \ref{defgauss}.}. However, one can also replace the well-posed initial value problem by a well-posed initial-boundary value problem - and one can obtain, with small changes and some additional work in the proof, qualitatively the same results.
\item  
Having fixed an \(N\)-energy one intends to work with, one has to have an energy estimate of the form \eqref{energyestimate} at one's disposal, which is guaranteed by the condition \eqref{energyestimatecond}. The estimate \eqref{energyestimate} allows us to infer that the approximation by the Gaussian beam is \emph{global} in space. In particular, it is only under this condition that it is justified to say in \eqref{FirstStatement} and \eqref{SecondStatement} that the energy of the actual solution is \emph{localised} along a null geodesic\footnote{\label{Subtlety}That one needs condition \eqref{energyestimatecond} for ensuring that the energy is indeed localised is in fact another minor novelty in the study of Gaussian beams on \emph{general} Lorentzian manifolds (note that in the case of \(N\) being a Killing vector field, condition \eqref{energyestimatecond} is trivially satisfied). For an example for a violation of condition \eqref{energyestimatecond} we refer to the discussion after \eqref{energyboundimp} on page \pageref{energyboundimp}.}. However, as we show in Remark \ref{rem} one \emph{always} has a \emph{local} approximation, which is, together with the geometric characterisation of the energy, sufficient for obtaining control of the wave in a small neighbourhood of the underlying null geodesic \emph{regardless of condition} \eqref{energyestimatecond}. This then allows to establish, for example, the very general Theorem \ref{LED} which only requires global hyperbolicity (or some other form of well-posedness for the wave equation, cf.\ point 1.).
\end{enumerate}
In particular the method of Gaussian beams is not in need of any special structure on the Lorentzian manifold like Killing vector fields (as for example needed for the mode analysis or for the construction of quasimodes).

We would also like to emphasise here that in order to apply \eqref{SecondStatement} \emph{one only needs to understand the behaviour of the null geodesics of the underlying Lorentzian manifold}! This knowledge is often in reach and thus Gaussian beams provide in many cases an easy and feasible way for obtaining control of highly oscillatory solutions to the wave equation. In this sense the theory presented in Part I of this paper forms a good `black box result' which can be applied to various different problems.
\newline
\newline
We conclude this section with a brief comparison of the Gaussian beam approximation with the geometric optics approximation: 
Let us call the geometric optics approximation, which considers approximate solutions of the form \eqref{geomopansatz}, the `naive' geometric optics approximation. Although it applies under the same general conditions as the Gaussian beam approximation, in general the time \(T\), up to which one has good control over the solution, cannot be chosen arbitrarily large since the approximate solution breaks down at caustics. In Appendix \ref{BD} we show that caustics necessarily form along null geodesics that possess conjugate points. A prominent example of such null geodesics are the trapped null geodesics at the photon sphere in the Schwarzschild spacetime (cf. Section \ref{photonsphere} for the proof that these null geodesics have conjugate points). However, the formation of caustics is not a serious limitation of the geometric optics approximation, since one can extend the approximate solution through the caustics making use of Maslov's canonical operator. The approximate solution obtained in this way is, however, no longer of the simple form \eqref{geomopansatz}. The advantage of the Gaussian beam approximation is that the simple ansatz \eqref{geomopansatz} does not break down at caustics, it yields an approximation up to all finite time $T$.

\subsection{`High frequency' waves in the physics literature}
\label{HF}

In physics, the notion of a \emph{local observer's energy} arose with the emergence of Einstein's theory of relativity: Suppose an observer travels along a timelike curve \(\sigma : I \to M\) with unit velocity \(\dot{\sigma}\). Then, with respect to a Lorentz frame of his, he measures the local energy density of a wave \(u\) to be \(\mathbb{T}(u)(\dot{\sigma},\dot{\sigma})\), where \(\mathbb{T}(u)\) is the stress-energy  tensor of the wave \(u\), cf.\ \eqref{seten} in Section \ref{notation}. By considering the three parameter family of observers whose velocity vector field is given by the normal \(\nt\) to a foliation of \(M\) by spacelike slices \(\st\), the physical definition of energy is contained in the mathematical one (which is given by \eqref{N-energy}).

The prevalent description of highly oscillatory (or `high frequency') waves in the physics literature is that the waves (or `photons') propagate along null geodesics \(\gamma\) and each of these rays (or photons) carries an energy-momentum \(4\)-vector \(\dot{\gamma}\), where the dot is with respect to some affine parametrisation. In the high frequency limit, the number of photons is preserved. Thus, the energy of the wave, as measured by a local observer with world line \(\sigma\), is determined by the energy component \(-g(\dot{\gamma},\dot{\sigma})\) of the momentum \(4\)-vector \(\dot{\gamma}\). By considering a highly oscillatory wave that `gives rise to just one photon', one recovers the characterisation of the energy of a Gaussian beam, \eqref{expo}, given in this paper.

In the physics literature (see for example the classic \cite{MTW}, Chapter \(22.5\)), this description is justified using the naive geometric optics approximation. For the following brief discussion we refer the reader unfamiliar with the geometric optics approximation to Appendix \ref{geomop}. 

The conservation law 
\begin{equation}
\label{importantcons}
\mathrm{div} \, (a^2 \grad\,\phi) = 0 \;,
\end{equation}
which can be easily inferred from equation \eqref{transport}, is interpreted as the conservation of the number-flux vector \(S = a^2 \grad\,\phi\) of the photons. The leading component in \(\lambda\) of the renormalised\footnote{i.e.\ divided by \(\lambda^2\)} stress-energy tensor \(T(u_\lambda)\) of the wave \(u_\lambda = a \cdot e^{i \lambda \phi}\) in the geometric optics limit is then given by
\begin{equation*}
\mathbb{T}(u_\lambda) = \grad\, \phi \otimes S \;,
\end{equation*}
from which it then follows that each photon carries a \(4\)-momentum \(\grad \,\phi = \dot{\gamma}\).

In particular, making use of the conservation law \eqref{importantcons}, it is not difficult\footnote{Although, to the best of our knowledge, it is nowhere done explicitly.} to prove a geometric characterisation of the energy of waves in the naive geometric optics limit analogous to the one we prove in this paper for Gaussian beams. However, as we have mentioned in the previous section, the naive geometric optics approximation has the undesirable feature that it breaks down at caustics. 

The characterisation of the energy of Gaussian beams is more difficult since \eqref{importantcons} is replaced only by an \emph{approximate} conservation law\footnote{Cf.\ the discussion below \eqref{secondcons} in Section \ref{CharacterisationEnergy}.}. Moreover, it provides a rigorous justification of the temporal behaviour of the local observer's energy of photons, \emph{which also applies to photons along whose trajectory caustics would form}.

\subsection{Notation}
\label{notation}

Given a Lorentzian manifold $(M,g)$, we denote the canonical isomorphisms induced by the metric $g$ between the tangent and cotangent space by $\sharp : T_x^*M \to T_xM$ and $\flat : T_xM \to T_x^*M$, where $x \in M$ and for $\alpha \in T_x^*M$ and $X \in T_xM$  the isomorphisms $\sharp$ and $\flat$ are given by $\alpha^\sharp := g^{-1}(\alpha, \cdot)$ and $X^\flat := g(X,\cdot)$. Here $g^{-1}$ denotes the inverse of the metric $g$.
Moreover, we denote with $\cdot$ the inner product of two vectors as well as the inner product of two covectors, i.e.\ for $\alpha, \beta \in T^*_xM$ we write $\alpha \cdot \beta := g^{-1}(\alpha, \beta)$, and for $X, Y \in T_xM$ we write $X \cdot Y := g(X,Y)$. We also introduce the notation $\grad \,f := (df)^\sharp$ for the gradient of a function $f \in C^\infty(M, \R)$. The Levi-Civita connection on the Lorentzian manifold $(M,g)$ is denoted by $\nabla$, and we write $\mathrm{div}\, Z := \nabla_\mu Z^\mu$ for the divergence of a smooth vector field $Z$ on $M$. 
Furthermore, we define the \emph{wave operator} \(\Box_g\) by
\begin{equation*}
\Box_g u :=  \nabla^\mu \nabla_\mu u \;.
\end{equation*}
From here on we will however omit the index \(g\) on \(\Box_g\), since it is clear from the context which Lorentzian metric is referred to.

Whenever we are given a time oriented Lorentzian manifold \((M,g)\) that is (partly) foliated by spacelike slices \(\{\Sigma_\tau\}_{\tau \in [0,\tau^*)}\), \(0 < \tau^* \leq \infty\), we denote the future directed unit normal to the slice \(\Sigma_\tau\) with \(n_{\Sigma_\tau}\). Moreover, the induced Riemannian metric on \(\Sigma_\tau\) is then denoted by \(\bar{g}_\tau\) and we set \(R_{[0,T]}:=\bigcup_{0 \leq \tau \leq T} \Sigma_\tau\).

For \(u \in C^\infty(M,\C)\) we define the \emph{stress-energy tensor} \(\T(u)\) 
by
\begin{equation}
\label{seten}
\T(u):=\frac{1}{2}\overline{du} \otimes du + \frac{1}{2}du \otimes \overline{du} - \frac{1}{2}g(\cdot,\cdot)  g^{-1}(du,\overline{du})\;.
\end{equation}
Given in addition a vector field $N$, we also define the current $J^N(u)$ by
\[J^N(u):= \big[\T(u)(N,\cdot)\big]^\sharp\;.\]
Finally, if \(N\) is future directed timelike, we call
\begin{equation} 
\label{N-energy}
E_\tau^N(u):=\int_{\Sigma_\tau}J^N(u) \cdot n_{\Sigma_\tau} \vol_{\bar{g}_\tau}
\end{equation}
the \emph{\(N\)-energy} of \(u\) at time \(\tau\), where \(\vol_{\bar{g}_\tau}\) denotes the volume element corresponding to the metric \(\bar{g}_\tau\).\footnote{Cf.\ also \cite{ChoquetGR}, Appendix III, Sections 2.3 and 2.4 (in particular Definition (2.27)) for a detailed discussion of the notion of $N$-energy.} If \(A \subseteq \Sigma_\tau\), then \(E_{\tau,A}^N(u)\) denotes the \(N\)-energy of \(u\) at time \(\tau\) in the volume \(A\), i.e., the integration in \eqref{N-energy} is only over \(A\). 

The notion \eqref{N-energy} of the \(N\)-energy of a function \(u\) is especially helpful whenever we have an adequate knowledge of \(\Box u\), since one can then infer detailed information about the behaviour of the \(N\)-energy (cf.\ the energy estimate \eqref{energyestimate} in the next section), and thus also about the behaviour of \(u\) itself. Hence, the stress-energy tensor \eqref{seten} together with the notion of the \(N\)-energy is particularly useful for solutions \(u\) of the \emph{wave equation}
\begin{equation}
\label{waveeq}
\Box u =0.
\end{equation}
For more on the stress-energy tensor and the notion of energy, we refer the reader to \cite{Taylor1}, chapters 2.7 and 2.8.

Given a Lorentzian manifold \((M,g)\) and \(A \subseteq M\), we denote with \(J^+(A)\) the causal future of \(A\), i.e., all the points \(x\in M\) such that there exists a future directed causal curve starting at some point of \(A\) and ending at \(x\). The causal past of \(A\), \(J^-(A)\), is defined analogously\footnote{See also Chapter \(14\) in \cite{ONeill}.}. Finally, \(C\) and \(c\) will always denote positive constants.

We remark that for simplicity of notation we restrict our considerations to \(3+1\)-dimensional Lorentzian manifolds \((M,g)\). However, all results extend in an obvious way to dimensions \(n+1\), \(n \geq 1\). Moreover, all given manifolds, functions and tensor fields are assumed to be smooth, although this is only for convenience and clearly not necessary.

\section{Part I: The theory of Gaussian beams on Lorentzian manifolds}

\subsection{Solutions of the wave equation with localised energy}
\label{underlying}

This section and the next are devoted to the proof of Theorem \ref{localised}, which summarises the state of the art knowledge concerning the construction of solutions with localised energy using the approximation by Gaussian beams.

\begin{theorem}
\label{localised}
Let \((M,g)\) be a time oriented globally hyperbolic Lorentzian manifold with time function \(t\), foliated by the level sets \(\Sigma_\tau = \{t=\tau\}\), where \(\Sigma_0\) is a Cauchy hypersurface\footnote{Note that \cite{BerSan05} shows that every globally hyperbolic Lorentzian manifold admits a smooth time function.}. Furthermore, let \(\gamma\) be a null geodesic that intersects \(\Sigma_0\) and \(N\) a timelike, future directed vector field. 

For any neighbourhood \(\N\) of \(\gamma\), for any \(T>0\) with \(\Sigma_T \cap \I(\gamma) \neq \emptyset\), and for any \(\mu >0\) there exists a solution \(v \in C^\infty(M,\C)\) of the wave equation \eqref{waveeq} with \(E^N_0(v) =1\) and a \(\tilde{u} \in C^\infty(M,\C)\) with \(\supp(\tilde{u}) \subseteq \N\) such that
\begin{equation}
\label{thmlocapp}
E^N_\tau(v - \tilde{u}) < \mu \qquad \quad \forall \; 0 \leq \tau \leq T\;,
\end{equation}
provided that we have on \(R_{[0,T]} \cap J^+(\N \cap \so)\)
\begin{equation}
\label{energyestimatecond}
\begin{aligned}
\frac{1}{|dt(\nt)|} + |g(N,\nt)| &\leq C <\infty \quad \textnormal{ and }\quad  0 < c \leq |g(N,N)| \\
|\nabla N(\nt,\nt)| + \sum_{i=1}^3|\nabla &N(\nt,e_i)| +\sum_{i,j =1}^3|\nabla N(e_i,e_j)| \leq C  < \infty
\end{aligned}
\end{equation}
where \(c\) and \(C\) are positive constants and \(\{\nt,e_1,e_2,e_3\}\) is an orthonormal frame.
\end{theorem}

\begin{center}
\def\svgwidth{6cm}
\begingroup%
  \makeatletter%
  \providecommand\color[2][]{%
    \errmessage{(Inkscape) Color is used for the text in Inkscape, but the package 'color.sty' is not loaded}%
    \renewcommand\color[2][]{}%
  }%
  \providecommand\transparent[1]{%
    \errmessage{(Inkscape) Transparency is used (non-zero) for the text in Inkscape, but the package 'transparent.sty' is not loaded}%
    \renewcommand\transparent[1]{}%
  }%
  \providecommand\rotatebox[2]{#2}%
  \ifx\svgwidth\undefined%
    \setlength{\unitlength}{529.06015625bp}%
    \ifx\svgscale\undefined%
      \relax%
    \else%
      \setlength{\unitlength}{\unitlength * \real{\svgscale}}%
    \fi%
  \else%
    \setlength{\unitlength}{\svgwidth}%
  \fi%
  \global\let\svgwidth\undefined%
  \global\let\svgscale\undefined%
  \makeatother%
  \begin{picture}(1,1.00905599)%
    \put(0,0){\includegraphics[width=\unitlength]{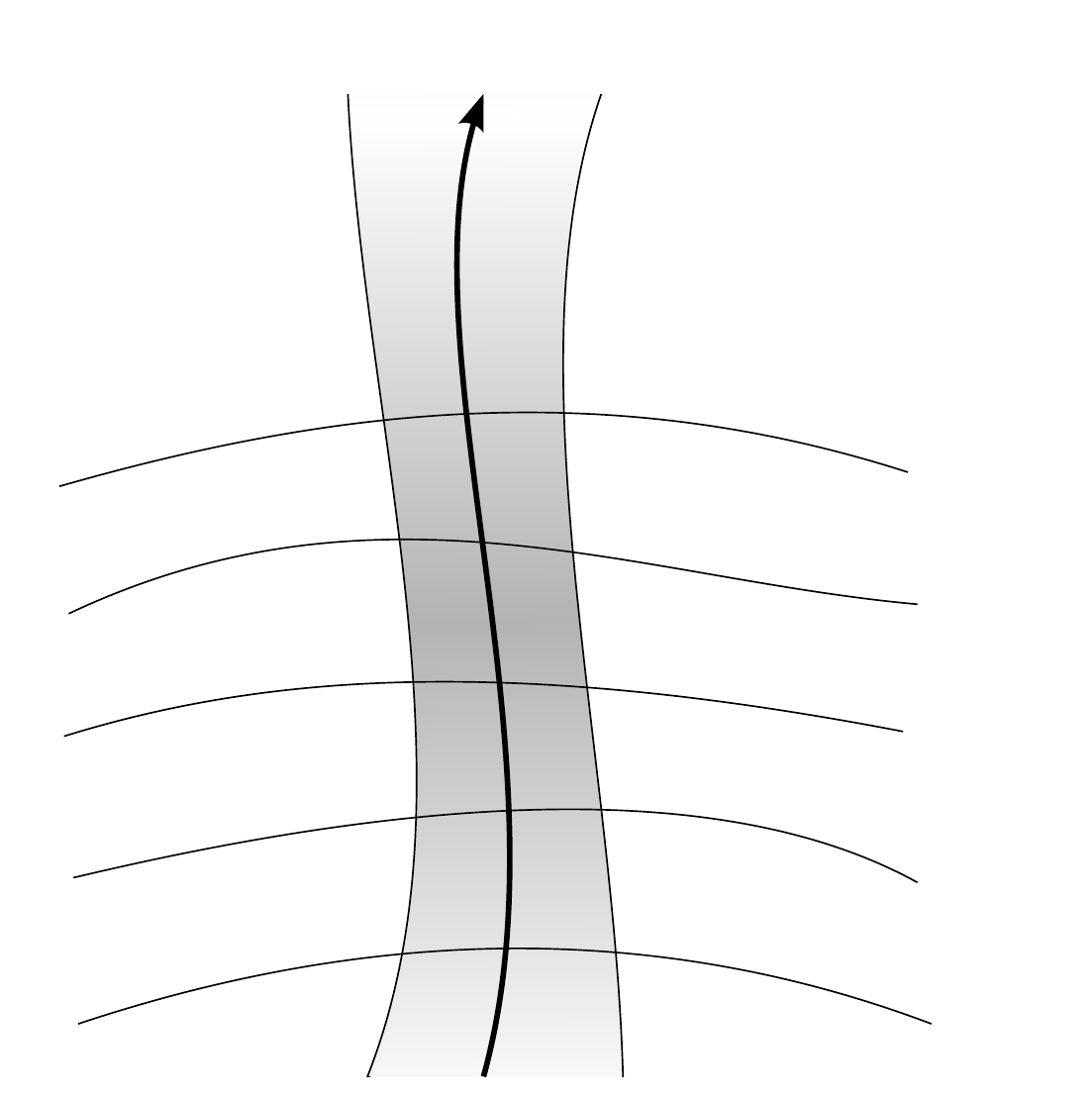}}%
    \put(0.44995267,0.56304423){\color[rgb]{0,0,0}\makebox(0,0)[lb]{\smash{\(\gamma\)}}}%
    \put(0.87320871,0.44016344){\color[rgb]{0,0,0}\makebox(0,0)[lb]{\smash{\(\Sigma_\tau\)}}}%
    \put(0.88231103,0.05559205){\color[rgb]{0,0,0}\makebox(0,0)[lb]{\smash{\(\Sigma_0\)}}}%
    \put(0.58193578,0.29680251){\color[rgb]{0,0,0}\makebox(0,0)[lb]{\smash{\(\mathcal{N}\)}}}%
    \put(0.86244107,0.56206584){\color[rgb]{0,0,0}\makebox(0,0)[lb]{\smash{\(\Sigma_T\)}}}%
  \end{picture}%
\endgroup%

\end{center}

Note that \eqref{thmlocapp} together with \(\supp(\tilde{u}) \subseteq \N\) make the statement, that the solution \(v\) hardly disperses up to time \(T\), rigorous. The energy of the solution \(v\) stays localised for finite time.

\begin{proof}
The function \(\tilde{u}\) in the theorem is the Gaussian beam, the \emph{approximate} solution to the wave equation \eqref{waveeq} which we need to construct. Recall that a Gaussian beam
\(u_\lambda \in C^\infty(M,\C)\) is of the form
\begin{equation}
\label{ansatz}
u_\lambda (x) = a_\N (x)e^{i\lambda \phi(x)} \;,
\end{equation}
where  \(\lambda >0\) is a parameter that determines how quickly the Gaussian beam oscillates, and \(a_\N\) and \(\phi\) are smooth, complex valued functions on \(M\), that do \emph{not} depend on \(\lambda\). However, \(a_\N\) depends on the neighbourhood \(\N\) of the null geodesic \(\gamma\). In Section \ref{GaussianBeams} we construct the functions \(a_\N\) and \(\phi\) in such a way that \(u_\lambda\) satisfies the following three conditions:
The \emph{first condition} is
\begin{equation}
\label{firstcond}
||\Box u_\lambda ||_{L^2(R_{[0,T]})} \leq C(T) \;, 
\end{equation}
where the constant \(C(T)\) depends on \(a_\N, \phi\) and \(T\), but not on \(\lambda\). The \emph{second condition} is
\begin{equation}
\label{secondcond}
E_0^N(u_\lambda) \to \infty \quad \textnormal{ for } \lambda \to \infty \;,
\end{equation}
where \(N\) is the timelike vector field from Theorem \ref{localised}. Finally, the \emph{third condition} is
\begin{equation}
\label{thirdcond}
u_\lambda \textnormal{ is supported in } \N.
\end{equation}
Assuming for now that we have already found functions \(a_\N\) and \(\phi\) such that the conditions \eqref{firstcond}, \eqref{secondcond} and \eqref{thirdcond} are satisfied, we finish the proof of Theorem \ref{localised}. In order to normalise the initial energy of the approximate solutions \(u_\lambda\), we define
\begin{equation*}
\tilde{u}_\lambda := \frac{u_\lambda}{\sqrt{E_0^N(u_\lambda)}} \;,
\end{equation*}
which, moreover, yields 
\begin{equation*}
||\Box\tilde{u}_\lambda||_{L^2(R_{[0,T]})} \to 0 \qquad  \textnormal{ for } \quad \lambda \to \infty \;.
\end{equation*}
This says that when the Gaussian beam becomes more and more oscillatory (i.e.\ for bigger and bigger \(\lambda\)), the closer it comes to being a proper solution to the wave equation.

We now define the actual solution \(v_\lambda\) of the wave equation - the one that is being approximated by the \(\tilde{u}_\lambda\) - to be the solution of the following initial value problem:
\begin{equation*}
\begin{aligned}
\Box v &= 0 \\
v\big|_{\Sigma_0} &= \tilde{u}_\lambda\big|_{\Sigma_0} \\
n_{\Sigma_0}v\big|_{\Sigma_0} &= n_{\Sigma_0}\tilde{u}_\lambda\big|_{\Sigma_0} \;.
\end{aligned}
\end{equation*}
Here, we make use of the fact that the Lorentzian manifold \((M,g)\) is globally hyperbolic and thus allows for a well-posed initial value problem for the wave equation.
Moreover, the condition \eqref{energyestimatecond} ensures that we have an energy estimate of the form
\begin{equation}
\label{energyestimate}
\int_{\st} J^N(u) \cdot \nt \,\vol_{\bar{g}_\tau}\leq C(T,N,\{\st\}) \Big( \int_{\so} J^N (u)\cdot \no \, \vol_{\bar{g}_0}  + ||\Box u||^2_{L^2(R_{[0,T]})} \Big) \qquad \forall \; 0 \leq \tau \leq T 
\end{equation}
at our disposal (see for example \cite{Taylor1}, chapter 2.8). Thus, we obtain
\begin{equation*}
E^N_\tau(\vl - \utl) \leq C(T,N,\st) \cdot ||\Box\utl||^2_{L^2(R_{[0,T]})} \qquad \qquad \forall \; 0\leq \tau \leq T \;,
\end{equation*}
which goes to zero for \(\lambda \to \infty\). Given now \(\mu > 0\), it suffices to choose \(\lambda_0 >0\) big enough and to set \(\tilde{u} := \tilde{u}_{\lambda_0}\) and \(v := v_{\lambda_0}\), which then finishes the proof under the assumption of the conditions \eqref{firstcond}, \eqref{secondcond} and \eqref{thirdcond}. 
\end{proof}

We end this section with a couple of remarks about Theorem \ref{localised}:

\begin{remark}
\label{rem}
As already mentioned, the condition \eqref{energyestimatecond} ensures that we have the energy estimate \eqref{energyestimate}. Note that it is automatically satisfied if the region under consideration, \(R_{[0,T]} \cap J^+(\N \cap \so)\), is relatively compact, which will be the case in many concrete applications. 

Moreover, by choosing, if necessary, \(\N\) a bit smaller, we can always arrange that \(\Sigma_T \cap \N \) is relatively compact and that \(\N \cap R_{[0,T]} \subseteq J^-(\Sigma_T \cap \N)\). Doing then the energy estimate in the relatively compact region \(J^-(\Sigma_T \cap \N) \cap J^+(\Sigma_0)\), we obtain
\begin{equation}
\label{laterimportant}
E^N_{\tau, \N \cap \st}(v - \tilde{u}) < \mu \qquad \quad \forall \; 0 \leq \tau \leq T
\end{equation}
independently of \eqref{energyestimatecond}. Of course, the information given by \eqref{laterimportant} is not interesting here, since Theorem \ref{localised} does not provide more information about \(\tilde{u}\) than its region of support. However, in Section \ref{CharacterisationEnergy} we will derive more information about the approximate solution \(\tilde{u}\) and then \eqref{laterimportant} will tell us about the temporal behaviour of the localised energy of \(v\), cf. Theorem \ref{symbiosis}.
\end{remark}

\begin{remark}
\label{rem2}
By taking the real or the imaginary part of \(\utl\) and \(\vl\) it is clear that we can choose \(\tilde{u}\) and \(v\) in Theorem \ref{localised} to be real valued.
\end{remark}

\subsection{The construction of Gaussian beams}
\label{GaussianBeams}

Before we start with the construction of Gaussian beams, let us mention that other presentations of this subject can be found for example in \cite{BabichBuldreyev} or \cite{Ral83}. The latter reference also includes the construction of Gaussian beams for more general hyperbolic PDEs.

Given now a neighbourhood \(\N\) of a null geodesic \(\gamma\), we will construct functions \(a_\N\), \(\phi \in C^\infty(M,\C)\) such that the approximate solution \(u_\lambda = a_\N \cdot e^{i\lambda \phi}\) satisfies the conditions \eqref{firstcond}, \eqref{secondcond} and \eqref{thirdcond}. This will then finish the proof of Theorem \ref{localised}.
We compute
\begin{equation}
\label{boxuGauss}
\Box u_\lambda = - \lambda^2 (d\phi \cdot d\phi)a_\N e^{i\lambda \phi} + i\lambda \Box \phi \cdot a_\N e^{i\lambda \phi} + 2i\lambda\, \grad\,\phi (a_\N)\cdot e^{i\lambda \phi} + \Box a_\N \cdot e^{i\lambda \phi} \;.
\end{equation}
If we required \(d\phi \cdot d\phi =0\) (eikonal equation) and \(2\grad \,\phi(a_\N) + \Box \phi \cdot a_\N = 0\), we would be able to satisfy \eqref{firstcond}.\footnote{We would also be able to satisfy \eqref{secondcond} and, at least up to some finite time \(T\), \eqref{thirdcond}, see Appendix \ref{geomop}.} This, however, would lead us to the naive geometric optics approximation (see Appendix \ref{geomop}), whose major drawback is that in general the solution \(\phi\) of the eikonal equation breaks down at some point along \(\gamma\) due to the formation of caustics. The method of Gaussian beams takes a slightly different approach. We only require an \emph{approximate} solution \(\phi \in C^\infty(M,\C)\) of the eikonal equation in the sense that
\begin{equation*}
d\phi \cdot d\phi \textnormal{ vanishes on }\gamma\textnormal{ to high order.}\footnote{The exact order to which we require \(d\phi \cdot d\phi\) to vanish on \(\gamma\) will be determined later.}
\end{equation*}
Moreover, we demand that
\begin{align}
&\phi\big|_\gamma \textnormal{ and } d\phi\big|_\gamma \textnormal{ are real valued} \label{real}\\
&\Ima \big(\nabla\nabla\phi\big|_\gamma\big) \textnormal{ is positive definite on a }3\textnormal{-dimensional subspace transversal to }\dot{\gamma}\;, \label{posdef}
\end{align}
where \(\Ima \big(\nabla\nabla\phi\big|_x\big)\), \(x \in M\), denotes the imaginary part of the bilinear map \(\nabla \nabla \phi\big|_x : T_xM \times T_xM \to \C\). 
Let us assume for a moment that \eqref{real} and \eqref{posdef} hold. Taking slice coordinates for \(\gamma\), i.e., a coordinate chart \((U, \varphi)\), \(\varphi : U \subseteq M \to \R^4\), such that \(\varphi\big(\I(\gamma) \cap U\big) = \{x_1=x_2=x_3=0\}\), we obtain 
\begin{equation}
\label{Imphi}
\Ima(\phi)(x) \geq c\cdot(x_1^2 + x_2^2 + x_3^2)\;,
\end{equation}
at least if we restrict \(\phi\) to a small enough neighbourhood of \(\gamma\). Note that such slice coordinates exist, since the global hyperbolicity of \((M,g)\) implies that \(\gamma\) is an embedded submanifold of \(M\). This is easily seen by appealing to the strong causality condition\footnote{Cf.\ for example \cite{ONeill}, Chapter 14, for more on the strong causality condition.}. Let us now denote the real part of \(\phi\) by \(\phi_1\) and the imaginary part by \(\phi_2\). We then have
\begin{equation*}
u_\lambda = a_\N \cdot e^{i \lambda \phi_1} \cdot e^{-\lambda \phi_2}\;.
\end{equation*}
We see that the last factor imposes the shape of a Gaussian on \(u_\lambda\), centred around \(\gamma\) -- this explains the name. Moreover, for \(\lambda\) large this Gaussian will become more and more narrow, i.e., less and less weight is given to the values of \(a_\N\) away from \(\gamma\). 

We rewrite \eqref{boxuGauss} as
\begin{equation}
\label{boxulambda}
\Box u_\lambda = -\lambda^2 \underbrace{(d\phi \cdot d\phi)} \cdot a_\N e^{i\lambda \phi_1} \cdot e^{-\lambda \phi_2} + i\lambda \big(\underbrace{2\grad \,\phi (a_\N) + \Box \phi \cdot a_\N}\big) \cdot e^{i\lambda \phi_1} \cdot e^{-\lambda \phi_2} +  \Box a_\N \cdot e^{i \lambda \phi_1} \cdot e^{-\lambda \phi_2} \;.
\end{equation}

Intuitively, if we can arrange for the underbraced terms to vanish on \(\gamma\) to some order and if we choose large \(\lambda\), then we will pick up only very small contributions. The next lemma makes this rigorous:

\begin{lemma}
\label{decaythroughGauss}
Let \(f \in C^\infty_0([0,T] \times \R^3,\C)\) vanish along \(\{x_1 = x_2 = x_3 =0\}\) to order \(S\), i.e., all partial derivatives up to and including the order \(S\) of \(f\) vanish along \(\{x_1 = x_2 = x_3 =0\}\), and let $c>0$ be a constant.
We then have

\begin{equation*}\LeftEqNo
\int_{[0,T]\times \R^3}|f(x)|^2 e^{-\lambda\cdot c(x_1^2 + x_2^2 +x_3^2)} \, dx \leq C \lambda^{-(S+1) - \frac{3}{2}} \tag{i}
\end{equation*}
and
\begin{equation*}\LeftEqNo
\int_{[0,T]\times \R^3}|f(x)| e^{-\lambda\cdot c(x_1^2 + x_2^2 +x_3^2)} \, dx \leq C \lambda^{-\frac{(S+1)}{2} - \frac{3}{2}}\;, \tag{ii}
\end{equation*}
where \(C\) depends on \(f\) (and on \(T\)).
\end{lemma}

\begin{proof}
We prove $(i)$ here, since it is used in the following. The formulation $(ii)$ of Lemma \ref{decaythroughGauss} is appealed to in the proof of Theorem \ref{main} in Section \ref{CharacterisationEnergy} - the proof is analogous. 

Introduce stretched coordinates \(y_0 := x_0\), \(y_i:=\sqrt{\lambda}x_i\) for \(i=1,2,3\). Since \(f\) vanishes along the \(x_0\) axis to order \(S\) and has compact support, we get \(|f(x)| \leq C \cdot |\underline{x}|^{S+1}\) for all \(x=(x_0,\underline{x}) \in [0,T] \times \R^3\); thus
\[|f(y_0, \frac{\underline{y}}{\sqrt{\lambda}})| \leq C \cdot \frac{|\underline{y}|^{S+1}}{\lambda^{\frac{S+1}{2}}} \;.\]
This yields
\begin{equation}
\int_{[0,T]\times \R^3} |f(x)|^2 e^{-\lambda\cdot c|\underline{x}|^2} \;dx \leq \int_{[0,T]\times \R^3} C \cdot |\underline{y}|^{2(S+1)}e^{-c|\underline{y}|^2} \,dy \cdot \lambda^{-(S+1) - \frac{3}{2}}\;.
\end{equation}
\end{proof}

We summarise the approach taken by the Gaussian beam approximation in the following 

\begin{lemma}
\label{red}
Within the setting of Theorem \ref{localised}, assume we are given \(a, \phi \in C^\infty(M, \C)\) which satisfy \eqref{real} and \eqref{posdef}. Moreover, assume
\begin{align}
d\phi \cdot d\phi \quad &\textnormal{ vanishes to second order along }\gamma \label{gaussfirstcond} \\
2\grad \,\phi(a) + \Box \phi \cdot a \quad &\textnormal{ vanishes to zeroth order along }\gamma \label{gausssecondcond} \\
a\big(\I(\gamma) \cap \so\big) \neq 0 \quad &\textnormal{ and } \quad d\phi\big(\I(\gamma) \cap \so\big) \neq 0 \label{initialenergycond}
\end{align}
Given a neighbourhood \(\N\) of \(\gamma\), we can then multiply \(a\) by a suitable bump function \(\chi_\N\) which is equal to one in a neighbourhood of \(\gamma\) and satisfies \(\supp (\chi_\N) \subseteq \N\), such that 
\begin{equation*}
u_\lambda = u_{\lambda, \N} = a_\N e^{i\lambda \phi} 
\end{equation*} 
satisfies \eqref{firstcond}, \eqref{secondcond} and \eqref{thirdcond}, where \(a_\N := a \cdot \chi_\N\).
\end{lemma}

\begin{proof}
Cover \(\gamma\) by slice coordinate patches and let \(\tilde{\chi}\) be a bump function which meets the following three requirements:
\begin{enumerate}[i)]
\item  \(\tilde{\chi}\) is equal to one in a neighbourhood of \(\gamma\) 
\item \eqref{Imphi} is satisfied for all \(x \in \supp(\tilde{\chi})\) 
\item \(R_{[0,T]} \cap \supp(\tilde{\chi})\) is relatively compact in \(M\) for all \(T > 0\) with \(\Sigma_T \cap \I(\gamma) \neq \emptyset\)\,.
\end{enumerate}
Pick now a second bump function \(\tilde{\chi}_\N\) which is again equal to one in a neighbourhood of \(\gamma\) and is supported in \(\N\). We then define \(\chi_\N := \tilde{\chi} \cdot \tilde{\chi}_\N\). Clearly, \eqref{thirdcond} is satisfied.

In order to see that \eqref{firstcond} holds, note that the conditions \eqref{real}, \eqref{posdef}, \eqref{gaussfirstcond} and \eqref{gausssecondcond} are still satisfied by the pair \((a_\N, \phi)\).  Moreover note that due to condition iii)  the integrand is supported in a compact region for each \(T >0\) with \(\Sigma_T \cap \I(\gamma) \neq \emptyset\) . Thus, the spacetime volume of this region is finite. We thus obtain \eqref{firstcond} from \eqref{boxulambda} and Lemma \ref{decaythroughGauss}.

Finally, we have 
\begin{equation*}
E_0^N(u_\lambda) \geq C \cdot (\lambda^{\frac{1}{2}} - 1)\;.
\end{equation*}
This follows since the highest order term in \(\lambda\) in \(E_0^N(u_\lambda)\) is 
\begin{equation*}
\lambda^2 \cdot \int_{\so} |a_{\N}|^2 N\phi_1 \cdot \no \phi_1 e^{-2\lambda \phi_2} \,\vgo \;,
\end{equation*}
and the same scaling argument used in the proof of Lemma \ref{decaythroughGauss} shows that the term \(e^{-2\lambda \phi_2}\) leads to a \(\lambda^{-\frac{3}{2}}\) damping - and only to a \(\lambda^{-\frac{3}{2}}\) damping due to condition \eqref{initialenergycond} (together with \eqref{gaussfirstcond} and \eqref{real}).
Thus, \eqref{secondcond} is satisfied as well and the lemma is proved.
\end{proof}



Given a null geodesic \(\gamma\) on \((M,g)\), we now construct a Gaussian beam along \(\gamma\), i.e., we construct functions \(a, \phi \in C^\infty(M, \C)\) which satisfy \eqref{real}, \eqref{posdef}, \eqref{gaussfirstcond}, \eqref{gausssecondcond} and \eqref{initialenergycond}. By Lemma \ref{red}, this then finishes the proof of Theorem \ref{localised}.

Note that the conditions \eqref{real}, \eqref{posdef}, \eqref{gaussfirstcond}, \eqref{gausssecondcond} and \eqref{initialenergycond} only depend on the values and the derivatives of \(\phi\) and \(a\) on \(\gamma\).
This allows for, instead of constructing \(\phi\) and \(a\) directly, constructing compatible first and second derivatives of \(\phi\) along \(\gamma\) and the function \(a\) along \(\gamma\) such that the above conditions are satisfied.
With the first and second derivatives of \(\phi\) being \emph{compatible} we mean the following consistency statement 
\begin{equation}
\label{compatible}
\partial_\mu \partial_\nu \phi \big(\gamma(s)\big) \dot{\gamma}^\nu(s) = \frac{d}{ds} \partial_\mu \phi\big(\gamma(s)\big)\;.
\end{equation}
From this data we can then build functions \(\phi, a \in C^\infty(M,\C)\) whose derivatives along \(\gamma\) agree with the constructed ones\footnote{This construction is known as Borel's Lemma.} - and thus, \(\phi\) and \(a\) will satisfy the above requirements. 
We start with the construction of \(\phi\). 

Let \(s\) be some affine parameter for the future directed null geodesic \(\gamma\) such that \(\gamma(0) \in \so\). We set\footnote{By slight abuse of notation we will denote the covector field along \(\gamma\) which will later be the differential of \(\phi\) already by \(d\phi\). Similarly for the second derivatives.}
\begin{equation}
\label{firstd}
d\phi(s) := \dot{\gamma}^\flat(s) \;.
\end{equation}
Moreover, we require that \(\phi(0) \in \R\). The definition \eqref{firstd} then determines \(\phi(s) \in \R\) for all \(s\); hence \eqref{real} is satisfied.
Since \(\dot{\gamma}\) is a null vector, we clearly have \(d\phi \cdot d\phi = 0\) along \(\gamma\). We now pick a slice coordinate chart that covers part of \(\gamma\) and set \(f(x) := \frac{1}{2}g^{\mu \nu}(x)\partial_\mu \phi (x) \, \partial_\nu \phi (x)\). Note that the notion of `vanishing to second order' is independent of the choice of coordinates. In order to find the conditions that the second derivative of \(\phi\) has to satisfy, we compute
\begin{equation}
\label{firstorder}
0 \overset{!}{=} \pk f \big|_\gamma = \frac{1}{2}(\pk g^{\mu \nu}) \pmu\phi \pn\phi\big|_\gamma + g^{\mu \nu} \pmu \phi \pk \pn \phi \big|_\gamma = -\dot{(\pk\phi)} + \dot{\gamma}^\nu \pn \pk \phi \;,
\end{equation}
where we have used that we have already fixed \eqref{firstd} and that \(\gamma\) is a null geodesic, thus it satisfies the equations \eqref{characteristics} of the geodesic flow on \(T^*M\).
The condition \eqref{firstorder} is exactly the compatibility condition \eqref{compatible}, thus \(f\) vanishes to first order along \(\gamma\) if we choose the second derivatives of \(\phi\) to be compatible with the first ones. Moreover, we compute
\begin{equation}
\label{secondorder}
\begin{aligned}
0 \overset{!}{=} \pk \pr f \big|_\gamma = \frac{1}{2} (\pk \pr &g^{\mu \nu}) \pmu \phi \pn \phi \big|_\gamma + (\pk g^{\mu \nu}) \pr \pmu \phi \cdot \pn \phi \big|_\gamma + (\pr g^{\mu \nu}) \pmu \phi \cdot \pk \pn \phi \big|_\gamma \\&+ g^{\mu \nu} \pr \pmu \phi \cdot \pk \pn \phi \big|_\gamma + \underbrace{g^{\mu \nu} \pmu \phi}_{=\dot{\gamma}^\nu} \pr \pk \pn \phi \big|_\gamma \;.
\end{aligned}
\end{equation}
The condition \eqref{secondorder} has actually a lot of structure. In order to see this more clearly, let \(H:T^*M \to \R\) be given by \(H(\zeta):=\frac{1}{2}g^{-1}(\zeta, \zeta)\), and having chosen a coordinate system \(\{x^\mu\}\) for part of \(M\) we denote the corresponding canonical coordinate system on part of \(T^*M\) by \(\{x^\mu,p^\nu\} = \{\xi^\alpha\}\), where \(\mu, \nu \in \{0, \ldots 3\}\) and \(\alpha \in \{0, \ldots 7\}\). We define the following matrices
\begin{alignat*}{2}
A_{\kappa \rho} (s) &:= \frac{1}{2} (\pk \pr g^{\mu \nu}) \pmu \phi \pn \phi \,\big(\gamma(s)\big) \quad &= \frac{\partial^2 H}{\partial x^\kappa \partial x^{\rho}} \big(\gamma (s)\big) \\
B_{\kappa \rho}(s) &:= \pk g^{\rho \nu} \pn \phi \,\big(\gamma(s)\big) &= \frac{\partial^2 H}{\partial x^\kappa \partial p^{\rho}} \big(\gamma (s)\big) \\
C_{\kappa \rho} (s) &:= g^{\kappa \rho}\,\big(\gamma (s)\big) &= \frac{\partial^2 H}{\partial p^\kappa \partial p^{\rho}} \big(\gamma (s)\big) \\
M_{\kappa \rho} (s) &:= \pk \pr \phi \,\big(\gamma (s)\big) \;, &
\end{alignat*}
and rewrite \eqref{secondorder} as 
\begin{equation}
\label{riccati}
0= A + B  M + M B^T + MCM + \frac{d}{ds}M \;.
\end{equation}
This quadratic ODE for the matrix  \(M\) is called a \emph{Riccati equation}. We would like to ensure that we can find a global solution that satisfies \eqref{posdef} and is compatible with the first derivatives. 

There is a well-known way to solve \eqref{riccati}, which boils down here to finding a suitable set of Jacobi fields - or using the language of Appendix \ref{BD}, a suitable Jacobi tensor. We consider the system of matrix ODEs 
\begin{equation}
\label{reduction}
\begin{aligned}
\dot{J} &=  B^TJ + CV \\
\dot{V} &= -AJ - BV \;,
\end{aligned}
\end{equation}
where \(J\) and \(V\)  are \(4\times 4\) matrices. If \(J\) is invertible then it is an easy exercise to verify that \(M:=VJ^{-1}\) solves \eqref{riccati}. We will show that we can choose initial data such that \(J\) is invertible for all time. But first let us make some remarks about \eqref{reduction}.

Although \eqref{riccati} depends on the choice of coordinates and thus has no geometric interpretation, a vector solution of the system of ODEs \eqref{reduction} is a geometric quantity:
Let us denote the Hamiltonian flow of \(H\) by \(\Psi_t : T^*M \to T^*M\), which is exactly the geodesic flow on \(T^*M\). The vector solutions of \eqref{reduction} are exactly those flow lines of the lifted flow \((\Psi_t)_* : T(T^*M) \to T(T^*M)\) that project down on the lifted geodesic \(s \mapsto \dot{\gamma}^\flat(s) \in T^*M\). In order to see this, let\footnote{In the following vectors are denoted by tilded capital letters in order to distinguish them from the untilded matrices.}
\begin{equation*}
\tilde{X} = \tilde{X}^\alpha \frac{\partial}{\partial \xi^\alpha}\Big|_{\dot{\gamma}^\flat(0)}  = \Jt^\mu \frac{\partial}{\partial x^\mu}\Big|_{\dot{\gamma}^\flat(0)} + \Vt^\nu \frac{\partial}{\partial p^\nu}\Big|_{\dot{\gamma}^\flat(0)} \in T_{\dot{\gamma}^\flat(0)}(T^*M)\;.
\end{equation*}
The pushforward via \(\Psi_t\) is then a vector field along \(\dot{\gamma}^\flat(s)\),
\begin{equation*}
(\Psi_s)_* \tilde{X} = \frac{\partial \Psi^\alpha_s}{\partial \xi^\beta}\Big|_{\dot{\gamma}^\flat(0)} \tilde{X}^\beta \, \frac{\partial}{\partial \xi^\alpha}\Big|_{\dot{\gamma}^\flat(s)} =: \Jt^\mu(s) \frac{\partial}{\partial x^\mu}\Big|_{\dot{\gamma}^\flat(s)} + \Vt^\nu(s) \frac{\partial}{\partial p^\nu}\Big|_{\dot{\gamma}^\flat(s)}  \;,
\end{equation*}
whose \(x^\rho\) component satisfies
\begin{align*}
\frac{d}{ds}\Big|_{s=s_0} (\Psi_s)_* \tilde{X} (x^\rho) &= \frac{\partial}{\partial s}\Big|_{s=s_0} \Big[ \frac{\partial (x^\rho \circ \Psi_s)}{\partial \xi^\alpha}\Big|_{\dot{\gamma}^\flat(0)} \tilde{X}^\alpha\Big] \\[2pt]
&= \frac{\partial}{\partial \xi^\alpha}\Big|_{\dot{\gamma}^\flat(0)} \frac{\partial}{\partial s}\Big|_{s=s_0} (x^\rho \circ \Psi_s) \, \tilde{X}^\alpha \\[2pt]
&= \frac{\partial}{\partial \xi^\alpha}\Big|_{\dot{\gamma}^\flat(0)} \Big(\frac{\partial H}{\partial p^\rho} \circ \Psi_{s_0}\Big) \, \tilde{X}^\alpha \\[2pt]
&= \frac{\partial^2 H}{\partial \xi^\alpha \partial p^\rho} \Big|_{\Psi_{s_0}\big(\dot{\gamma}^\flat(0)\big)} \cdot \frac{\partial \Psi_{s_0}^\alpha}{\partial \xi^{\beta}} \Big|_{\dot{\gamma}^\flat(0)} \tilde{X}^\beta \\[2pt]
&= \frac{\partial^2 H}{\partial \xi^\alpha \partial p^{\rho}} \Big|_{\dot{\gamma}^\flat(s_0)} \cdot \big[(\Psi_{s_0})_* \tilde{X}\big]^\alpha\\[2pt]
&= \frac{\partial^2 H}{\partial x^\kappa \partial p^\rho} \Big|_{\dot{\gamma}^\flat(s_0)} \Jt^\kappa(s_0) + \frac{\partial^2 H}{\partial p^\kappa \partial p^\rho} \Big|_{\dot{\gamma}^\flat(s_0)} \Vt^\kappa(s_0)\;.
\end{align*}
Here, we have used
\begin{equation*}
\frac{d}{ds}\Big|_{s=s_0} (x^\rho \circ \Psi_s) (\xi_0) = \frac{\partial H}{\partial p^\rho} \Big|_{\Psi_{s_0}(\xi_0)} \;,
\end{equation*}
see equation \eqref{characteristics}.
The computation for the \(\frac{\partial}{\partial p^\rho}\) components is analogous.
Thus, if 
\begin{equation*}
\tilde{X}(s) = \Jt^\mu(s) \frac{\partial}{\partial x^\mu}\Big|_{\dot{\gamma}^\flat(s)} + \Vt^\nu(s) \frac{\partial}{\partial p^\nu}\Big|_{\dot{\gamma}^\flat(s)}
\end{equation*}
is a vector solution of \eqref{reduction}, we see that
\begin{equation*}
\pi_*\tilde{X}(s) = \Jt^\mu(s) \frac{\partial}{\partial x^\mu}\Big|_{\gamma(s)}
\end{equation*}
is a Jacobi field along \(\gamma\), where \(\pi : T^*M \to M\) is the canonical projection map. Hence, we can construct a matrix solution \((J,V)\) with invertible \(J\) if, and only if, we can find four everywhere linearly independent Jacobi fields along \(\gamma\). This shows that if we demanded \(J\) to be real valued, we would encounter the same obstruction as in the geometric optics approach, i.e., the solution \(M\) would break down at caustics. 

Moreover, note that the Hamiltonian flow \(\Psi_t\) leaves the symplectic form \(\omega\) on \(T^*M\) invariant, which is given in \(\{x^\mu,p^\nu\}\) coordinates by
\[\begin{pmatrix}
0& \mathbbm{1} \\
-\mathbbm{1} & 0
\end{pmatrix} \;.\]
So in particular, given two vector valued solutions \(\tilde{X}(s) = \Jt^\mu(s) \frac{\partial}{\partial x^\mu}\Big|_{\dot{\gamma}^\flat(s)} + \Vt^\nu(s) \frac{\partial}{\partial p^\nu}\Big|_{\dot{\gamma}^\flat(s)} \) and \(\hat{X}(s) = \hat{J}^\mu(s) \frac{\partial}{\partial x^\mu}\Big|_{\dot{\gamma}^\flat(s)} + \hat{V}^\nu(s) \frac{\partial}{\partial p^\nu}\Big|_{\dot{\gamma}^\flat(s)}\) of \eqref{reduction}, we have that 
\begin{equation}
\label{conserved}
\omega\big(\tilde{X}(s), \hat{X}(s)\big) =     \begin{matrix} \Big(\Jt(s) & \Vt(s) \Big)  \\  &    \end{matrix} 
\begin{pmatrix}
0& \mathbbm{1} \\
-\mathbbm{1} & 0
\end{pmatrix}
\begin{pmatrix}\hat{J}(s) \\ \hat{V}(s)  \end{pmatrix}
\quad \textnormal{ is constant.}\footnote{The reader might find the following remark instructive: Although solutions of \eqref{reduction} are geometric quantities, the splitting in \(J\) and \(V\) is not a geometric one. To be more precise, while \(J\) gives rise to a vector field on \(M\), \(V\) depends on the choice of the coordinates. One can, however, turn this splitting into a geometric one, namely by making use of the splitting of \(T_{\dot{\gamma}^\flat(s)}(T^*M)\) in a vertical and a horizontal subspace, which is induced by the Levi-Civita connection. In this approach, one considers the second covariant derivative of \(f\) instead of the partial derivatives in \eqref{secondorder} and thus obtains an ODE for \(\nabla \nabla \phi\). Again, one can reduce the so obtained equation to a system of linear ODEs for a \(1\)-contravariant and \(1\)-covariant tensor \(J\) and a \(2\)-covariant tensor \(V\) along \(\gamma\) such that \(\tr\, V \otimes J^{-1}\) solves again the original equation for \(\nabla \nabla \phi\). The system of linear ODEs is now equivalent to the Jacobi equation (for a Jacobi tensor \(J\))
\begin{equation*}
D_t^2J + R(J,\dot{\gamma})\dot{\gamma} =0
\end{equation*}
with \(V=D_tJ^\flat\). The background for the reduction of the nonlinear ODE for \(\nabla \nabla \phi\) to a linear second order ODE is provided by equation \eqref{firstderivative} of Appendix \ref{BD}. Given two solutions \(J(s)\) and \(J'(s)\) of the Jacobi equation, we obtain that
\begin{equation*}
g\big(D_tJ(s), J'(s)\big) - g\big(J(s), D_tJ'(s)\big) \quad \textnormal{ is constant.}
\end{equation*}
This follows either from \eqref{conserved} or by a direct computation, making use of the Jacobi equation and the symmetry properties of the Riemannian curvature tensor. In this slightly more geometric approach, the following discussion is then analogous.}
\end{equation}

We now prescribe suitable initial data for \eqref{reduction} such that \(J(s)\) is invertible for all \(s\) and \(\sum_{\mu = 0}^3 V_{\kappa \mu}J^{-1}_{\mu \rho}(s) = M_{\kappa \rho}(s) =: \partial_\kappa \partial_\rho \phi \big(\gamma(s)\big)\) is symmetric, satisfies \eqref{riccati}, \eqref{posdef}, and \eqref{compatible}. Therefore choose \(M(0)\) such that\footnote{Note that the right hand side of ii) is determined by \eqref{firstd}.}
\vspace{3mm}
\begin{quote}
\begin{enumerate}[i)]
\item \(M(0)\) is symmetric \vspace{3mm}
\item \(M(0)_{\mu \nu} \dot{\gamma}^\nu = \dot{(\pmu\phi)}(0) \)  \vspace{3mm} 
\item \(\Ima\big(M(0)_{\mu \nu}\big)dx^\mu\big|_{\gamma(0)} \otimes dx^\nu\big|_{\gamma(0)} \) is positive definite on a three dimensional subspace of \(T_{\gamma(0)}M\) that is transversal to \(\dot{\gamma}\)\vspace{3mm}
\end{enumerate}
\end{quote}
and solve \eqref{reduction} with initial data
\begin{equation}
\label{initialdata}
\begin{pmatrix} J(0) \\ V(0) \end{pmatrix} = \begin{pmatrix} \mathbbm{1} \\ M(0) \end{pmatrix} \;.
\end{equation}
Since \eqref{reduction} is a linear ODE, we get, in the chart we are working with, a global solution   
\begin{equation*}
[0,s_{\mathrm{max}}) \ni s \mapsto
\begin{pmatrix} J(s) \\ V(s) \end{pmatrix} \;.
\end{equation*}
We show that \(J(s)\) is invertible for all \(s\) by contradiction. Thus, assume there is an \(s_0 >0\) such that \(J(s_0)\) is degenerate, i.e., there is a column vector \(0 \neq f \in \C^4\)  such that \(J(s_0) f =0\). We define
\begin{equation}
\label{Xtilde}
\tilde{X}_f(s) =  \big(J(s) f\big)^\mu \frac{\partial}{\partial x^\mu}\Big|_{\dot{\gamma}^\flat(s)} + \big(V(s) f\big)^\nu \frac{\partial}{\partial p^\nu}\Big|_{\dot{\gamma}^\flat(s)}\;,
\end{equation}
which is a vector solution to \eqref{reduction}. Using \eqref{conserved}, we compute
\begin{align*}
0&= \omega\big(\tilde{X}_f(s_0), \overline{\tilde{X}_f(s_0)}\big) = \omega\big(\tilde{X}_f(0), \overline{\tilde{X}_f(0)}\big) = [J(0)f] \cdot [\overline{V(0)} \overline{f}] - [V(0)f] \cdot [\overline{J(0)} \overline{f}] \\[1pt]  
&= -2i f  \cdot \big[\Ima(M(0))\big] \overline{f} \;,
\end{align*}
where we used that \(M(0)\) is symmetric. Since \(\Ima(M(0))\) is positive definite on a three dimensional subspace transversal to \(\dot{\gamma}\), this yields \(f^\mu \frac{\partial}{\partial x^\mu}\Big|_{\gamma(0)} = z \cdot \dot{\gamma}(0)\), for some \(0\neq z \in \C\). Without loss of generality we can assume that \(z=1\), since if necessary we consider \(z^{-1} \cdot f\) instead of \(f\). Using \eqref{initialdata} and ii) of the properties of \(M(0)\), we infer that 
\begin{equation*}
 \tilde{X}_f(0) = \dot{\gamma}^\mu(0) \frac{\partial}{\partial x^\mu}\Big|_{\dot{\gamma}^\flat(0)} + \sum_{\nu =0}^3 \dot{d\phi}_\nu(0) \frac{\partial}{\partial p^\nu}\Big|_{\dot{\gamma}^\flat(0)}\;.
\end{equation*}

On the other hand, since \(s \mapsto \dot{\gamma}^\flat(s) =: \sigma(s) \in T^*M\) is a flow line of \(\Psi_t\), we have that 
\begin{equation*}
\dot{\sigma}(s) = (\Psi_s)_*(\dot{\sigma}(0))\;,
\end{equation*}
and thus, \( s \mapsto \dot{\sigma}(s) \in T(T^*M)\) is a solution of \eqref{reduction}.
Written out in components, we have
\begin{equation}
\label{sigma}
\dot{\sigma}(s) = \dot{\gamma}^\mu(s) \frac{\partial}{\partial x^\mu}\Big|_{\dot{\gamma}^\flat(s)} + \sum_{\nu =0}^3 \dot{d\phi}_\nu(s) \frac{\partial}{\partial p^\nu}\Big|_{\dot{\gamma}^\flat(s)}\;,
\end{equation}
and thus in particular \(\dot{\sigma}(0) = \tilde{X}_f(0)\).   
Since two solutions of \eqref{reduction} that agree initially are actually equal, we infer that 
\begin{equation}
\label{theyagree}
\tilde{X}_f(s) = \dot{\sigma}(s) \textnormal{ for all } s.
\end{equation}
Projecting \eqref{theyagree} down on \(TM\) using \(\pi_*\), we obtain the contradiction
\begin{equation*}
0=  \big(J(s_0)f)^\mu \frac{\partial}{\partial x^\mu}\Big|_{\gamma(s_0)} = \pi_* \tilde{X}_f(s_0) = \pi_*\dot{\sigma}(s_0) = \dot{\gamma}(s_0) \neq 0\;.
\end{equation*} 
This shows that \(J(s)\) is invertible for all \(s \in [0,s_{\mathrm{max}})\) and hence, we obtain a global solution \(M(s)\) to the Riccati equation. 

Since \(M(0)\) is chosen to be symmetric and the Riccati equation \eqref{riccati} is invariant under transposition, it follows that \(M(s)\) is symmetric for all \(s\).

In order to see that this choice of second derivatives of \(\phi\) is compatible with our prescription of the first derivatives of \(\phi\), \eqref{firstd}, i.e., in order to show  that \eqref{compatible} holds, we choose \(f \in \C^4\) such that \(f^\mu = \dot{\gamma}^\mu(0)\). Recall that we were also led to this choice in the proof of \(J\) being invertible, and so we can deduce from \eqref{Xtilde}, \eqref{sigma} and \eqref{theyagree} that
\begin{equation}
\label{compatibility}
\begin{pmatrix} \dot{\gamma}^\mu(s) \\ \dot{d\phi}_\nu(s) \end{pmatrix} = \begin{pmatrix} J(s)_{\mu \rho} \dot{\gamma}^\rho(0)  \\ V(s)_{\nu \rho} \dot{\gamma}^\rho(0) \end{pmatrix} \;.
\end{equation} 
Using this, the compatibility \eqref{compatible} follows:
\begin{equation*}
M_{\mu \nu}(s) \dot{\gamma}^\nu (s) = \sum_{\rho, \kappa =0}^3 V_{\mu \rho}(s) J^{-1}_{\rho \kappa}(s) J_{\kappa \eta}(s) \dot{\gamma}^\eta(0) = \dot{(\partial_\mu \phi)}(s)\;.
\end{equation*}

Finally, for showing that \eqref{posdef} holds, we compute for \(f \in \C^4\) and using the notation from \eqref{Xtilde}
\begin{align*}
\omega\big(\tilde{X}_f(s), \overline{\tilde{X}_f(s)}\big) &= \big[J(s)f\big] \cdot \big[\overline{V(s)f}\big] - \big[V(s) f\big]\cdot \big[\overline{J(s)f}\big] = \big[J(s)f\big] \cdot \big[\overline{M(s) J(s) f}\big] - \big[M(s) J(s) f\big] \cdot \big[\overline{J(s)f}\big] \\
&= -2i \big[\Ima(M(s)) J(s)f\big]\cdot \big[\overline{J(s)f}\big] \;,
\end{align*}
where we made use of the symmetry of \(M(s)\).
Together with \eqref{conserved}, we obtain
\begin{equation*}
-2i \big[\Ima(M(0)) f\big]\cdot \big[\overline{f}\big] = -2i \big[\Ima(M(s)) J(s)f\big]\cdot \big[\overline{J(s)f}\big]\;.
\end{equation*}
Since \(J(s)\) is an isomorphism for all \(s\), this shows that \(\Ima(M(s))\) stays positive definite on a three dimensional subspace transversal to \(\dot{\gamma}(s)\), where we also use \eqref{compatibility}.
This finishes the construction of the second derivatives of \(\phi\) in a coordinate chart. 

Staying in this chart, the condition \eqref{gausssecondcond} is a linear first order ODE for a function \(a(s)\) along \(\gamma\), and thus prescribing initial data \(a(0) \neq 0\), the existence of a global solution \(a(s)\) with respect to this chart is guaranteed. Writing down the formal Taylor series up to order two for \(\phi\) and up to order zero for \(a\) in the slice coordinates (special case of Borel's Lemma), we construct two functions \(a, \phi \in C^\infty(U, \C)\) that satisfy \eqref{real}, \eqref{posdef}, \eqref{gaussfirstcond}, \eqref{gausssecondcond} and \eqref{initialenergycond}, where \(U\) is the domain of the slice coordinate chart.
\newline
\newline
Let \(\gamma : [0,S) \to M\) be the affine parametrisation of \(\gamma\), where \(0< S \leq \infty\). Let us for the following presentation assume that \(S=\infty\) -- the case \(S < \infty\) is even simpler.
We cover \(\I(\gamma)\) by slice coordinate charts \((U_k, \varphi_k)\), \(k \in \mathbb{N}\), such that there is a partition of \([0,\infty)\) by intervals \([s_{k-1}, s_k]\) with \(s_0 =0\) and \(s_{k-1} < s_k\) that satisfies \(\gamma\big([s_{k-1},s_k]\big) \subseteq U_k\).  We then construct functions \(a_k, \phi_k \in C^\infty(U_k, \C)\) that satisfy \eqref{real}, \eqref{posdef}, \eqref{gaussfirstcond}, \eqref{gausssecondcond} (and \eqref{initialenergycond} for \(k=1\)) as follows: The case \(k=1\) was presented above. For \(k>1\) we repeat the construction from above with some slight modifications: If \(M_{k-1}(s_k)\) denotes the solution of \eqref{riccati} in the chart \(U_{k-1}\) at time \(s_{k-1}\), we now express \(M_{k-1}(s_k)\) in the \(\varphi_k\) coordinates\footnote{The transformation is of course given by the rule by which second coordinate derivatives of scalar functions transform.} and solve \eqref{riccati} in both time directions. We proceed analogously for \(a\).

Extending \(\{U_k\}_{k\in \mathbb{N}}\)  to an open cover of \(M\) by \(U_0 \subseteq M\) in such a way that \(U_0 \cap \I(\gamma) = \emptyset\) and taking a partition of unity \(\{\eta_k\}_{k \in \mathbb{N}_0}\) subordinate to this open cover, we glue all the local functions  \(\phi_k\) and  \(a_k\) together to obtain  \(\phi := \sum_{k=1}^\infty \phi_k \eta_k\) and \(a := \sum_{k=1}^\infty a_k \eta_k\), which are in \(C^\infty(M, \C)\) and satisfy \eqref{real}, \eqref{posdef}, \eqref{gaussfirstcond}, \eqref{gausssecondcond} and \eqref{initialenergycond}.
This finally completes the proof of Theorem \ref{localised}. 
\newline
\newline
For future reference, we make the following
\begin{definition}
\label{defgauss}
Let \((M,g)\) be a time oriented globally hyperbolic Lorentzian manifold with time function \(t\), foliated by the level sets \(\Sigma_\tau = \{t=\tau\}\). Furthermore, let \(\gamma : [0,S) \to M\) be an affinely parametrised future directed null geodesic with \(\gamma(0) \in \Sigma_0\), where \(0< S \leq \infty\), and let \(N\) be a timelike, future directed vector field. 

Given functions \(a, \phi \in C^\infty(M,\C)\) that satisfy \eqref{real}, \eqref{posdef}, \eqref{gaussfirstcond}, \eqref{gausssecondcond}, \(a\big(\I(\gamma) \cap \so\big) \neq 0\) and \eqref{firstd}, we call
the function 
\begin{equation*}
u_{\lambda, \N} = a_\N e^{i\lambda \phi}
\end{equation*}
a \emph{Gaussian beam along \(\gamma\) with structure functions \(a\) and \(\phi\) and with parameters \(\lambda\) and \(\N\)}. Here, \(a_\N = a \cdot \chi_\N = a \cdot \tilde{\chi} \cdot \tilde{\chi}_\N\) with \(\tilde{\chi}\) and \(\tilde{\chi}_\N\) as in the proof of Lemma \ref{red}. Moreover, we call the function
\begin{equation*}
\tilde{u}_{\lambda, \N} = \frac{u_{\lambda, \N}}{\sqrt{E_0^N(u_{\lambda, \N})}} \cdot \sqrt{E}
\end{equation*}
a \emph{Gaussian beam along \(\gamma\) with structure functions \(a\) and \(\phi\), with parameters \(\lambda\) and \(\N\), and with initial \(N\)-energy \(E\)}, where \(E\) is a strictly positive real number. Let us emphasise, that when we say `a Gaussian beam along \(\gamma\)', \(\gamma\) encodes here not only the image of \(\gamma\), but also the affine parametrisation.
\end{definition}
We end this section with the remark that for the sole \emph{construction} of the Gaussian beams the assumption of the global hyperbolicity of \((M,g)\) can be replaced by the assumption that the null geodesic \(\gamma: \R \supseteq I  \to M\) is a smooth embedding, i.e., in particular \(\gamma(I)\) being an embedded submanifold. Moreover, note that if \(\gamma : \R \supseteq I \to M\) is a smooth injective immersion and if \([a,b] \subseteq I\) with \(a,b \in \R\), then \(\gamma|_{(a,b)} : (a,b) \to M\) is a smooth embedding. It thus follows that the above construction is always possible for null geodesics with no self-intersections on general Lorentzian manifolds  - at least up to some finite affine time in the domain of \(\gamma\).

\subsection{Geometric characterisation of the energy of Gaussian beams}
\label{CharacterisationEnergy}

In this section we characterise the energy of a Gaussian beam in terms of the energy of the underlying null geodesic. The following theorem is the main result of Part I of this paper:

\begin{theorem}
\label{main}
Let \((M,g)\) be a time oriented globally hyperbolic Lorentzian manifold with time function \(t\), foliated by the level sets \(\Sigma_\tau = \{t=\tau\}\). Moreover, let \(N\) be a timelike future directed vector field and \(\gamma : [0, S) \to M\) an affinely parametrised future directed null geodesic with \(\gamma(0) \in \so\), where \(0 < S \leq \infty\). 

For any \(T>0\) with \(\I(\gamma) \cap \Sigma_T \neq \emptyset\) and for any \(\mu >0\) there exists a \(\lambda_0 >0\) such that any Gaussian beam \(\tilde{u}_{\lambda, \N}\) along \(\gamma\) with structure functions \(a\) and \(\phi\), with parameters \(\lambda \geq \lambda_0\) and \(\N\), and with initial \(N\)-energy
equal to \(-g(N,\dot{\gamma})\big|_{\gamma(0)}\) satisfies
\begin{equation}
\label{maineq}
\Big|\,E^N_\tau(\tilde{u}_{\lambda, \N}) - \big[-g(N,\dot{\gamma})\big|_{\I(\gamma)\cap\st}\big] \Big| < \mu \qquad \forall \; 0\leq \tau \leq T \;.
\end{equation}
\end{theorem}

Before we give the proof, we make a couple of remarks:
\begin{enumerate}[i)]
\item
The only information about a Gaussian beam we made use of in Theorem \ref{localised}, apart from it being an approximate solution, was that it is supported in a given neighbourhood \(\N\) of the null geodesic \(\gamma\). This then yielded, together with \eqref{thmlocapp}, an estimate on the energy outside of the neighbourhood \(\N\) of the actual solution to the wave equation, i.e., we could construct solutions to the wave equation with \emph{localised} energy. However, Theorem \ref{localised} does not make any statement about the \emph{temporal behaviour} of this localised energy. The above theorem fills this gap by investigating the temporal behaviour of the energy of the approximate solution, i.e., of the Gaussian beam. Together with \eqref{thmlocapp} (or even with \eqref{laterimportant}!) this then gives an estimate on the temporal behaviour of the localised energy of the actual solution to the wave equation.
\item
Note that if \(N\) is a timelike Killing vector field, the \(N\)-energy  \(-g(N, \dot{\gamma})\) of the null geodesic \(\gamma\) is constant, and thus, so is approximately the \(N\)-energy of the Gaussian beam.
\item
By our Definition \ref{defgauss} a Gaussian beam is a complex valued function. However, by taking the real or the imaginary part, one can also define a real valued Gaussian beam. The result of Theorem \ref{main} also holds true in this case, and can be proved using exactly the same technique - only the computations become a bit longer, since we have to deal with more terms.
\item 
Although we have stated the above theorem again using the general assumptions needed for Theorem \ref{localised}, we actually do not need more assumptions than we need for the construction of a Gaussian beam, cf.\ the final remark of the previous section.
\end{enumerate}

\begin{proof}
Recall from Definition \ref{defgauss} that a Gaussian beam \(\tilde{u}_{\lambda, \N}\) along \(\gamma\) with structure functions \(a\) and \(\phi\), with parameters \(\N\) and \(\lambda\), and with initial \(N\)-energy equal to \(-g(N,\dot{\gamma})\big|_{\gamma(0)}\) is a function
\begin{equation*}
\tilde{u}_{\lambda, \N} = \frac{u_{\lambda, \N}}{\sqrt{E_0^N(u_{\lambda, \N})}} \cdot \sqrt{-g(N,\dot{\gamma})\big|_{\gamma(0)}} 
= \frac{a_\N e^{i\lambda \phi}}{\sqrt{E_0^N(u_{\lambda, \N})}} \cdot \sqrt{-g(N,\dot{\gamma})\big|_{\gamma(0)}}\;,
\end{equation*}
where the functions \(a_\N\) and \(\phi\) satisfy  \eqref{real}, \eqref{posdef}, \eqref{gaussfirstcond}, \eqref{gausssecondcond}, \eqref{initialenergycond}, \eqref{firstd}, \(\supp(a_\N) \subseteq \N\), \(\N \cap R_{[0,T]}\) is relatively compact for all \(T>0\) with \(\Sigma_T \cap \I(\gamma) \neq \emptyset\),  and for a cover of \(\gamma\) with slice coordinate patches \eqref{Imphi} holds for all \(x \in \supp(a_\N)\). 

We will show
\begin{equation}
\label{Aim}
E_\tau^N(\tilde{u}_{\lambda, \N}) = \frac{E_\tau^N(u_{\lambda, \N})}{E_0^N(u_{\lambda, \N})} \cdot \Big[-g(N,\dot{\gamma})\big|_{\gamma(0)}\Big] = -g(N,\dot{\gamma})\big|_{\I(\gamma) \cap \st} + o(\lambda) \;,
\end{equation}
where $o (\lambda)$ goes to zero uniformly in $0 \leq \tau \leq T$ for $\lambda \to \infty$. This would then prove the theorem.

In the following we compute the leading order term of $E_\tau^N(u_{\lambda, \N})$ in $\lambda$:
\begin{align*}
J^N(u_{\lambda, \N}) \cdot \nt &= \Real(N\uln \cdot \overline{\nt \uln}) - \frac{1}{2}g(N,\nt) \, d\uln \cdot \overline{d \uln} \\[6pt]
&= \lambda^2 |a_\N|^2 N\phi_1 \cdot \nt \phi_1 \cdot \elp + \lambda^2 |a_\N|^2 N\phi_2 \cdot \nt \phi_2 \cdot \elp + \bigO(\lambda) \cdot \elp \\[6pt]
&\qquad -\frac{1}{2}g(N,\nt)\Big[ \lambda^2 |a_\N|^2 \, (d\phi_1 \cdot d\phi_1) \, \elp + \lambda^2 |a_\N|^2 \,(d\phi_2 \cdot d\phi_2)\, \elp + \bigO(\lambda) \cdot \elp\Big] \;.
\end{align*}
Note that \(d\phi_2\big|_{\gamma(\tau)} =0\), so these terms are of lower order after integration over \(\st\). The same holds for the \(d\phi_1 \cdot d\phi_1\) term. Thus, we get
\begin{equation}
\label{EnHighTerm}
E^N_\tau(\uln) = \underbrace{\lambda^2 \, \int_{\st} |a_\N|^2 \,N\phi_1 \cdot \nt\phi_1 \, \elp \;\vgt}_{=\bigO(\lambda^{\frac{1}{2}})} \; + \underbrace{\textnormal{ lower order terms }}_{=\bigO(1)} \;.
\end{equation}

The main part of the proof is an approximate conservation law. Recall that \(a_\N\) and \(\phi\) satisfy \eqref{gaussfirstcond} and \eqref{gausssecondcond}. These equations yield
\begin{equation}
\label{firstcons}
\begin{aligned}
\grad\, \phi \,\big(|a_\N|^2\big) &= \grad\, \phi \, (a_\N) \cdot \overline{a_\N} + a_\N \cdot \grad \, \phi\,(\overline{a_\N}) \\&= -\frac{1}{2}\big( \Box \phi \cdot a_\N \overline{a_\N} +a_\N \, \overline{\Box \phi} \cdot \overline{a_\N} \big) = -\Real (\Box \phi) |a_\N|^2 \quad \textnormal{ along } \gamma
\end{aligned}
\end{equation}
and
\begin{equation*}
d\phi \cdot d\phi = (d\phi_1 + i d\phi_2) \cdot (d\phi_1 + i d\phi_2) = d\phi_1 \cdot d\phi_1 - d\phi_2 \cdot d\phi_2 + 2i \, d\phi_1 \cdot d\phi_2
\end{equation*}
vanishes to second order along \(\gamma\), thus in particular
\begin{equation}
\label{secondcons}
d\phi_1 \cdot d\phi_2 = \grad \phi_1 \, (\phi_2) \quad \textnormal{vanishes along }\gamma \textnormal{ to second order.}
\end{equation}
Lemma \ref{decaythroughGauss} $(ii)$, together with \eqref{firstcons} and \eqref{secondcons}, show that the current
\begin{equation*}
X_{\lambda, \N} =\lambda^2 \cdot|a_\N|^2 \elp \,\grad\,\phi_1
\end{equation*}
is approximately conserved in the sense that
\begin{equation*}
\begin{aligned}
&\int_{\Rot} \mathrm{div} \,X_{\lambda, \N} \; \vg \\&\quad= \lambda^2 \cdot \int_{\Rot} \Big(\underbrace{\big[\grad \phi_1 \, (|a_\N|^2) + \Box \phi_1 \cdot |a_\N|^2 \big] \elp}_{=\lambda^{-\frac{1}{2} }\cdot \lambda^{-\frac{3}{2}} = \lambda^{-2} \textnormal{ after integration}}  - \underbrace{2 \lambda \grad\phi_1 \,(\phi_2) \cdot |a_\N|^2 \elp}_{= \lambda \cdot \lambda^{-\frac{3}{2}}\cdot \lambda^{-\frac{3}{2}} = \lambda^{-2} \textnormal{ after int.}} \Big) \;\vg = \bigO(1) \;,
\end{aligned}
\end{equation*}
but
\begin{equation*}
\int_{\st} X_{\lambda,\N} \cdot \nt \; \vgt = \lambda^2 \cdot \int_{\st} |a_\N|^2 \nt\phi_1 \,\elp \;\vgt = \bigO(\lambda^{\frac{1}{2}})\;.
\end{equation*}
In particular, we obtain\footnote{\label{foot}In the geometric optics approximation we have indeed a proper conservation law, which is interpreted in the physics literature as conservation of photon number, cf.\ for example \cite{MTW}, Chapter 22.5.}
\begin{equation}
\label{appcons}
\begin{aligned}
\Big| \, \lambda^2 \cdot \int_{\st} &|a_\N|^2 \nt\phi_1 \,\elp \;\vgt - \lambda^2 \cdot \int_{\so} |a_\N|^2 \no\phi_1 \,\elp \;\vgo \,\Big| \\&= \Big| \, \int_{\Rot} \mathrm{div} \,X_{\lambda,\N} \; \vg \,\Big| = \bigO(1)\;.
\end{aligned}
\end{equation}

We also observe that by Lemma \ref{decaythroughGauss} $(ii)$ we have
\begin{equation}
\label{EvOnGeo}
\lambda^2 \cdot \int_{\st} |a_\N|^2 \big(N\phi_1 - N\phi_1 \big|_{\I(\gamma)\cap \st}\big) \cdot \nt\phi_1 \,\elp \;\vgt = \bigO(1)\;.
\end{equation}
It thus follows from \eqref{EnHighTerm}, \eqref{appcons}, and \eqref{EvOnGeo} that
\begin{equation*}
\begin{split}
E^N_\tau(\uln) &= \lambda^2 \, \int_{\st} |a_\N|^2 \,N\phi_1 \cdot \nt\phi_1 \, \elp \;\vgt + \bigO(1) \\[3pt]
&= \lambda^2 \cdot N\phi_1\big|_{\I(\gamma)\cap \st} \, \int_{\st} |a_\N|^2  \nt\phi_1 \, \elp \;\vgt + \bigO(1) \\[3pt]
&= \lambda^2 \cdot N\phi_1\big|_{\I(\gamma)\cap \st} \, \int_{\so} |a_\N|^2  \nt\phi_1 \, \elp \;\vgo + \bigO(1) \\[3pt]
&= \frac{N\phi_1\big|_{\I(\gamma)\cap \st}}{N\phi_1\big|_{\I(\gamma)\cap \so}} \cdot E^N_0(\uln) + \bigO(1) \\[3pt]
&=  \frac{g(N, \dot{\gamma})\big|_{\I(\gamma)\cap \st}}{g(N,\dot{\gamma})\big|_{\I(\gamma)\cap \so}} \cdot  E^N_0(\uln) + \bigO(1) \;.
\end{split}
\end{equation*}
Substituting this into the expression for $E^N_{\tau}(\tilde{u}_{\lambda,\N})$, i.e., the first equation in \eqref{Aim}, we obtain the second equation of \eqref{Aim}. This finishes the proof of Theorem \ref{main}.
\end{proof}

\subsection{Some general theorems about the Gaussian beam limit of the wave equation}
\label{generalthms}

We can now make a much more detailed statement about the behaviour of solutions \(v\) of the wave equation in the Gaussian beam limit than Theorem \ref{localised} does:

\begin{theorem}
\label{symbiosis}
Let \((M,g)\) be a time oriented globally hyperbolic Lorentzian manifold with time function \(t\), foliated by the level sets \(\Sigma_\tau = \{t=\tau\}\), where \(\Sigma_0\) is a Cauchy hypersurface. Furthermore, let \(\gamma : [0, S) \to M\) be an affinely parametrised future directed null geodesic with \(\gamma(0) \in \so\), where \(0 < S \leq \infty\).  Finally, let \(N\) be a timelike, future directed vector field. 

For any neighbourhood \(\N\) of \(\gamma\), for any \(T>0\) with \(\Sigma_T \cap \I(\gamma) \neq \emptyset\), and for any \(\mu >0\), there exists a solution \(v \in C^\infty(M,\C)\) of the wave equation \eqref{waveeq} with \(E^N_0(v) = - g(N, \dot{\gamma})\big|_{\gamma(0)}\) such that
\begin{equation}
\label{energyinside}
\Big| E^N_{\tau, \N \cap \Sigma_\tau}(v) - \big[ - g(N,\dot{\gamma})\big|_{\I{\gamma} \cap \st} \big]\Big| < \mu \qquad \quad \forall \; 0 \leq \tau \leq T\;
\end{equation}
and\footnote{We denote the complement of $\N$ in $M$ with $\N^c$.}
\begin{equation}
\label{energyoutside}
E^N_{\tau, \N^c \cap \st} (v) < \mu \qquad \quad \forall \; 0 \leq \tau \leq T\;,
\end{equation}
provided that we have on \(R_{[0,T]} \cap J^+(\N \cap \so)\)
\begin{equation}
\label{energyestimatecond2}
\begin{aligned}
\frac{1}{|dt(\nt)|} + |g(N,\nt)| &\leq C <\infty \quad \textnormal{ and }\quad  0 < c \leq |g(N,N)| \\
|\nabla N(\nt,\nt)| + \sum_{i=1}^3|\nabla &N(\nt,e_i)| +\sum_{i,j =1}^3|\nabla N(e_i,e_j)| \leq C  < \infty
\end{aligned}
\end{equation}
where \(c\) and \(C\) are positive constants and \(\{\nt,e_1,e_2,e_3\}\) is an orthonormal frame.

Moreover, by choosing \(\N\), if necessary, a bit smaller, \eqref{energyinside} holds independently of \eqref{energyestimatecond2}. 
\end{theorem}

\begin{proof}
This follows easily from Theorem \ref{localised}, Theorem \ref{main}, the second part of Remark \ref{rem} and the triangle inequality for the square root of the \(N\)-energy. 
\end{proof}  
Let us again remark that the solution \(v\) of the wave equation in Theorem \ref{symbiosis} can also be chosen to be real valued.

The next theorem is a direct consequence of Theorem \ref{symbiosis} and can be used in particular, but not only for, proving upper bounds on the rate of the energy decay of waves on globally hyperbolic Lorentzian manifolds if we only allow the initial energy on the right hand side of the decay statement.

\begin{theorem}
\label{LED}
Let \((M,g)\) be a time oriented globally hyperbolic Lorentzian manifold with time function \(t\), foliated by the level sets \(\Sigma_\tau = \{t=\tau\}\), where \(\Sigma_0\) is a Cauchy hypersurface. Furthermore, let \(\mathcal{T}\) be an open subset of \(M\). Assume there is an affinely parametrised future directed null geodesic \(\gamma : [0, S) \to M\) with \(\gamma(0) \in \Sigma_0\), where \(0< S \leq \infty\), that is completely contained in \(\mathcal{T}\).
Let \[\tau^* := \sup\big{\{}\hat{\tau} \in [0,\infty) \,\big|\,\I(\gamma) \cap \st \neq \emptyset \textnormal{ for all } 0 \leq \tau < \hat{\tau}\big{\}}\;.\] Moreover, let \(N\) be a timelike, future directed vector field and \(P : [0,\tau^*) \to (0,\infty)\) a function\footnote{There is no assumption on the regularity of the function $P$.}.

If there is no constant \(C>0\) such that 
\begin{equation*}
 -g(N,\dot{\gamma})\big|_{\I(\gamma) \cap \st} \leq  P(\tau) C
\end{equation*}
holds for all \(0 \leq \tau < \tau^*\), then there exists no constant \(C>0\) such that 
\begin{equation}´
\label{stat}
E^N_{\tau, \mathcal{T} \cap \Sigma_\tau}(u)  \leq P(\tau) C E^N_0(u)
\end{equation}
holds for all solutions \(u\) of the wave equation \eqref{waveeq} for \(0 \leq \tau < \tau^*\). 
\end{theorem}

\begin{proof}
Assume the contrary, i.e., that there exists a constant \(C_0 >0\) such that \eqref{stat} holds. There is then a \(0 \leq \tau_0 < \tau^*\) with \( -g(N,\dot{\gamma})\big|_{\I(\gamma) \cap \Sigma_{\tau_0}} > -g(N,\dot{\gamma})\big|_{\I(\gamma) \cap \so} C_0 P(\tau_0) \). Choosing now \(\mu>0\) small enough and a neighbourhood \(\N \subseteq \mathcal{T}\) of \(\gamma\) small enough such that \eqref{energyinside} of Theorem \ref{symbiosis} applies without reference to \eqref{energyestimatecond2}, we obtain a contradiction.
\end{proof}

A very robust method for proving decay of solutions of the wave equation was given in \cite{DafRod09b} by Dafermos and Rodnianski (but also see \cite{MeTaTo12}). This method requires in particular an integrated local energy decay (ILED) statement (possibly with loss of derivative), i.e., a statement of the form \eqref{ILEDstat}. The next theorem gives a sufficient criterion for an ILED statement having to lose regularity.

\begin{theorem}
\label{ILED}
Let \((M,g)\) be a time oriented globally hyperbolic Lorentzian manifold with time function \(t\), foliated by the level sets \(\Sigma_\tau = \{t=\tau\}\), where \(\Sigma_0\) is a Cauchy hypersurface. Furthermore, let \(\mathcal{T}\) be an open subset of \(M\). Assume there is an affinely parametrised future directed null geodesic \(\gamma : [0, S) \to M\) with \(\gamma(0) \in \Sigma_0\), where \(0< S \leq \infty\), that is completely contained in \(\mathcal{T}\). Let \(N\) be a timelike, future directed vector field and set \[\tau^* := \sup\big{\{}\hat{\tau} \in [0,\infty) \,\big|\,\I(\gamma) \cap \st \neq \emptyset \textnormal{ for all } 0 \leq \tau < \hat{\tau}\big{\}}\;.\]

If
\begin{equation*}
\int_0^{\tau^*} -g(N,\dot{\gamma})\big|_{\I(\gamma) \cap \st} \;d\tau = \infty \;,
\end{equation*}
where \(\dot{\gamma}\) is with respect to some affine parametrization, then there exists no constant \(C>0\) such that
\begin{equation}
\label{ILEDstat}
\int_0^{\tau^*} \,\int_{\st \cap \mathcal{T}} J^N(u) \cdot \nt \; \vgt \; d\tau \leq C E^N_0(u)
\end{equation}
holds for all solutions \(u\) of the wave equation \eqref{waveeq}. 
\end{theorem}

The proof of this theorem goes along the same lines as the one of Theorem \ref{LED}. The reader might have noticed that whether an ILED statement of the form \eqref{ILEDstat} exists or not depends heavily on the choice of the time function. On the other hand, it also depends heavily on the choice of the time function whether an ILED statement is helpful or not. So, for instance, we only have an estimate of the form 
\begin{equation*}
\int_{\mathcal{T} \cap R_{[0,\tau^*]}} J^N(u) \cdot \nt \;\mathrm{vol}_g \leq C \cdot \int_0^{\tau^*} \,\int_{\st \cap \mathcal{T}} J^N(u) \cdot \nt \; \vgt \; d\tau\;,
\end{equation*}
where \(C>0\), if the time function \(t\) is chosen such that \(\frac{1}{|dt(\nt)|} \leq C\) is satisfied for all \(0 \leq \tau \leq \tau^*\). Such an estimate, together with an ILED statement, is very convenient whenever one needs to control spacetime integrals that are quadratic in the first derivatives of the field.

\section{Part II: Applications to  black hole spacetimes}
\label{applications}

In the following we give a selection of applications of Theorems \ref{symbiosis}, \ref{LED} and \ref{ILED}. A rich variety of behaviours of the energy is provided by black hole spacetimes arising in general relativity\footnote{Another physically interesting application would be for example to the study of waves in time dependent inhomogeneous media.}. Although we will briefly introduce the Lorentzian manifolds that represent these black hole spacetimes, the reader completely unfamiliar with those is referred to \cite{HawkEllis} for a more detailed discussion, including the concept of a so called Penrose diagram and an introduction to general relativity. 

We first restrict our considerations to the \(2\)-parameter family of Reissner-Nordstr\"om black holes, which are exact solutions to the Einstein-Maxwell equations.  The spherical symmetry of these spacetimes (and the accompanying simplicity of the metric) allows for an easy presentation without hiding any crucial details.  In Section \ref{KerrSec} we then discuss the Kerr family and show that analogous results hold. 

\subsection{Applications to Schwarzschild and Reissner-Nordstr\"om black holes}
\label{RNSec}
The \(2\)-parameter family of Reissner-Nordstr\"om spacetimes is given by 
\begin{equation}
\label{RNfamily}
g= -(1 - \frac{2m}{r} + \frac{e^2}{r^2})\,dt^2 + \Big(1 - \frac{2m}{r} + \frac{e^2}{r^2}\Big)^{-1}\,dr^2 + r^2\,d\theta^2 + r^2 \sin^2\theta\,d\varphi^2 \;,
\end{equation}
initially defined on the manifold \(M:=\R \times (m+\sqrt{m^2 - e^2}, \infty) \times \mathbb{S}^2\), for which \((t,r,\theta,\varphi)\) are the standard coordinates. We restrict the real parameters \(m\) and \(e\), which model the mass and the charge of the black hole, respectively, to the range \(0 \leq e \leq m\), \(m\neq 0\). 

For \(e=0\) we obtain the \(1\)-parameter Schwarzschild subfamily which solves the vacuum Einstein equations. The manifold \(M\) and the metric \eqref{RNfamily} can be  analytically extended (such that they still solve the Einstein equations). The so called Penrose diagram of the maximal analytic extension of the Schwarzschild family is given below:

\begin{center}
\def\svgwidth{9cm}
\begingroup%
  \makeatletter%
  \providecommand\color[2][]{%
    \errmessage{(Inkscape) Color is used for the text in Inkscape, but the package 'color.sty' is not loaded}%
    \renewcommand\color[2][]{}%
  }%
  \providecommand\transparent[1]{%
    \errmessage{(Inkscape) Transparency is used (non-zero) for the text in Inkscape, but the package 'transparent.sty' is not loaded}%
    \renewcommand\transparent[1]{}%
  }%
  \providecommand\rotatebox[2]{#2}%
  \ifx\svgwidth\undefined%
    \setlength{\unitlength}{527.17768555bp}%
    \ifx\svgscale\undefined%
      \relax%
    \else%
      \setlength{\unitlength}{\unitlength * \real{\svgscale}}%
    \fi%
  \else%
    \setlength{\unitlength}{\svgwidth}%
  \fi%
  \global\let\svgwidth\undefined%
  \global\let\svgscale\undefined%
  \makeatother%
  \begin{picture}(1,0.56132311)%
    \put(0,0){\includegraphics[width=\unitlength]{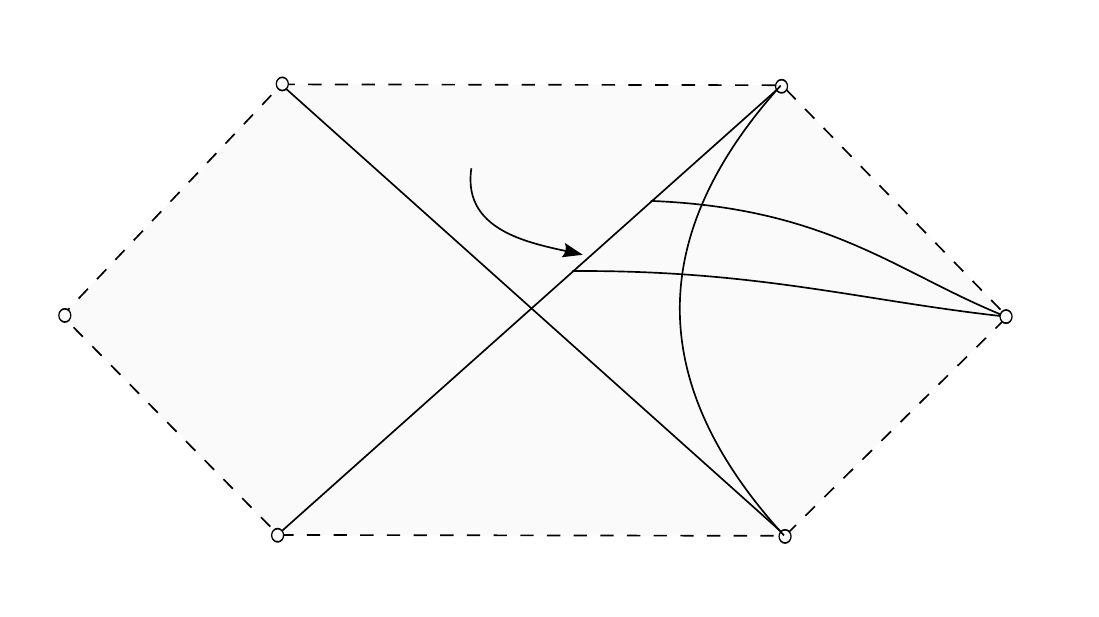}}%
    \put(0.7617677,0.29962187){\color[rgb]{0,0,0}\makebox(0,0)[lb]{\smash{\(\Sigma_0\)}}}%
    \put(0.700022,0.37094869){\color[rgb]{0,0,0}\makebox(0,0)[lb]{\smash{\(\Sigma_\tau\)}}}%
    \put(0.54009288,0.34746962){\color[rgb]{0,0,0}\rotatebox{41.61468695}{\makebox(0,0)[lb]{\smash{\(r=2m\)}}}}%
    \put(0.45864826,0.49682872){\color[rgb]{0,0,0}\makebox(0,0)[lb]{\smash{\(r=0\)}}}%
    \put(0.71918447,0.49337544){\color[rgb]{0,0,0}\makebox(0,0)[lb]{\smash{\emph{i}\(^+\)}}}%
    \put(0.92934806,0.26661986){\color[rgb]{0,0,0}\makebox(0,0)[lb]{\smash{\emph{i}\(^0\)}}}%
    \put(0.6995293,0.03041353){\color[rgb]{0,0,0}\makebox(0,0)[lb]{\smash{\emph{i}\(^-\)}}}%
    \put(0.81616714,0.40043454){\color[rgb]{0,0,0}\rotatebox{-45.39591878}{\makebox(0,0)[lb]{\smash{\(\mathcal{I}^+\)}}}}%
    \put(0.83347146,0.12931763){\color[rgb]{0,0,0}\rotatebox{44.87806989}{\makebox(0,0)[lb]{\smash{\(\mathcal{I}^-\)}}}}%
    \put(0.46467206,0.03148841){\color[rgb]{0,0,0}\makebox(0,0)[lb]{\smash{\(r=0\)}}}%
    \put(0.23727864,0.50090521){\color[rgb]{0,0,0}\makebox(0,0)[lb]{\smash{\emph{i}\(^+\)}}}%
    \put(0.01222104,0.26963178){\color[rgb]{0,0,0}\makebox(0,0)[lb]{\smash{\emph{i}\(^0\)}}}%
    \put(0.23418897,0.03041358){\color[rgb]{0,0,0}\makebox(0,0)[lb]{\smash{\emph{i}\(^-\)}}}%
    \put(0.30425243,0.38556626){\color[rgb]{0,0,0}\rotatebox{-41.572799}{\makebox(0,0)[lb]{\smash{\(r=2m\)}}}}%
    \put(0.10715079,0.16835104){\color[rgb]{0,0,0}\rotatebox{-45.47781624}{\makebox(0,0)[lb]{\smash{\(\mathcal{I}^-\)}}}}%
    \put(0.13669219,0.37545293){\color[rgb]{0,0,0}\rotatebox{46.4073224}{\makebox(0,0)[lb]{\smash{\(\mathcal{I}^+\)}}}}%
    \put(0.63701929,0.2379531){\color[rgb]{0,0,0}\makebox(0,0)[lb]{\smash{\(r=3m\)}}}%
    \put(0.41564379,0.42016166){\color[rgb]{0,0,0}\makebox(0,0)[lb]{\smash{\(\mathcal{H}^+\)}}}%
  \end{picture}%
\endgroup%
 \label{Pen}
\end{center}
The diamond shaped region to the right corresponds to the Lorentzian manifold \((M,g)\) we started with; it represents the exterior of the black hole. The triangle to the top corresponds to the interior of the black hole, which is separated from the exterior by the so called event horizon, the line from the centre to the top-right \emph{i}\(^+\). The remaining parts of the Penrose diagram play no role in the following discussion. 

The black hole stability problem (see the introduction of \cite{DafRod08}) motivates the study of the wave equation in the exterior of the black hole (the event horizon included). In accordance with our discussion in Section \ref{energymethodGB}, we consider the framework of the energy method for the study of the wave equation.
A suitable notion of energy for the black hole exterior is obtained via \eqref{N-energy} through the foliation given by \(\st = \{t^* = \tau\}\) for \(t^* \geq c > -\infty\), where  \(t^* = t + 2m\log(r-2m)\), together with the timelike vector field \(N := -(dt^*)^\sharp\).\footnote{ \label{suitable} We are intentionally quite vague about what we mean by `suitable notion of energy'. Instead of considering a foliation that ends at spacelike infinity \(\iota^0\), it is sometimes desirable to work with a foliation that ends at future null infinity \(\mathcal{I}^+\). In a stationary spacetime, however, it is always convenient (and indeed `suitable'...) to work with a foliation and an energy measuring vector field \(N\) both of which are invariant under the flow of the Killing vector field. The obvious advantage is that the constants in Sobolev embeddings do not depend on the leaf - of course provided that higher energy norms are also defined accordingly. The precise choice of the timelike vector field \(N\) in a \emph{compact} region of one leaf is completely irrelevant, since all the energy norms are equivalent in a compact region. In particular one can deduce that the following result about trapping at the photonsphere in Schwarzschild remains unchanged if we choose a different timelike vector field \(N\) which commutes with \(\partial_t\) and a different foliation by spacelike slices. In fact note that the behaviour of the energy of the null geodesic, \(-g(N,\dot{\gamma})\), does not depend at all on the choice of the foliation!}

\subsubsection{Trapping at the photon sphere}
\label{photonsphere}
There are null geodesics in the Schwarzschild spacetime that stay forever on the photon sphere at \(\{r=3m\}\). Indeed, one can check that the curve \(\gamma\), given by
\begin{equation*}
\gamma(s) = ( s , 3m , \frac{\pi}{2} , (27m^2)^{-\frac{1}{2}}s )
\end{equation*}
in \((t,r,\theta, \varphi)\) coordinates is an affinely parametrised null geodesic, whose  \(N\)-energy is given by  \(-g(N,\dot{\gamma}) = 1\). We now apply Theorem \ref{LED}: The time oriented\footnote{The time orientation is given by the timelike vector field \(N\).} globally hyperbolic Lorentzian manifold can be taken to be the domain of dependence \(\mathcal{D}(\so)\) of \(\so\) in \((M,g)\). Moreover, we choose the time function to be given by the restriction of \(t^*\) to \(\mathcal{D}(\so)\), and the vector field \(N\) and null geodesic \(\gamma(s)\) in Theorem \ref{LED} are given by \(N\) and \(\gamma\big(s - 2m\log(m)\big)\) from above. Since \(-g(N,\dot{\gamma}) = 1\) holds, Theorem \ref{LED} now states that given any open neighbourhood \(\mathcal{T}\) of \(\I(\gamma)\) in \(\mathcal{D}(\so)\), there is  no function \(P : [0,\infty) \to (0,\infty)\) with \(P(\tau) \to 0\) for \(\tau \to \infty\) such that 
\begin{equation*}
E^N_{\tau, \mathcal{T} \cap \st}(u) \leq P(\tau) E^N_0(u)
\end{equation*}
holds for all solutions \(u\) of the wave equation for all \(\tau \geq 0\). It follows, that an LED statement for such a region can only hold if it loses differentiability. One can infer the analogous result about ILED statements from Theorem \ref{ILED}.

Let us mention here that \(\gamma\) has conjugate points. Indeed, the Jacobi field \(J\) with initial data \(J(0) =0\) and \(D_sJ(0) = \partial_\theta|_{\gamma(0)}\) vanishes in finite affine time \(s>0\): First note that the vector field
\begin{equation*}
s \mapsto \partial_{\theta}|_{\gamma(s)}
\end{equation*}
along \(\gamma\) is parallel, i.e., \(D_s \partial_{\theta}\big|_{\gamma(s)} =0\). Moreover, a direct computation yields
\begin{equation*}
R(\partial_\theta, \dot{\gamma})\, \dot{\gamma}\,\big|_{\gamma(s)} = \frac{1}{27m^2}\partial_{\theta}\big|_{\gamma(s)}\,,
\end{equation*}
where \(R(\cdot, \cdot) \,\cdot\) is the Riemann curvature endomorphism. Thus, it follows that the vector field
\begin{equation*}
J(s) = (27m^2)^{\frac{1}{2}} \sin \big( (27m^2)^{-\frac{1}{2}} s\big) \cdot \partial_{\theta}\big|_{\gamma(s)}
\end{equation*}
satisfies the Jacobi equation \(D_t^2J + R(J,\dot{\gamma}) \,\dot{\gamma} =0\). Moreover, it clearly satisfies the above initial conditions and vanishes in finite affine time.

It now follows from Theorem \ref{breakdown} that one cannot construct localised solutions to the wave equation along the trapped null geodesic \(\gamma\) using the naive geometric optics approximation alone. Indeed, one would need to bridge these caustics using Maslov's canonical operator.

That one can indeed prove an (I)LED statement with a loss of derivative was shown in \cite{DafRod09a} (see also \cite{BlueSter06}). In fact, it is sufficient to lose only an \(\varepsilon\) of a derivative, see \cite{BlueSoff08} and also \cite{DafRod08}. For a numerical study of the behaviour of a wave trapped at the photon sphere we refer the interested reader to \cite{ZenGall12}.

Other, similar, examples are trapping at the photon sphere in higher dimensional Schwarzschild \cite{Schl10} or in Reissner-Nordstr\"om \cite{Are11a} and \cite{BlueSoff08}.

\subsubsection{The red-shift effect at the event horizon - and its relevance for scattering constructions from the future}
\label{red2}
Another kind of behaviour of the energy is exhibited by the trapping occuring at the event horizon of the Schwarzschild spacetime. Recall that the event horizon \(\mathcal{H}^+\) at \(\{r=2m\}\) is a null hypersurface, spanned by null geodesics. In \((t^*,r,\theta,\varphi)\) coordinates the affinely parametrised generators are given by
\begin{equation*}
\gamma(s) = ( \frac{1}{\kappa} \mathrm{log}(s) , 2m , \theta_0 , \varphi_0  ) \;,
\end{equation*} 
where \(\kappa = \frac{1}{4m}\) is the surface gravity, \(s \in (0,\infty)\) and \(\theta_0\), \(\varphi_0\) are constants.
Thus, we have
\begin{equation}
\label{energyhorizon}
-\big(\dot{\gamma}(s),N\big) = \frac{1}{\kappa s} = \frac{1}{\kappa} e^{-\kappa t^*} \;,
\end{equation}
i.e., the energy of the corresponding Gaussian beam decays exponentially - a direct manifestation of the celebrated red-shift effect. For more on the impact of the red-shift effect on the study of the wave equation on Schwarzschild we refer the reader to the original paper \cite{DafRod09a} by Dafermos and Rodnianski, but also see their \cite{DafRod08}. 

Let us emphasise again that the null geodesics at the photon sphere as well as those at the horizon are trapped in the sense that they never escape to null infinity - but only those at the photon sphere form an obstruction for an LED statement without loss of differentiability to hold; the `trapped' energy at the horizon decays exponentially. This is in stark contrast to the obstacle problem where every trapped light ray automatically leads to an obstruction for an LED statement without loss of derivatives to hold (see \cite{Ral69}). This new variety of how the `trapped' energy behaves is due to the lack of a global timelike Killing vector field. 

Let us now investigate the role played by the red-shift effect in scattering  constructions from the future. While the red-shift effect is conducive to proving bounds on  solutions to the wave equation in the `forward problem', it turns into a blue-shift in the `backwards problem'\footnote{We call the initial value problem on \(\so\) to the future the `forward problem', while solving a mixed characteristic initial value problem on \(\mathcal{H}^+(\tau) \cup \st\) to the past (or indeed a scattering construction from the future with data on \(\mathcal{H}^+\) and \(\mathcal{I}^+\)) is called the `backwards problem'. Here, we have denoted the (closed) portion of the event horizon \(\mathcal{H}^+\) that is cut out by \(\so\) and \(\st\) by \(\mathcal{H}^+(\tau)\).}; it amplifies energy near the horizon.

\begin{center}
\def\svgwidth{5.5cm}
\begingroup%
  \makeatletter%
  \providecommand\color[2][]{%
    \errmessage{(Inkscape) Color is used for the text in Inkscape, but the package 'color.sty' is not loaded}%
    \renewcommand\color[2][]{}%
  }%
  \providecommand\transparent[1]{%
    \errmessage{(Inkscape) Transparency is used (non-zero) for the text in Inkscape, but the package 'transparent.sty' is not loaded}%
    \renewcommand\transparent[1]{}%
  }%
  \providecommand\rotatebox[2]{#2}%
  \ifx\svgwidth\undefined%
    \setlength{\unitlength}{423.96782227bp}%
    \ifx\svgscale\undefined%
      \relax%
    \else%
      \setlength{\unitlength}{\unitlength * \real{\svgscale}}%
    \fi%
  \else%
    \setlength{\unitlength}{\svgwidth}%
  \fi%
  \global\let\svgwidth\undefined%
  \global\let\svgscale\undefined%
  \makeatother%
  \begin{picture}(1,1.02878325)%
    \put(0,0){\includegraphics[width=\unitlength]{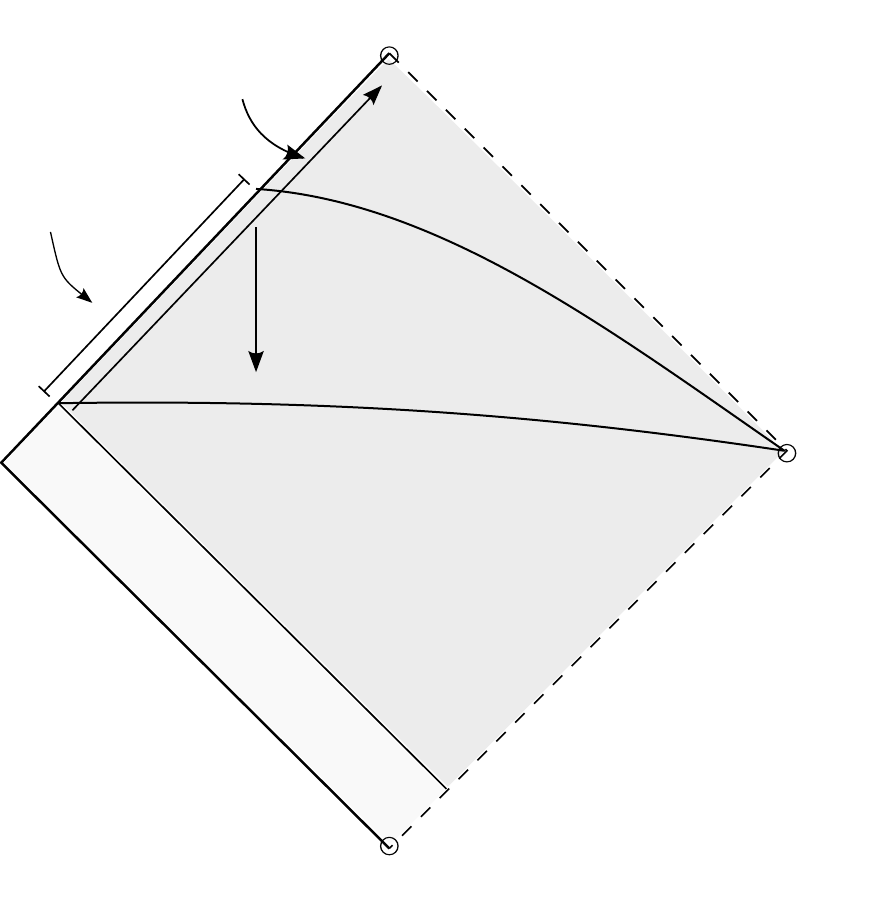}}%
    \put(0.90983989,0.49681222){\color[rgb]{0,0,0}\makebox(0,0)[lb]{\smash{\emph{i}\(^0\)}}}%
    \put(0.45697539,0.97932859){\color[rgb]{0,0,0}\makebox(0,0)[lb]{\smash{\emph{i}\(^+\)}}}%
    \put(0.41114977,0.0142958){\color[rgb]{0,0,0}\makebox(0,0)[lb]{\smash{\emph{i}\(^-\)}}}%
    \put(0.02555958,0.77820772){\color[rgb]{0,0,0}\makebox(0,0)[lb]{\smash{\(\mathcal{H}^+(\tau)\)}}}%
    \put(0.44888854,0.79602627){\color[rgb]{0,0,0}\makebox(0,0)[lb]{\smash{\(\Sigma_\tau\)}}}%
    \put(0.57827844,0.49950776){\color[rgb]{0,0,0}\makebox(0,0)[lb]{\smash{\(\Sigma_0\)}}}%
    \put(0.41654108,0.34046606){\color[rgb]{0,0,0}\makebox(0,0)[lb]{\smash{\(\mathcal{D}(\Sigma_0)\)}}}%
    \put(0.31903618,0.70148791){\color[rgb]{0,0,0}\makebox(0,0)[lb]{\smash{solve}}}%
    \put(0.31831758,0.63663497){\color[rgb]{0,0,0}\makebox(0,0)[lb]{\smash{backwards}}}%
    \put(0.24779207,0.93398552){\color[rgb]{0,0,0}\makebox(0,0)[lb]{\smash{\(\gamma_\tau\)}}}%
  \end{picture}%
\endgroup%

\end{center}

\begin{proposition}
\label{backsharp}
For every \(\mu >0\) and for every \(\tau >0\) there exists a smooth solution\footnote{We denote with \(\overline{\mathcal{D}(\so)}\) the closure of \(\mathcal{D}(\so)\) in the maximal analytic extension of Schwarzschild, see the Penrose diagram on page \pageref{Pen}.} \(v \in C^\infty(\overline{\mathcal{D}(\so)}, \C)\) to the wave equation \eqref{waveeq} with \(E^N_{\tau}(v) = 1\) and \(\int_{\mathcal{H}^+(\tau)} J^N(v) \, \lrcorner \, \vg < \mu\), which satisfies \(E^N_0(v) \geq e^{\kappa \tau} - \mu \), where \(\kappa = \frac{1}{4m}\) is the surface gravity of the Schwarzschild black hole. 
\end{proposition}
Here, \(J^N(v) \, \lrcorner \, \vg\) denotes the \(3\)-form obtained by inserting the vector field \(J^N(v)\) into the first slot of \(\vg\). Let us also remark that  \(\mu\) should be thought of as a small positive number, while \(\tau\) rather as a big one.

\begin{proof}
As in Section \ref{photonsphere} we consider the Lorentzian manifold \(\mathcal{D}(\so)\) with time function \(t^*\) and timelike vector field \(N\). Since geodesics depend smoothly on their initial data, it follows from \eqref{energyhorizon} that we can find for every \(\tau >0\) an affinely parametrised radially outgoing null geodesic\footnote{Radially outgoing null geodesics are the lines parallel to, and to the right of, \(\mathcal{H}^+\) in the Penrose diagram. In \((u,r,\theta,\varphi)\) coordinates, where \(u(t,r,\theta, \varphi) := t- 2m\log(r-2m) - r\), these null geodesics are tangent to \(\frac{\partial}{\partial r}\).}  \(\gamma_\tau\) in \(\mathcal{D}(\so)\) with \(\big|-\big(\dot{\gamma}_\tau,N\big)|_{\I(\gamma_\tau) \cap \so} - e^{\kappa \tau}\big| < \frac{\mu}{2}\) and \(-\big(\dot{\gamma }_\tau,N\big)|_{\I(\gamma_\tau) \cap \st} =1 \). 
We note that for our choice of time function and vector field \(N\) the condition \eqref{energyestimatecond} is satisfied, which does not only give us the energy estimate \eqref{energyestimate}, but here also the refined version 
\begin{equation}
\label{ee}
\int_{\mathcal{H}^+(\tau)} J^N(u)\, \lrcorner\, \vg + E^N_{\tau}(u) \leq C(\tau)\big( E^N_0(u) + ||\Box u||^2_{L^2(R_{[0,T]})}\big) \;,
\end{equation}
which holds in \(\overline{\mathcal{D}(\so)}\) for all \(\tau >0\) and for all \(u \in C^\infty(\overline{\mathcal{D}(\so)}, \R)\). The estimate \eqref{ee} is derived in the same way as \eqref{energyestimate}, namely by an application of Stokes' theorem to \(J^N(u) \,\lrcorner\, \vg\), followed by Gronwall's inequality. The estimate \eqref{ee} gives in addition to \eqref{thmlocapp} in Theorem \ref{localised} the estimate
\begin{equation}
\label{refined}
\int_{\mathcal{H}^+(\tau)} J^N(v - \tilde{u}) \,\lrcorner\, \vg < \mu\;,
\end{equation}
where \(\tilde{u}\) is the Gaussian beam and \(v\) is the actual solution, as in Theorem \ref{localised}. We now apply Theorem \ref{symbiosis}, where the Lorentzian manifold is given by \(\mathcal{D}(\so)\), the time function by \(t^*\), the timelike vector field by \(N\) and for given \(\tau >0\), the affinely parametrised null geodesic is taken to be \(\gamma_\tau\) from above. For our purposes we can choose any neighbourhood \(\N\) of \(\I(\gamma_\tau)\) \emph{in} \(\mathcal{D}(\so)\). Theorem \ref{symbiosis} then ensures the existence of a solution \(v \in C^\infty(\mathcal{D}(\so), \C)\) to the wave equation with \(E^N_0(v) \geq e^{\kappa \tau} - \mu \) and \(E^N_{\tau}(v) = 1\) -- possibly after renormalising the energy at time \(\tau\) of \(v\) to be exactly \(1\). It is not difficult to show, for example by considering the Cauchy problem for a slightly larger globally hyperbolic Lorentzian manifold which contains the event horizon, that \(v\) can be chosen to extend smoothly to the event horizon. We then obtain \(\int_{\mathcal{H}^+(\tau)} J^N(v) \, \lrcorner \, \vg < \mu\) from \eqref{refined}, since recall that the Gaussian beam \(\tilde{u}\) in Theorem \ref{localised} is supported in \(\N\), which is disjoint from \(\mathcal{H}^+\). This finishes the proof.
\end{proof}

The above proposition shows that for every \(\tau >0\) one can prescribe initial data for the mixed characteristic initial value problem on \(\mathcal{H}^+ \cup \st\) such that the total initial energy is equal to one, while the energy of the  solution obtained by solving backwards grows exponentially to  \(\approx e^{\kappa \tau}\) on \(\so\). In \cite{DafHolRod13}, Dafermos, Holzegel and Rodnianski approach the scattering problem from the future for the Einstein equations (with initial data prescribed on \(\mathcal{H}^+\) and \(\mathcal{I}^+\)) by considering it as the limit of finite backwards problems, which - for the wave equation - are qualitatively the same as the backwards problem with initial data on \(H^+(\tau)\) and \(\st\). In order to take the limit of the finite problems, uniform control over the solutions is required: Dafermos et al.\ use a backwards energy estimate which bounds the energy on \(\so\) by the initial energy on \(\mathcal{H}^+\) and \(\st\), multiplied by \(C \cdot \exp(c \tau)\), where \(c\) and \(C\) are constants that are independent of \(\tau\). Proposition \ref{backsharp} shows now that this estimate is sharp in the sense that one cannot avoid exponential growth (at least not as long as one does not sacrifice regularity in the estimate). In particular, working with this estimate enforces the assumption of exponential decay on the scattering data in \cite{DafHolRod13}.

\subsubsection{The blue-shift near the Cauchy horizon of a sub-extremal Reissner-Nordstr\"om black hole}
\label{sub-extremal}
We now move on to the sub-extremal Reissner-Nordstr\"om black hole, i.e., to the parameter range \(0<e < m\) in \eqref{RNfamily}. More precisely, we consider again its maximal analytic extension. Part of the Penrose diagram is given below: 

\begin{center}
\def\svgwidth{8cm}
\begingroup%
  \makeatletter%
  \providecommand\color[2][]{%
    \errmessage{(Inkscape) Color is used for the text in Inkscape, but the package 'color.sty' is not loaded}%
    \renewcommand\color[2][]{}%
  }%
  \providecommand\transparent[1]{%
    \errmessage{(Inkscape) Transparency is used (non-zero) for the text in Inkscape, but the package 'transparent.sty' is not loaded}%
    \renewcommand\transparent[1]{}%
  }%
  \providecommand\rotatebox[2]{#2}%
  \ifx\svgwidth\undefined%
    \setlength{\unitlength}{397.58457031bp}%
    \ifx\svgscale\undefined%
      \relax%
    \else%
      \setlength{\unitlength}{\unitlength * \real{\svgscale}}%
    \fi%
  \else%
    \setlength{\unitlength}{\svgwidth}%
  \fi%
  \global\let\svgwidth\undefined%
  \global\let\svgscale\undefined%
  \makeatother%
  \begin{picture}(1,1.1665886)%
    \put(0,0){\includegraphics[width=\unitlength]{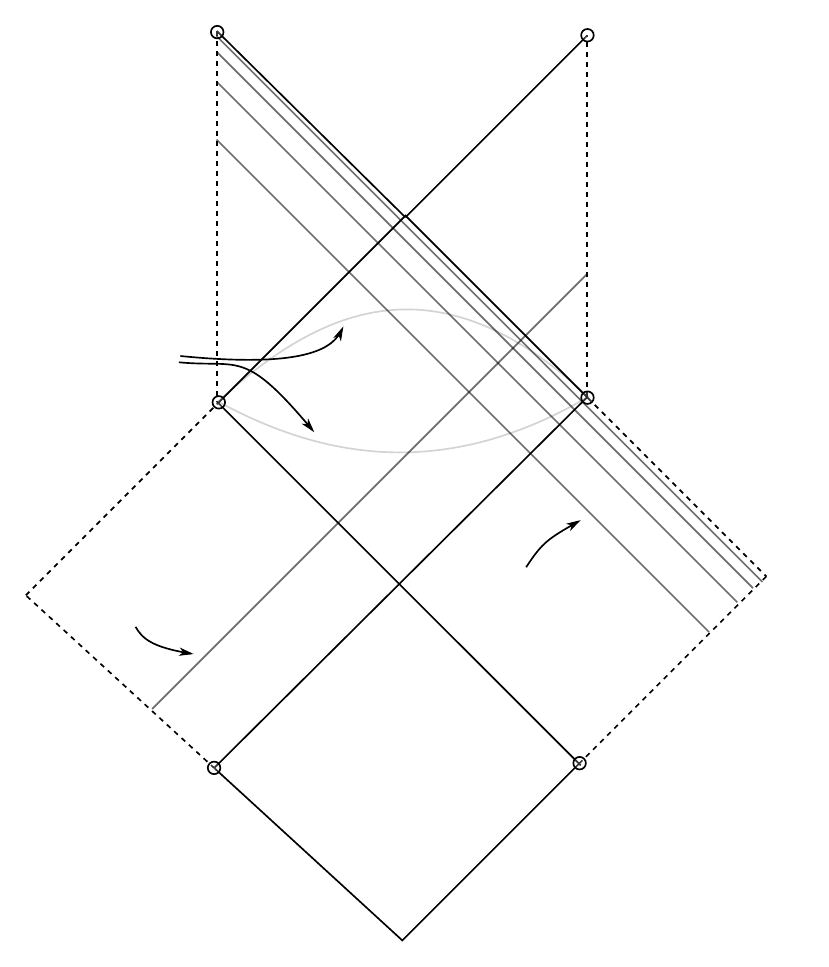}}%
    \put(0.81348285,0.59936823){\color[rgb]{0,0,0}\rotatebox{-46.55904601}{\makebox(0,0)[lb]{\smash{\(\mathcal{I}^+\)}}}}%
    \put(0.81824573,0.30635281){\color[rgb]{0,0,0}\rotatebox{43.90880184}{\makebox(0,0)[lb]{\smash{\(\mathcal{I}^-\)}}}}%
    \put(0.57088621,0.56539376){\color[rgb]{0,0,0}\rotatebox{43.50347509}{\makebox(0,0)[lb]{\smash{\(\mathcal{H}^+\)}}}}%
    \put(0.72997213,0.94202906){\color[rgb]{0,0,0}\rotatebox{-92.51717097}{\makebox(0,0)[lb]{\smash{\(r=0\)}}}}%
    \put(0.24359868,0.87493747){\color[rgb]{0,0,0}\rotatebox{90.6592305}{\makebox(0,0)[lb]{\smash{\(r=0\)}}}}%
    \put(0.58146325,0.44698807){\color[rgb]{0,0,0}\makebox(0,0)[lb]{\smash{\(v=const\)}}}%
    \put(0.10417096,0.42425988){\color[rgb]{0,0,0}\makebox(0,0)[lb]{\smash{\(u=const\)}}}%
    \put(0.03219827,0.72162052){\color[rgb]{0,0,0}\makebox(0,0)[lb]{\smash{\(r=const\)}}}%
    \put(0.54806311,0.86559277){\color[rgb]{0,0,0}\rotatebox{-41.89181184}{\makebox(0,0)[lb]{\smash{\(r=r_-\)}}}}%
    \put(0.30384289,0.74193718){\color[rgb]{0,0,0}\rotatebox{46.05515158}{\makebox(0,0)[lb]{\smash{\(r=r_-\)}}}}%
    \put(0.72540852,0.68752818){\color[rgb]{0,0,0}\makebox(0,0)[lb]{\smash{\(i^+\)}}}%
    \put(0.93375041,0.46024617){\color[rgb]{0,0,0}\makebox(0,0)[lb]{\smash{\(i^0\)}}}%
    \put(0.66479998,0.32008897){\color[rgb]{0,0,0}\makebox(0,0)[lb]{\smash{\(I\)}}}%
    \put(0.37690937,0.58903934){\color[rgb]{0,0,0}\makebox(0,0)[lb]{\smash{\(II\)}}}%
    \put(0.28220854,0.84662568){\color[rgb]{0,0,0}\makebox(0,0)[lb]{\smash{\(III\)}}}%
    \put(0.6439658,1.00193501){\color[rgb]{0,0,0}\makebox(0,0)[lb]{\smash{\(IV\)}}}%
  \end{picture}%
\endgroup%

\end{center}
Again, the diamond-shaped region \(I\) represents the black hole exterior and  corresponds to the Lorentzian manifold on which the metric \(g\) from \eqref{RNfamily} was initially defined. The regions \(II\), \(III\) and \(IV\) represent the black hole interior. Recall that Reissner-Nordstr\"om is a spherically symmetric spacetime. The `radius' of the spheres of symmetry is given by a globally defined function \(r\). We write \(D(r):= 1 - \frac{2m}{r} + \frac{e^2}{r^2}\) and denote the two roots of \(D\) with \(r_\pm = m \pm \sqrt{m^2 - e^2}\). The future Cauchy horizon\footnote{We consider a Cauchy surface \(\Sigma_0\) of the big diamond shaped region as shown in the next diagram, i.e., a Cauchy surface of the region pictured in the above diagram without the regions \(III\) and \(IV\),.} is given by \(r=r_-\).
The coordinate functions \((\theta, \varphi)\) parametrise the spheres of symmetry in the usual way and are globally defined up to one meridian. Regions \(I - III\) are covered by a \((v,r,\theta, \varphi)\) coordinate chart, where in the region \(I\), the function \(v\) is given by \(v= t + r_I^*\), where \(r_I^*\) is a function of \(r\), satisfying \(\frac{dr_I^*}{dr} = \frac{1}{D}\). 
With respect to these coordinates, the Lorentzian metric takes the form
\begin{equation*}
g= -D \,dv^2 + dv \otimes dr + dr \otimes dv + r^2 \,d\theta^2 + r^2 \sin^2\theta \,d\varphi^2 \;.
\end{equation*}

Introducing a function \(r_{II}^*\) in region \(II\), which satisfies \(\frac{dr_{II}^*}{dr} = \frac{1}{D}\) in this region, and defining a function\footnote{One could also assign the functions \(t\) an index, specifying in which region they are defined. Note that these different functions \(t\) do \emph{not} patch together to give a globally defined smooth function!} \(t:= v - r_{II}^*\), we obtain a \((t,r,\theta,\varphi)\) coordinate system for region \(II\) in which the metric \(g\) is again given by the algebraic expression \eqref{RNfamily}.  
The regions \(II\) and \(IV\) are covered by a coordinate system \((u,r,\theta, \varphi)\), where the function \(u\) is given in region \(II\) by \(u= t- r_{II}^*\).  

Having laid down the coordinate functions we work with, we now investigate the family of affinely parametrised ingoing null geodesics, given in \((v,r,\theta,\varphi)\) coordinates by
\begin{equation*}
\gamma_{v_0}(s) = ( v_0 , -s , \theta_0 , \varphi_0  ) \;,
\end{equation*}
where \(s\in (-\infty,0)\) and we keep \(\theta_0\), \(\varphi_0\) fixed. Clearly, we have\footnote{Let us denote with a subscript on the partial derivative which other coordinate (apart from \(\theta\) and \(\varphi\)) remains fixed.}  \(\dot{\gamma}_{v_0} = - \frac{\partial}{\partial r} \big|_v\).
We are interested in the energy of these null geodesics in region \(II\) close to \(i^+\) (in the topology of the Penrose diagram), i.e.\ close to the Cauchy horizon separating region $II$ from region $IV$. A suitable notion of energy is given by a regular vector field that is future directed timelike in a neighbourhood of \(i^+\). In order to construct such a vector field, we consider \((u,v,\theta,\varphi)\) coordinates in region \(II\). A straightforward computation shows that 
\begin{equation*}
\begin{split}
N&:= - \frac{1}{r_+ - r} \frac{\partial}{\partial u}\Big|_v + \frac{1}{r - r_-} \frac{\partial}{\partial v}\Big|_u \\[2pt]
&= -\frac{1}{r_+ - r} \frac{\partial}{\partial u} \Big|_r - \frac{1}{2r^2}(r_+ - r_-) \frac{\partial}{\partial r}\Big|_u  \\[2pt]
&= \frac{r_- -r_+}{2r^2} \frac{\partial}{\partial r}\Big|_v + \frac{1}{r - r_-} \frac{\partial}{\partial v}\Big|_r
\end{split}
\end{equation*}
is future directed timelike in a neighbourhood of \(i^+\) intersected with region \(II\) and can be extended to a smooth timelike vector field defined on a neighbourhood of \(i^+\). We obtain
\begin{equation}
\label{comp}
-(N,\dot{\gamma}_{v_0}) = \frac{1}{r-r_-} \;,
\end{equation}
the \(N\)-energy of the null geodesics \(\gamma_{v_0}\) gets infinitely blue-shifted near the Cauchy horizon.

For later reference let us note that the rate, in advanced time \(v\), with which the \(N\)-energy \eqref{comp} of \(\gamma_{v_0}\) blows up along a hypersurface of constant \(u\), is exponential. This is seen as follows: One has 
\begin{equation*}
r^*_{II}(r) = r + \frac{1}{2\kappa_+} \log(r_+-r) + \frac{1}{2\kappa_-} \log(r-r_-) + const \;,
\end{equation*}
where \(\kappa_{\pm} = \frac{r_\pm - r_\mp}{2r_\pm^2}\) are the surface gravities of the event and the Cauchy horizon, respectively.
Thus, for large \(r^*_{II}\) one has \(\frac{1}{r-r_-}(r^*_{II}) \sim e^{-2\kappa_- r^*_{II}}\). Finally, along \(\{u = u_0 = const\}\), we have \(r^*_{II}(v) = \frac{1}{2}(v - u_0)\). It thus follows that the \(N\)-energy \eqref{comp} of \(\gamma_{v_0}\) blows up like \(e^{-\kappa_- v}\) along a hypersurface of constant \(u\).

Let us now consider spacelike slices \(\so\) and \(\Sigma_1\) as in the diagram below, where \(\so\) asymptotes to a hypersurface of constant \(t\) and \(\Sigma_1\) is extendible as a smooth spacelike slice into the neighbouring regions. 
\begin{center}
\def\svgwidth{6cm}
\begingroup%
  \makeatletter%
  \providecommand\color[2][]{%
    \errmessage{(Inkscape) Color is used for the text in Inkscape, but the package 'color.sty' is not loaded}%
    \renewcommand\color[2][]{}%
  }%
  \providecommand\transparent[1]{%
    \errmessage{(Inkscape) Transparency is used (non-zero) for the text in Inkscape, but the package 'transparent.sty' is not loaded}%
    \renewcommand\transparent[1]{}%
  }%
  \providecommand\rotatebox[2]{#2}%
  \ifx\svgwidth\undefined%
    \setlength{\unitlength}{397.58457031bp}%
    \ifx\svgscale\undefined%
      \relax%
    \else%
      \setlength{\unitlength}{\unitlength * \real{\svgscale}}%
    \fi%
  \else%
    \setlength{\unitlength}{\svgwidth}%
  \fi%
  \global\let\svgwidth\undefined%
  \global\let\svgscale\undefined%
  \makeatother%
  \begin{picture}(1,1.1665886)%
    \put(0,0){\includegraphics[width=\unitlength]{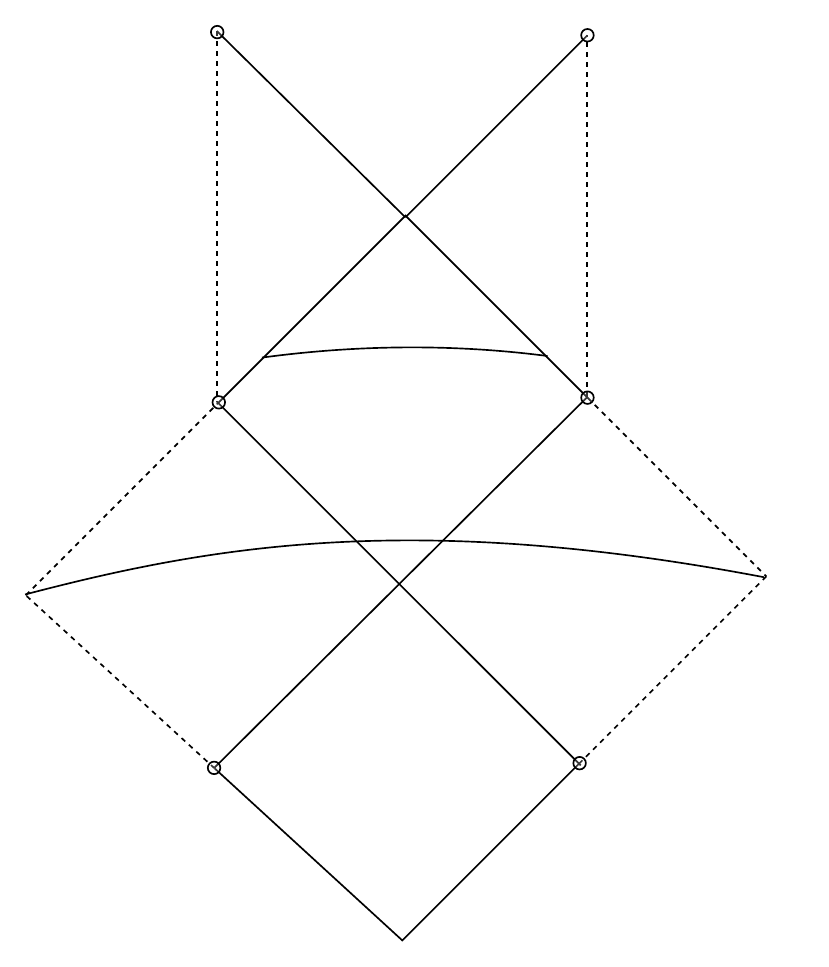}}%
    \put(0.81348285,0.59936823){\color[rgb]{0,0,0}\rotatebox{-46.55904601}{\makebox(0,0)[lb]{\smash{\(\mathcal{I}^+\)}}}}%
    \put(0.82771578,0.30067079){\color[rgb]{0,0,0}\rotatebox{43.90880184}{\makebox(0,0)[lb]{\smash{\(\mathcal{I}^-\)}}}}%
    \put(0.57088621,0.56539376){\color[rgb]{0,0,0}\rotatebox{43.50347509}{\makebox(0,0)[lb]{\smash{\(\mathcal{H}^+\)}}}}%
    \put(0.72997213,0.94202906){\color[rgb]{0,0,0}\rotatebox{-92.51717097}{\makebox(0,0)[lb]{\smash{\(r=0\)}}}}%
    \put(0.24359868,0.87493747){\color[rgb]{0,0,0}\rotatebox{90.6592305}{\makebox(0,0)[lb]{\smash{\(r=0\)}}}}%
    \put(0.54806311,0.86559277){\color[rgb]{0,0,0}\rotatebox{-41.89181184}{\makebox(0,0)[lb]{\smash{\(r=r_-\)}}}}%
    \put(0.30384289,0.74193718){\color[rgb]{0,0,0}\rotatebox{46.05515158}{\makebox(0,0)[lb]{\smash{\(r=r_-\)}}}}%
    \put(0.72540852,0.68752818){\color[rgb]{0,0,0}\makebox(0,0)[lb]{\smash{\(i^+\)}}}%
    \put(0.93375041,0.46024617){\color[rgb]{0,0,0}\makebox(0,0)[lb]{\smash{\(i^0\)}}}%
    \put(0.66669398,0.44698795){\color[rgb]{0,0,0}\makebox(0,0)[lb]{\smash{\(\Sigma_0\)}}}%
    \put(0.36743928,0.68942224){\color[rgb]{0,0,0}\makebox(0,0)[lb]{\smash{\(\Sigma_1\)}}}%
  \end{picture}%
\endgroup%

\end{center}
Since the normal \(n_{\Sigma_1}\) of \(\Sigma_1\) is also regular at the Cauchy horizon, it follows from \eqref{comp} that the \(n_{\Sigma_1}\)-energy of the null geodesics \(\gamma_{v_0}\) blows up along \(\Sigma_1\) when approaching the Cauchy horizon. Moreover, note that the \(n_{\so}\)-energy of the geodesics \(\gamma_{v_0}\) along \(\so\) is uniformly bounded as \(v_0 \to \infty\). We now apply Theorem \ref{symbiosis} to the family of the null geodesics \(\gamma_{v_0}\) with the following further input: the Lorentzian manifold is given by the domain of dependence \(\mathcal{D}(\Sigma_0)\) of \(\so\), the time function is such that \(\so\) and \(\Sigma_1\) are level sets, \(N\) is a timelike vector field that extends \(n_{\so}\) and \(n_{\Sigma_1}\), and finally \(\N\) is a small enough neighbourhoods of \(\gamma_{v_0}\). This yields
\begin{theorem}
\label{Blue1}
Let \(\so\) and \(\Sigma_1\) be spacelike slices in the sub-extremal Reissner-Nordstr\"om spacetime as indicated in the diagram below. Then there exists a sequence \(\{u_i\}_{i\in \mathbb{N}}\) of solutions to the wave equation with initial energy \(E^{n_{\so}}_0(u_i) =1\) on \(\so\) such that the \(n_{\Sigma_1}\)-energy on \(\Sigma_1\) goes to infinity, i.e., \(E^{n_{\Sigma_1}}_1(u_i) \to \infty\) for \(i \to \infty\).
\end{theorem} 
In particular we can infer from Theorem \ref{Blue1} that there is no uniform energy boundedness statement -- i.e., there is no constant \(C>0\) such that
\begin{equation}
\label{energyboundimp}
\int_{\Sigma_1} J^{n_{\Sigma_1}} (u) \cdot n_{\Sigma_1} \leq C \int_{\Sigma_0} J^{n_{\Sigma_0}} (u) \cdot n_{\Sigma_0} \;,
\end{equation}
holds for all solutions \(u\) of the wave equation.  

Let us remark here that the non-existence of a uniform energy boundedness statement has in particular the following consequence: one cannot choose a time function for the region bounded by \(\Sigma_0\) and \(\Sigma_1\) for which these hypersurfaces are level sets and, moreover, extend the normals of \(\Sigma_0\) and \(\Sigma_1\) to a smooth timelike vector field \(N\) \emph{in such a way that} an energy estimate of the form \eqref{energyestimate} holds. In particular this emphasises the importance of the condition \eqref{energyestimatecond} for the \emph{global} approximation scheme on general Lorentzian manifolds and points out the necessity of a \emph{local} understanding of the approximate solution provided by Theorem \ref{main} and \ref{symbiosis}.

We would also like to bring to the reader's attention that one actually expects that there is no energy boundedness statement at all, no matter how many derivatives one loses or whether one restricts the support of the initial data:
\begin{conjecture}
\label{Con}
For generic compactly supported smooth initial data on \(\Sigma_0\), the \(n_{\Sigma_1}\)-energy along \(\Sigma_1\) of the corresponding solution to the wave equation is infinite.
\end{conjecture}

Let us remark here that the analysis carried out in \cite{Daf05a} by Dafermos shows in particular that proving the above conjecture can be reduced to proving a lower bound on the decay rate of the spherical mean of the generic solution (as in Conjecture \ref{Con}) on the horizon.

Before we elaborate in Section \ref{Dis} on the mechanism that leads to the blow-up of the energy near the Cauchy horizon in Theorem \ref{Blue1}, let us investigate the situation for \emph{extremal} Reissner-Nordstr\"om black holes.

\subsubsection{The blue-shift near the Cauchy horizon of an extremal Reissner-Nordstr\"om black hole}
\label{BlueExtreme}

The extremal Reissner-Nordstr\"om black hole is given by the choice \(m=e\) of the parameters in \eqref{RNfamily}. We again consider the maximal analytic extension of the initially defined spacetime. Part of the Penrose diagram is given below:

\begin{center}
\def\svgwidth{4cm}
\begingroup%
  \makeatletter%
  \providecommand\color[2][]{%
    \errmessage{(Inkscape) Color is used for the text in Inkscape, but the package 'color.sty' is not loaded}%
    \renewcommand\color[2][]{}%
  }%
  \providecommand\transparent[1]{%
    \errmessage{(Inkscape) Transparency is used (non-zero) for the text in Inkscape, but the package 'transparent.sty' is not loaded}%
    \renewcommand\transparent[1]{}%
  }%
  \providecommand\rotatebox[2]{#2}%
  \ifx\svgwidth\undefined%
    \setlength{\unitlength}{382.34936523bp}%
    \ifx\svgscale\undefined%
      \relax%
    \else%
      \setlength{\unitlength}{\unitlength * \real{\svgscale}}%
    \fi%
  \else%
    \setlength{\unitlength}{\svgwidth}%
  \fi%
  \global\let\svgwidth\undefined%
  \global\let\svgscale\undefined%
  \makeatother%
  \begin{picture}(1,1.20368596)%
    \put(0,0){\includegraphics[width=\unitlength]{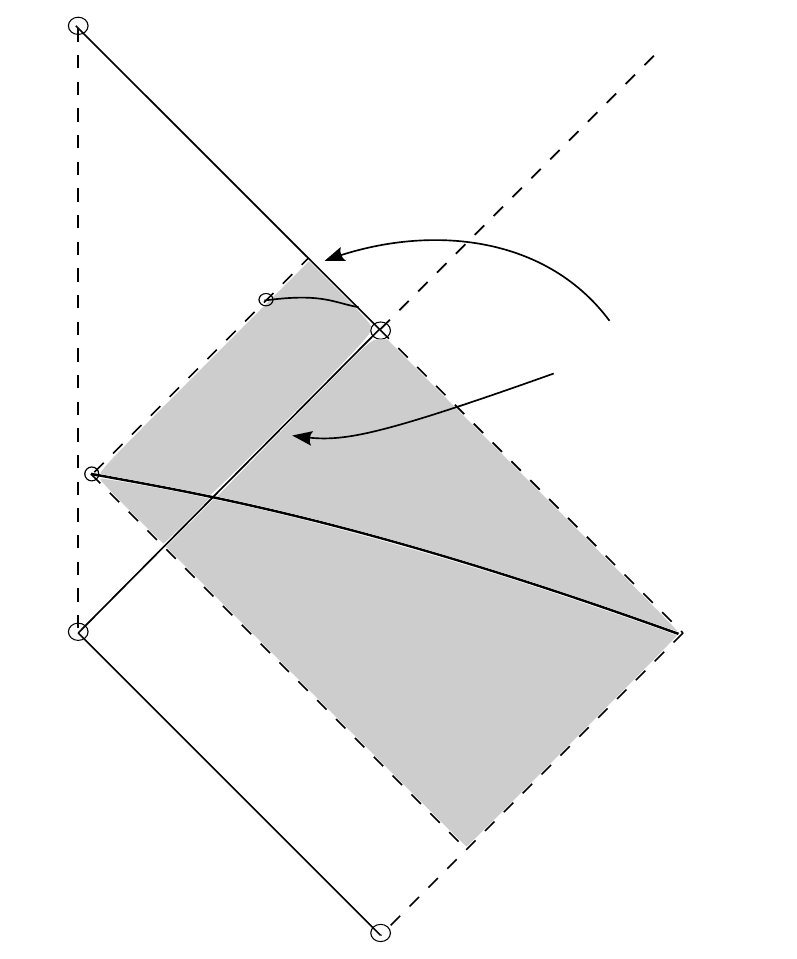}}%
    \put(0.06794213,0.73233168){\color[rgb]{0,0,0}\rotatebox{91.61201399}{\makebox(0,0)[lb]{\smash{\(r=0\)}}}}%
    \put(0.51951591,0.77612667){\color[rgb]{0,0,0}\makebox(0,0)[lb]{\smash{\(i^+\)}}}%
    \put(0.8783508,0.39111572){\color[rgb]{0,0,0}\makebox(0,0)[lb]{\smash{\(i^0\)}}}%
    \put(0.66687798,0.63791984){\color[rgb]{0,0,0}\rotatebox{-44.80161396}{\makebox(0,0)[lb]{\smash{\(\mathcal{I}^+\)}}}}%
    \put(0.72111551,0.157009){\color[rgb]{0,0,0}\rotatebox{46.73311562}{\makebox(0,0)[lb]{\smash{\(\mathcal{I}^-\)}}}}%
    \put(0.32760331,0.24940116){\color[rgb]{0,0,0}\makebox(0,0)[lb]{\smash{\(I\)}}}%
    \put(0.17119915,0.8581093){\color[rgb]{0,0,0}\makebox(0,0)[lb]{\smash{\(II\)}}}%
    \put(0.4332818,1.01028633){\color[rgb]{0,0,0}\makebox(0,0)[lb]{\smash{\(III\)}}}%
    \put(0.33183043,0.75031725){\color[rgb]{0,0,0}\makebox(0,0)[lb]{\smash{\(\Sigma_1\)}}}%
    \put(0.71227303,0.74186302){\color[rgb]{0,0,0}\makebox(0,0)[lb]{\smash{\(r = m\)}}}%
    \put(0.23395892,0.42307311){\color[rgb]{0,0,0}\rotatebox{46.21177709}{\makebox(0,0)[lb]{\smash{\(\mathcal{H}^+\)}}}}%
    \put(0.55268238,0.40119424){\color[rgb]{0,0,0}\makebox(0,0)[lb]{\smash{\(\Sigma_0\)}}}%
  \end{picture}%
\endgroup%

\end{center}
The region \(I\) represents again the black hole exterior and corresponds to the Lorentzian manifold on which the metric \(g\) from \eqref{RNfamily} was initially defined. The black hole interior extends over the regions \(II\) and \(III\). The discussion of the functions \(r,\theta\) and \(\varphi\) carries over from the sub-extremal case. However, in the extremal case, \(D(r)\) has a double zero at \(r=m\), the value of the radius of the spheres of symmetry on the event, as well as on the Cauchy horizon.  The regions \(I\) and \(II\) can be covered by `ingoing' null coordinates \((v,r,\theta,\varphi)\), where the function \(v\) is given in region \(I\) by \(v = t + r_{I}^*\), where again \(r^*_I(r)\) satisfies \(\frac{dr_I^*}{dr} = \frac{1}{D}\). In the same way as in the sub-extremal case one introduces \(r^*_{II}\) and defines a \((t,r,\theta,\varphi)\) coordinate system for the region \(II\). Finally, the regions \(II\) and \(III\) are covered by `outgoing' null coordinates \((u,r,\theta, \varphi)\), where we have \(u= t- r^*_{II}\) in region \(II\).

In ingoing null coordinates, the affinely parametrised radially ingoing null geodesics are given by \(
\gamma_{v_0}(s) = ( v_0 , -s , \theta_0 , \varphi_0  )\;\), 
where \(s\in (-\infty,0)\). Expressing the tangent vector of \(\gamma_{v_0}\) in region  \(II\) in outgoing coordinates, we obtain
\begin{equation}
\label{holonomy}
\dot{\gamma}_{v_0} = -\frac{\partial}{\partial r}\Big|_v = \frac{2}{D} \frac{\partial}{\partial u} \Big|_r - \frac{\partial}{\partial r} \Big|_u \;,
\end{equation}
which blows up at \(r=m\). Thus, we have for any future directed timelike vector field \(N\) in region \(II\) which extends to a regular timelike vector field to region \(III\), that the \(N\)-energy \(-g(\dot{\gamma}_{v_0},N)\) of \(\gamma_{v_0}\) blows up along the hypersurface \(\Sigma_1\) for \(v_0 \to \infty\). Choosing now a spacelike slice \(\Sigma_0\) as in the above diagram, again asymptoting to a \(\{t= const\}\) hypersurface at \emph{i}\(^0\), and restricting consideration to its domain of dependence, we obtain a globally hyperbolic spacetime (the shaded region) with respect to which we can apply Theorem \ref{symbiosis}, inferring the analogon of Theorem \ref{Blue1} for extremal Reissner-Nordstr\"om black holes.

For the discussion in the next section, we again investigate the rate, in advanced time \(v\), with which the \(N\)-energy \(-g(\dot{\gamma}_{v_0},N)\) blows up along a hypersurface of constant \(u\):  Here, we have 
\begin{equation*}
r^*_{II}(r) =r + m\log\big((r-m)^2\big) - \frac{m^2}{(r-m)} + const \;.
\end{equation*}
It follows that for large \(r^*_{II}\) one has \(\frac{1}{D}(r^*_{II}) \sim  (r^*_{II})^2\). Moreover, along \(\{ u = u_0 = const\}\), we have \(r^*_{II}(v) = \frac{1}{2}(v-u_0)\), from which it follows that the \(N\)-energy \(-g(\dot{\gamma}_{v_0},N)\) of the family of null geodesics \(\gamma_{v_0}\) blows up like \(v^2\).

\subsubsection{The strong and the weak blue-shift -- and their relevance for strong cosmic censorship}
\label{Dis}

In the example of sub-extremal Reissner-Nordstr\"om as well as in the example of extremal Reissner-Nordstr\"om, the energy of the Gaussian beams is blue-shifted near the Cauchy horizon. Although not important for the proof of the \emph{qualitative} result of Proposition \ref{Blue1} (and the analogous statement for the extremal case), the difference in the \emph{quantitative} blow-up rate of the energy in the two cases is conspicuous.

Let us first recall the familiar heuristic picture that explains the basic mechanism responsible for the blue-shift effect in both cases\footnote{Below, we give the picture for the sub-extremal case. However, the picture and the heuristics for the extremal case are exactly the same!}:
\begin{center}
\def\svgwidth{4cm}
\begingroup%
  \makeatletter%
  \providecommand\color[2][]{%
    \errmessage{(Inkscape) Color is used for the text in Inkscape, but the package 'color.sty' is not loaded}%
    \renewcommand\color[2][]{}%
  }%
  \providecommand\transparent[1]{%
    \errmessage{(Inkscape) Transparency is used (non-zero) for the text in Inkscape, but the package 'transparent.sty' is not loaded}%
    \renewcommand\transparent[1]{}%
  }%
  \providecommand\rotatebox[2]{#2}%
  \ifx\svgwidth\undefined%
    \setlength{\unitlength}{451.36787109bp}%
    \ifx\svgscale\undefined%
      \relax%
    \else%
      \setlength{\unitlength}{\unitlength * \real{\svgscale}}%
    \fi%
  \else%
    \setlength{\unitlength}{\svgwidth}%
  \fi%
  \global\let\svgwidth\undefined%
  \global\let\svgscale\undefined%
  \makeatother%
  \begin{picture}(1,1.03064491)%
    \put(0,0){\includegraphics[width=\unitlength]{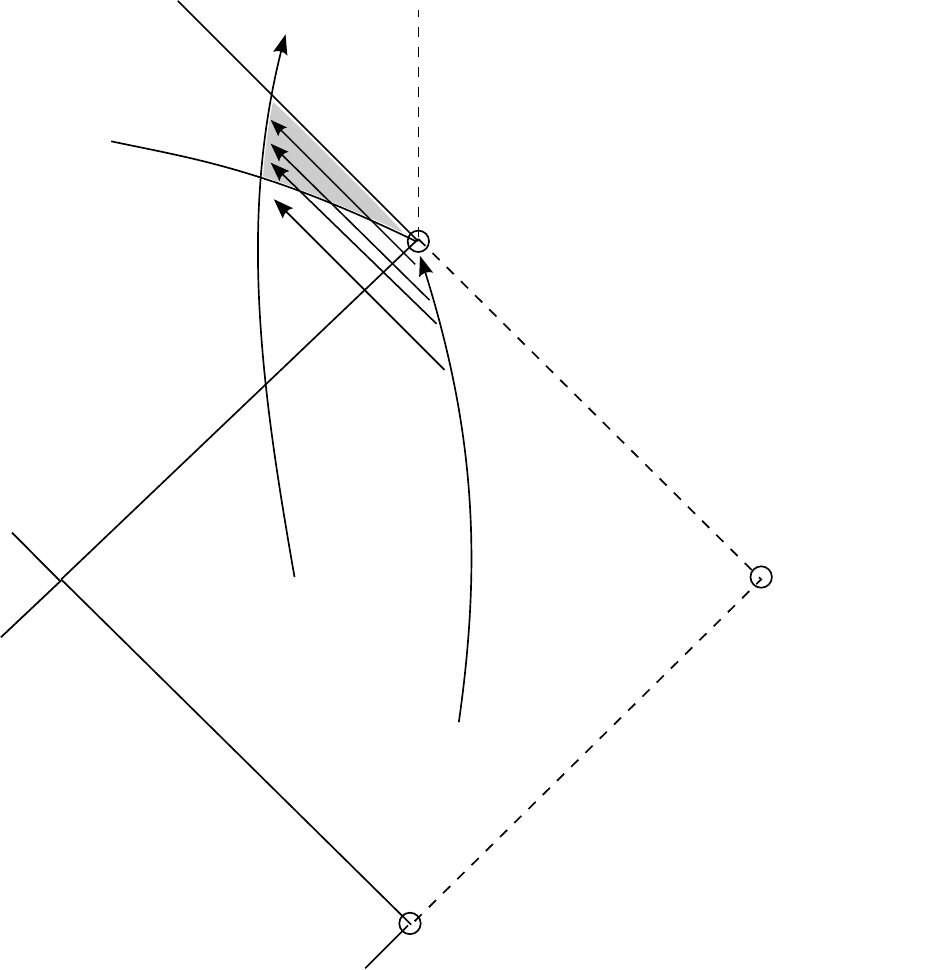}}%
    \put(0.64806636,0.62813214){\color[rgb]{0,0,0}\rotatebox{-47.10964241}{\makebox(0,0)[lb]{\smash{\(\mathcal{I}^+\)}}}}%
    \put(0.67545934,0.1625757){\color[rgb]{0,0,0}\rotatebox{45.60050893}{\makebox(0,0)[lb]{\smash{\(\mathcal{I}^-\)}}}}%
    \put(0.83732776,0.40435824){\color[rgb]{0,0,0}\makebox(0,0)[lb]{\smash{\emph{i}\(^0\)}}}%
    \put(0.47018985,0.78162411){\color[rgb]{0,0,0}\makebox(0,0)[lb]{\smash{\emph{i}\(^+\)}}}%
    \put(0.51576556,0.46259387){\color[rgb]{0,0,0}\makebox(0,0)[lb]{\smash{\(\sigma_0\)}}}%
    \put(0.30307873,0.56893717){\color[rgb]{0,0,0}\makebox(0,0)[lb]{\smash{\(\sigma_1\)}}}%
    \put(0.17772619,0.55422793){\color[rgb]{0,0,0}\rotatebox{43.7999252}{\makebox(0,0)[lb]{\smash{\(\mathcal{H}^+\)}}}}%
  \end{picture}%
\endgroup%

\end{center}
The observer \(\sigma_0\) travels along a timelike curve of infinite proper time to \emph{i}\(^+\) and, in regular time intervals, sends signals of the same energy into the black hole. These signals are received by the observer \(\sigma_1\), who travels into the black hole and crosses the Cauchy horizon, within \emph{finite} proper time - which leads to an infinite blue shift. This mechanism was first pointed out by Roger Penrose, see \cite{Pen68}, page 222.\footnote{There, he describes the above scenario in the following, more dramatic language (he considers the scenario of gravitational collapse, where the Einstein equations are coupled to some matter model and denotes the Cauchy horizon with \(H_+(\mathcal{H})\)):
\begin{quote}
There is a further difficulty confronting our observer who tries to cross \(H_+(\mathcal{H})\). As he looks out at the universe that he is ``leaving behind,'' he sees, in one final flash, as he crosses \(H_+(\mathcal{H})\), the entire later history of the rest of his ``old universe''. [...] If, for example, an unlimited amount of matter eventually falls into the star then presumably he will be confronted with an infinite density of matter along ``\(H_+(\mathcal{H})\).'' Even if only a finite amount of matter falls in, it may not be possible, in generic situations to avoid a curvature singularity in place of \(H_+(\mathcal{H})\).
\end{quote}}
Although the picture, along with its heuristics, allow for inferring the \emph{presence} of a blue-shift near the Cauchy horizon, they do not reveal the \emph{strength} of the blue-shift. For investigating the latter, it is important to note that the region in spacetime, which actually causes the blue shift, is a neighbourhood of the Cauchy horizon. This neighbourhood is not well-defined, however, one could think of it as being given by a neighbourhood of constant \(r\) -- the shaded region in the diagram of sub-extremal Reissner-Nordstr\"om above. The crucial difference between the sub-extremal and the extremal case is that in the extremal case, the blue-shift degenerates \emph{at the Cauchy horizon itself}, while in the sub-extremal case, it does not: the sub-extremal Cauchy horizon continues to blue-shift radiation. In particular, one can prove an analogous result to Proposition \ref{backsharp} there - but for the forward problem.

This degeneration of the blue-shift towards the Cauchy horizon in the extremal case leads to the (total) blue-shift
being weaker than the blue-shift in the sub-extremal case. Thus, the geometry of spacetime near the Cauchy horizon is crucial for understanding the strength of the blue-shift effect, and hence the blow-up rate of the energy.

We now continue with a heuristic discussion of the importance of the different blow-up rates. The reader might have noticed that we only made Conjecture \ref{Con} for the sub-extremal case; and indeed, the analogous conjecture for the extremal case is expected to be false: While in our construction we consider a family of ingoing wave packets whose energy along a fixed outgoing null ray to \(\mathcal{I}^+\) does \emph{not} decay in advanced time \(v\), the scattered `ingoing energy' of a wave with initial data as in Conjecture \ref{Con} will decay in advanced time \(v\) along such an outgoing null ray. Thus, the blow-up of the energy near the Cauchy horizon can be counteracted by the decay of the energy of the wave towards null infinity. In the extremal case, the blow-up rate is \(v^2\), which does not dominate the decay rate of the energy towards null infinity; the exponential blow-up rate \(e^{-\kappa_- v}\), however, does. These are the heuristic reasons for only formulating Conjecture \ref{Con} for the sub-extremal case. We conclude with a couple of remarks: Firstly, one should actually compare the decay rate of the ingoing energy not along an outgoing null ray to \(\mathcal{I}^+\), but along the event horizon - or even better, along a spacelike slice in the interior of the black hole approaching \emph{i}\(^+\) in the topology of the Penrose diagram. Secondly, we would like to repeat and stress the point made, namely that the heuristics given in the very beginning of this section, and which solely ensure the \emph{presence} of a blue-shift, are not sufficient to cause a \(C^1\) instability of the wave at the Cauchy horizon. For this to happen, the local geometry of the Cauchy horizon is crucial. Finally, let us conjecture, based on the fact that in the extremal case one gains powers of \(v\) in the blow-up rate at the Cauchy horizon when considering higher order energies, that there is some natural number \(k >1\) such that waves with initial data as in Conjecture \ref{Con} exhibit a \(C^k\) instability at the Cauchy horizon.

We conclude this section with recalling that the study of the wave equation on black hole backgrounds serves as a source of intuition for the behaviour of gravitational perturbations of these spacetimes. Thus, the following expected picture emerges: Consider a \emph{generic} dynamical spacetime which at late times approaches a sub-extremal Reissner-Nordstr\"om black hole. Then the Cauchy horizon is replaced by a weak null curvature singularity (for this notion see \cite{Daf05a}).

If we restrict consideration to the class of dynamical spacetimes which at late times approach an extremal Reissner-Nordstr\"om black hole, then the \emph{generic} spacetime \emph{within this class} has a more regular Cauchy horizon, which in particular is not seen as a singularity from the point of view of the low regularity well-posedness theory for the Einstein equations, see the resolution \cite{KlRodSzef12} of the \(L^2\)-curvature conjecture. This picture is also supported by the recent numerical work \cite{MuReTa13}.

\subsubsection{Trapping at the horizon of an extremal Reissner-Nordstr\"om black hole}
\label{extremal}
We again consider the extremal Reissner-Nordstr\"om black hole. With $v$ defined as in Section \ref{BlueExtreme}, we introduce the function \(t^* := v - r\). In the coordinates $(t^*,r,\theta, \varphi)$ the metric takes the form
\begin{equation*}
g= -D\,(dt^*)^2 + (1-D)\,(dt^* \otimes dr + dr \otimes dt^*) + (2-D)\,dr^2 + r^2\, d\theta^2 + r^2 \sin^2\theta\,d\varphi^2\;.
\end{equation*}
We see that the foliation of the exterior given by  \(\Sigma_\tau = \{t^* = \tau\}\) is a foliation by spacelike slices, which is invariant under the flow of the stationary Killing vector field $\partial_{t^*}$ and is regular at the event horizon $\mathcal{H}^+$ in the sense that it extends smoothly as a spacelike foliation across the event horizon. 

\begin{center}
\def\svgwidth{4cm}
\begingroup%
  \makeatletter%
  \providecommand\color[2][]{%
    \errmessage{(Inkscape) Color is used for the text in Inkscape, but the package 'color.sty' is not loaded}%
    \renewcommand\color[2][]{}%
  }%
  \providecommand\transparent[1]{%
    \errmessage{(Inkscape) Transparency is used (non-zero) for the text in Inkscape, but the package 'transparent.sty' is not loaded}%
    \renewcommand\transparent[1]{}%
  }%
  \providecommand\rotatebox[2]{#2}%
  \ifx\svgwidth\undefined%
    \setlength{\unitlength}{382.34936523bp}%
    \ifx\svgscale\undefined%
      \relax%
    \else%
      \setlength{\unitlength}{\unitlength * \real{\svgscale}}%
    \fi%
  \else%
    \setlength{\unitlength}{\svgwidth}%
  \fi%
  \global\let\svgwidth\undefined%
  \global\let\svgscale\undefined%
  \makeatother%
  \begin{picture}(1,1.20368596)%
    \put(0,0){\includegraphics[width=\unitlength]{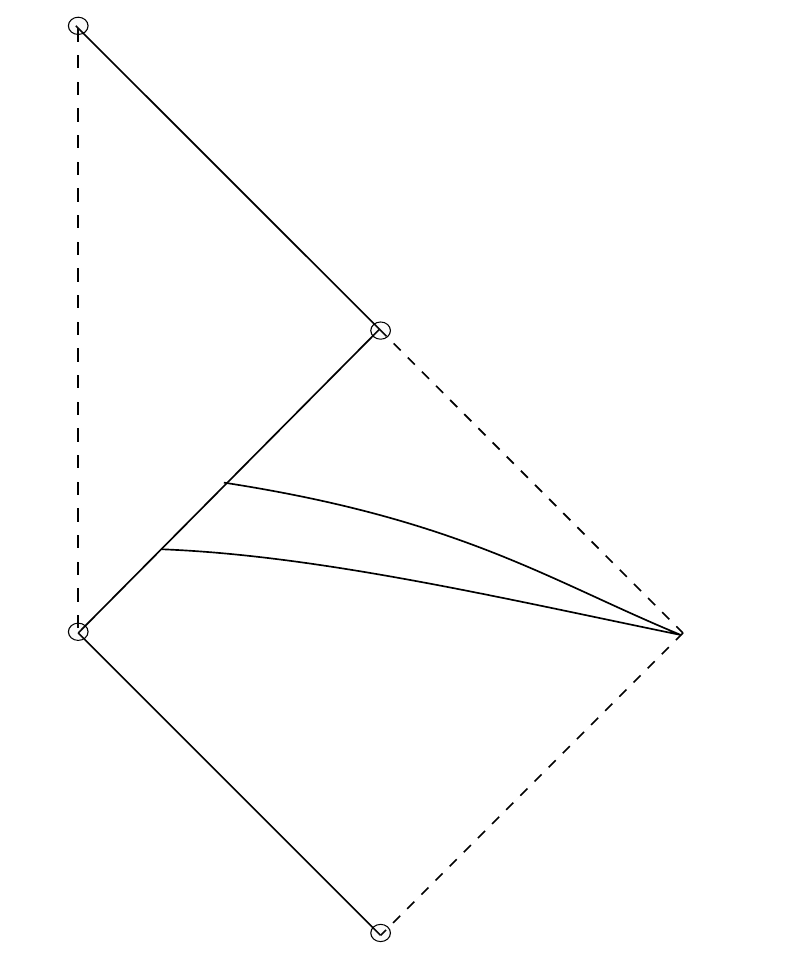}}%
    \put(0.06794213,0.73233167){\color[rgb]{0,0,0}\rotatebox{91.61201399}{\makebox(0,0)[lb]{\smash{\(r=0\)}}}}%
    \put(0.5047209,0.80360307){\color[rgb]{0,0,0}\makebox(0,0)[lb]{\smash{\(i^+\)}}}%
    \put(0.8783508,0.39111571){\color[rgb]{0,0,0}\makebox(0,0)[lb]{\smash{\(i^0\)}}}%
    \put(0.24562475,0.59021187){\color[rgb]{0,0,0}\rotatebox{44.43593592}{\makebox(0,0)[lb]{\smash{\(\mathcal{H}^+\)}}}}%
    \put(0.66687798,0.63791983){\color[rgb]{0,0,0}\rotatebox{-44.80161396}{\makebox(0,0)[lb]{\smash{\(\mathcal{I}^+\)}}}}%
    \put(0.7211155,0.15700899){\color[rgb]{0,0,0}\rotatebox{46.73311562}{\makebox(0,0)[lb]{\smash{\(\mathcal{I}^-\)}}}}%
    \put(0.3138491,0.42444353){\color[rgb]{0,0,0}\makebox(0,0)[lb]{\smash{\(\Sigma_0\)}}}%
    \put(0.41846547,0.59182965){\color[rgb]{0,0,0}\makebox(0,0)[lb]{\smash{\(\Sigma_\tau\)}}}%
  \end{picture}%
\endgroup%

\end{center}

An appropriate choice of timelike vector field for measuring the energy of waves in the black hole exterior is thus given by \(N=-(dt^*)^\sharp\), since it is also invariant under the flow of the Killing vector field $\partial_{t^*}$ and extends smoothly as a timelike vector field across the event horizon. Hence, the corresponding $N$-energy is non-degenerate at the event horizon. Note that these choices of foliation and timelike vector field $N$ correspond qualitatively to the choices made in the Schwarzschild spacetime in Sections \ref{photonsphere} and \ref{red2}.

In \cite{Are11a} and \cite{Are11b} Aretakis investigated the behaviour of waves on this spacetime and obtained stability (i.e., boundedness and decay results) as well as instability results (blow-up of certain higher order derivatives along the horizon); for further developments see also \cite{LuReall12}. The instability results originate from a conservation law on the extremal horizon once decay results for the wave are established.  In order to obtain these stability results Aretakis followed the new method introduced by Dafermos and Rodnianski in \cite{DafRod09b}.\footnote{Though in addition he had to work with a degenerate energy, which makes things more complicated.} The first important step is to prove an ILED statement. As in the Schwarzschild spacetime we have trapping at the photon sphere (here at \(\{r=2m\}\)), and as shown before, an ILED statement has to degenerate there in order to hold. The fundamentally new difficulty in the extremal setting arises from the degeneration of the red-shift effect at the horizon \(\mathcal{H}^+\), which was needed for proving an ILED statement that holds up to the horizon (see for example \cite{DafRod08}). And indeed, the energy of the generators of the horizon is no longer decaying: 
In \((t^*,r,\theta,\varphi)\) coordinates, the affinely parametrised generators are given by
\begin{equation*}
\gamma(s) = ( s , m , \theta_0 , \varphi_0  ) \;,
\end{equation*}
where \(s \in (-\infty,\infty)\) and again \(\theta_0\), \(\varphi_0\) are fixed. Hence, we see that the $N$-energy of the generators of the horizon is constant: \(-(N,\dot{\gamma}) =1\).

If we consider a globally hyperbolic subset of the depicted part of extremal Reissner-Nordstr\"om that contains the horizon \(\mathcal{H}^+\), for example by extending \(\so\) a bit through the event horizon and then considering its domain of dependence, we can directly infer from Theorem \ref{LED} and \ref{ILED}, by applying it to the null geodesic \(\gamma\) from above, that every (I)LED statement which concerns a neighbourhood of the horizon, necessarily has to lose differentiability. However, we can also infer the same result for the wave equation on the Lorentzian manifold \(\mathcal{D}(\so)\), where `a neighbourhood of the horizon' is `a neighbourhood of the horizon in the previous, bigger spacetime, intersected with \(\mathcal{D}(\so)\)': Analogous  to the proof of Proposition \ref{backsharp}, we consider a sequence of radially outgoing null geodesics in \(\mathcal{D}(\so)\) whose initial data on \(\so\) converges to the data of \(\gamma\) from above. For every `neighbourhood of the horizon', for every \(\tau_0 >0\) and for every (small) \(\mu >0\) there is then an element \(\gamma_0\) of the sequence such that \(-(N,\dot{\gamma}_0)|_{\I(\gamma_0) \cap \st} \in(1-\mu, 1+ \mu)\) for all \(0 \leq \tau \leq \tau_0\). This follows again from the smooth dependence of geodesics on their initial data. We now apply Theorem \ref{symbiosis} to this sequence of null geodesics to infer that for every `neighbourhood of the horizon' and for every \(\tau_0 >0\) we can construct a solution to the wave equation whose energy in this neighbourhood is, say, bigger than \(\frac{1}{2}\) for all times \(\tau\) with \(0 \leq \tau \leq \tau_0\). This proves again that there is no non-degenerate (I)LED statement concerning `a neighbourhood of the horizon' in \(\mathcal{D}(\so)\); the trapping at the event horizon obstructs local energy decay - which is in stark contrast to sub-extremal black holes. 

One should ask now whether an ILED statement with loss of derivative can actually hold. To answer this question, at least partially, it is helpful to decompose the angular part of the wave into spherical harmonics. In \cite{Are11a} Aretakis proved indeed an (I)LED statement with loss of one derivative for waves that are supported on the angular frequencies \(l\geq 1\). By constructing a localised solution with vanishing spherical mean we can show that this result is optimal in the sense that \emph{some} loss of derivative is again necessary. This can be done for instance by considering the superposition of two Gaussian beams that follow the generators \(\gamma_1(s) =(s, m, \frac{\pi}{2}, \frac{\pi}{2})\) and \(\gamma_2(s) =(s, m, \frac{\pi}{2}, \frac{3\pi}{2})\), where the initial value of beam one is exactly the negative of the initial value of beam two if translated in the \(\varphi\) variable by \(\pi\).\footnote{Let us mention here that in this particular situation the approximation using geometric optics is easier. Indeed, one can easily write down a solution of the eikonal equation such that the characteristics are the outgoing null geodesics. First one has to prove then the analogue of Theorem \ref{main}, which is easier since the approximate conservation law we used in the case of Gaussian beams is replaced by an exact conservation law for the geometric optics approximation, cf.\ footnote \ref{foot}. But then we can easily contradict the validity of (I)LED statements for any angular frequency: working in \((t^*,r,\theta,\varphi)\) coordinates, we choose the initial value of the function \(a\) (see Appendix \ref{geomop}) to have the angular dependence of a certain spherical harmonic and the radial dependence corresponds to a smooth cut-off, i.e., \(a\) initially is only non-vanishing for \(r \in [m,m+\varepsilon)\).}  
The question whether one can prove an ILED statement with loss of derivative in the case \(l=0\) is still open, though it is expected that the answer is negative. In order to obtain stability results for waves supported on all angular frequencies Aretakis had to use the degenerate energy (of course these results are weaker than results one would obtain if an ILED statement for the case \(l=0\) actually held).

\subsection{Applications to Kerr black holes} 
\label{KerrSec}
The Kerr family is a \(2\)-parameter family of solutions to the vacuum Einstein equations. Let us fix the manifold \(M:=\R \times (m + \sqrt{m^2 -a^2}, \infty) \times \mathbb{S}^2\), where \(m\) and \(a\) are real parameters that will model the mass and the angular momentum per unit mass of the black hole, respectively, and which are restricted to the range \(0\leq a \leq m\), \(0\neq m\). Let \((t,r,\theta,\varphi)\) denote the standard coordinates on the manifold \(M\) and define functions
\begin{align*}
\rho^2 &:= r^2 + a^2\cos^2\theta \\
\Delta &:= r^2 - 2mr +a^2 \\
g_{tt} &:= -1 + \frac{2mr}{\rho^2} \\
g_{t\varphi} &:= -\frac{2mra\sin^2\theta}{\rho^2} \\
g_{\varphi \varphi} &:= \big(r^2 +a^2 + \frac{2mra^2\sin^2\theta}{\rho^2}\big) \sin^2\theta \;.
\end{align*}
The metric on \(M\) is then defined by
\begin{equation*}
g= g_{tt}\,dt^2 - g_{t\varphi} \,(d\varphi \otimes dt + dt \otimes d\varphi) + g_{\varphi \varphi} \,d\varphi^2 + \frac{\rho^2}{\Delta} \,dr^2 + \rho^2 \,d\theta^2\;.
\end{equation*}
The roots of $\Delta(r)$ are denoted by $r_-$ and $r_+$, where $r_\pm = m \pm \sqrt{m^2 -a^2}$.
As for the Reissner-Nordstr\"om family, one can (and should) extend these spacetimes in order to understand their physical interpretation as a black hole. For details, we refer the reader again to \cite{HawkEllis}. Fixing the \(\theta\) coordinate to be \(\frac{\pi}{2}\) and moding out the \(\mathbb{S}^1\) corresponding to the \(\varphi\) coordinate, we again obtain pictorial representations of these spacetimes. For the sub-extremal case \(0 < a < m\), the diagram is the same as the one depicted in Section \ref{sub-extremal}, while in the extremal case \(a =m\), one obtains the same diagram as in Section \ref{extremal}.

\subsubsection{Trapping in (sub)-extremal Kerr}
\label{SecKerrTrap}
As in the case of the Schwarzschild spacetime there are trapped null geodesics in the domain of outer communications of the Kerr spacetime whose energy stays bounded away from zero and infinity if the energy measuring vector field \(N\) is sensibly chosen. In the case of \(a>0\), however, the set that accomodates trapped null geodesics is the closure of an \emph{open} set in spacetime, which is in contrast to the \(3\)-dimensional photonsphere in Schwarzschild and Reissner-Nordstr\"om. Before we explain in some more detail how to find the trapped geodesics, we set up a suitable  choice of foliation and energy measuring vector field:  

For (sub)-extremal Kerr we foliate the domain of outer communication (which is covered by the above \((t,r,\theta,\varphi)\) coordinates) in the same way as we did before for the Schwarzschild and the extremal Reissner-Nordstr\"om spacetimes, namely by first introducing an ingoing `null' coordinate \(v\) and then subtracting off \(r\) to get a good time coordinate \(t^*\). Slightly more general than needed at this point, let us define
\begin{equation*}
v_+ := t + r^* \qquad \textnormal{ and } \qquad \varphi_+ := \varphi + \bar{r}\;,
\end{equation*}
where \(r^*\) is defined up to a constant by \(\frac{dr^*}{dr} = \frac{r^2 + a^2}{\Delta}\) and \(\bar{r}\) is defined up to a constant by \(\frac{d\bar{r}}{dr} = \frac{a}{\Delta}\). The set of functions \((v_+, r, \theta, \varphi_+)\) form ingoing `null' coordinates (\(v_+\) is here the `null' coordinate, however, it does \emph{not} satisfy the eikonal equation \(d\phi \cdot d\phi =0\)), they cover the regions \(I, II\) and \(III\) in the spacetime diagram for sub-extremal Kerr\footnote{In the extremal case they cover all of the in Section \ref{extremal} depicted spacetime diagram.} and the metric takes the form
\begin{align*}
g= g_{tt} \,dv_+^2 + &g_{t\varphi}\,(dv_+ \otimes d\varphi_+ + d\varphi_+ \otimes dv_+) + (dv_+ \otimes dr + dr \otimes dv_+)  - a \sin^2\theta (dr \otimes d\varphi_+ + d\varphi_+ \otimes dr) \\ &+ g_{\varphi \varphi}\, d\varphi_+^2 + \rho^2 \, d\theta^2  \;.
\end{align*}
Finally, we define \(t^* := v_+ - r\). That this is indeed a good time coordinate is easily seen from writing the metric in \((t^*, r, \theta, \varphi_+)\) coordinates and restricting it to \(\{t^* = const\}\) slices: One obtains
\begin{equation*}
\bar{g}= (g_{tt} + 2)\, dr^2 + (g_{t\varphi} - a \sin^2\theta)\,(d\varphi_+ \otimes dr + dr \otimes d\varphi_+) + \rho^2 \,d\theta^2  + g_{\varphi \varphi}\, d\varphi_+^2 \;,
\end{equation*}
and the \((\theta, \theta)\) minor of this matrix is found to be \(2mr \sin^2\theta + (r^2 + a^2) \sin^2\theta - a^2\sin^4\theta\), which is positive away from the well understood coordinate singularity \(\theta = \{0,\frac{\pi}{2}\}\). Hence, the slices \(\Sigma_\tau := \{t^* = \tau\}\) are spacelike and it is easily seen that they asymptote to \(\{t=const\}\) slices near spacelike infinity and end on the \emph{future} event horizon. A suitable timelike vector field \(N\) for measuring the energy is again given by \(N:= -(dt^*)^\sharp\).

To be more precise about what we mean by a null geodesic being trapped, let us call a future complete affinely parametrised null geodesic $\gamma : [0,\infty) \to M$ (which is in particular contained in the black hole exterior $M$) \emph{trapped} if, and only if, it does not escape do infinity, i.e., for $s \to \infty$ we do \emph{not} have $(r \circ \gamma)(s) \to \infty$.
In the following we give a brief sketch of how one finds the trapped null geodesics. For a detailed discussion of the geodesic flow we refer the reader to \cite{ONeillKerr} or \cite{ChandBlackHole}.

The starting point for the investigation of the behaviour of the geodesics in the Kerr spacetime is the observation that the geodesic flow separates. An affinely parametrised \emph{null} geodesic \(\gamma(s) = \big((t(s), r(s), \theta(s), \varphi(s)\big)\) satisfies the following first order equations:
\begin{align}
\rho^2 \dot{t} &= a \Db + (r^2 + a^2) \frac{\Pb}{\Delta} \label{teq} \\
\rho^4 (\dot{r})^2 &= R(r) := -\mathcal{K}\Delta + \Pb^2 \label{req}\\
\rho^4 (\dot{\theta})^2 &= \Theta(\theta) := \mathcal{K} - \frac{\Db^2}{\sin^2\theta} \label{theteq} \\
\rho^2 \dot{\varphi} &= \frac{\Db}{\sin^2\theta} + \frac{a \Pb}{\Delta} \notag \;,
\end{align}
where \(\mathcal{K}\) is the Carter constant of the geodesic, \(\Pb(r) = (r^2 + a^2) E - La\) and \(\Db(\theta) = L - Ea\sin^2\theta\). Here,  \(E = -g(\partial_t, \dot{\gamma})\) is the energy of the geodesic\footnote{Note that \(\partial_t\) is not timelike everywhere! However, one still calls this quantity the `energy' of the null geodesic.} and \(L = g(\partial_\varphi, \dot{\gamma})\) is the angular momentum. Note that since the left hand side of \eqref{theteq} is positive, it follows that the carter constant $\mathcal{K}$ is non-negative. 

In order to find all trapped null geodesics, the investigation naturally starts with equation \eqref{req}.
The crucial observation is that a simple zero of \(R(r)\) corresponds to a turning point (in the \(r\)-coordinate) of the geodesic, while a double zero corresponds to an orbit of constant \(r\) (or to asymptotic approach).\footnote{Cf.\ Proposition 4.3.7 and Corollary 4.3.8  in Chapter 4 of \cite{ONeillKerr}} It follows that a necessary condition for a null geodesic being trapped is that the constants of motion $\mathcal{K}$, $L$, and $E$ can be chosen in such a way that either $R(r)$ has a double zero in $(r_+,\infty)$ or $R(r)$ has two simple zeros in $(r_+,\infty)$ and is non-negative in between. In the following we show that the latter case cannot occur.

We distinguish the two cases $E=0$ and $E\neq 0$. In the first case $R(r)$ is a polynomial of order two with $R(r) \to -\infty$ for $r \to \infty$ (recall that $\mathcal{K} \geq 0$). Moreover, $R(r)$ is non-negative in $[r_-,r_+]$. This shows that $R(r)$ can have at most one real root in $(r_+,\infty)$.

In the case $E \neq 0$, $R(r)$ is a polynomial of order four. Over the complex numbers, we can write $R(r)$ as
\begin{equation*}
R(r) = E^2\cdot (r - \lambda_1)(r - \lambda_2)(r - \lambda_3)(r-\lambda_4) = E^2 \cdot r^4 - E^2 (\lambda_1 + \lambda_2 + \lambda_3 + \lambda_4) \cdot r^3 + \ldots \;,
\end{equation*}
where $\lambda_i \in \C$, $i \in \{1,2,3,4\}$, are the complex roots of $R(r)$. Since $R(r)$ does not have a term of order three, we see that the sum of the complex roots of $R(r)$ has to yield zero. This directly excludes $R(r)$ having four positive zeros. We also note that $R(r)$ tends to  $\infty$ for $r \to \infty$; hence for $R(r)$ to have two simple zeros in $(r_+,\infty)$ and to be non-negative in between, we see that $R(r)$ has to have at least three zeros in $(r_+,\infty)$. But since $\mathcal{K} \geq 0$, we see that $R(r)$ is non-negative in $[r_-,r_+]$; i.e., if $R(r)$ has three zeros in $(r_+,\infty)$, then it needs to have a fourth  positive zero, which we have already ruled out. This shows that trapping can only occur due to a double zero of $R(r)$.

We now sketch how one finds the values of $r$ that accommodate trapped null geodesics (along with the constants of motion $\mathcal{K}$, $L$ and $E$). A detailed discussion is found in Section 63 (c) of \cite{ChandBlackHole}. 

Without loss of generality we can assume that \(E=1\). We then need to solve
\begin{equation*}
\begin{split}
R(r) &= -\mathcal{K}(r^2 - 2mr + a^2) + (r^2 +a^2 -La)^2 =0 \\[2pt]
\frac{d}{dr} R(r) &= 2\mathcal{K} (m-r) + 4r(r^2 + a^2 -La) =0 \;.
\end{split}
\end{equation*}
Eliminating $\mathcal{K}$, we obtain the two solutions $L_1(r) = \frac{r^2 + a^2}{a}$ and $L_2(r) = \frac{r^3 + ra^2 -3mr^2 + ma^2}{a(m-r)}$. In the first case we obtain $\mathcal{K}_1(r) =0$, which characterises the principal null geodesics (cf.\ Corollary 4.2.8 in \cite{ONeillKerr}) and is thus not compatible with orbits of constant $r$. We are thus left with the second solution $L_2(r)$, which implies $\mathcal{K}_2(r) = \frac{4r^2}{(m-r)^2} \Delta$. For the further analysis it is helpful to introduce the quantity $\mathcal{Q} = \mathcal{K} - (L-a)^2$, since it simplifies the analysis of the $\theta$-motion of the geodesic. We obtain 
\begin{equation*}
\mathcal{Q}_2(r) = \frac{r^3}{a^2(m-r)^2} \big(4a^2m - r(r-3m)^2\big)\;.
\end{equation*}
It can now be shown (cf.\ Section 63 (c) of \cite{ChandBlackHole}) that if we evaluate the right hand side of \eqref{theteq} at $L_2(r)$ and $\mathcal{K}_2(r)$, where $r$ is such that  $\mathcal{Q}_2(r) <0$, then we see that it is negative for all values of $\theta$. Hence, these values of $r$ do not accommodate trapped null geodesics. However, one can show that the values of $r$ where $\mathcal{Q}_2(r) \geq 0$ indeed allow the presence of trapped null geodesics. This region is bounded by the roots \(r_\delta\) and \(r_\rho\) of $\mathcal{Q}_2(r)$ which are bigger than $r_+$.

We now show that the \(N\)-energy of a trapped null geodesic \(\gamma_{r_0}\), trapped on the hypersurface \(\{r=r_0\}\) with \(r_0 \in [r_\delta, r_\rho]\), is bounded away from zero and infinity. One way to do this is to compute the \(N\)-energy directly:
\begin{equation*}
-(N,\dot{\gamma}) = (dt + dr^* -dr)(\dot{\gamma}) = \dot{t} = \frac{1}{\rho^2} \big[a\Db(\theta) + (r_0^2 + a^2)\frac{\Pb(r_0)}{\Delta(r_0)}\big]
\end{equation*}
where we have used equation \eqref{teq}. A further analysis of the behaviour of the \(\theta\) component of \(\gamma_{r_0}\) yields that its image is a closed subset of \([0,\pi]\), thus \(-(N,\dot{\gamma})(\theta)\) takes on its minimum and maximum. Since \(-(N,\dot{\gamma})\) is always strictly positive, this immediately yields that it is  bounded away from zero and infinity.

Invoking Theorem \ref{LED} we thus obtain

\begin{theorem}[Trapping in (sub)-extremal Kerr]
\label{TrapKerr}
Let \((M,g)\) be the domain of outer communications of a (sub)-extremal Kerr spacetime, foliated by the level sets of a time function \(t^*\) as above. Moreover, let \(N\) be the timelike vector field from above and \(\mathcal{T}\) an open set with the property that for all \(\tau \geq 0\) we have \(\mathcal{T} \cap \Sigma_\tau \cap [r_\delta, r_\rho] \neq \emptyset\). Then there is no function \(P:[0,\infty) \to (0,\infty)\) with \(P(\tau)\to 0\) for \(\tau \to \infty\) such that 
\begin{equation*}
E^N_{\tau, \mathcal{T}\cap \st}(u) \leq P(\tau) E^N_0(u)
\end{equation*}
holds for all solutions \(u\) of the wave equation.
\end{theorem}

Note that the same remark as made in footnote \ref{suitable} on page \pageref{suitable} applies, i.e., the theorem remains true if we choose a different timelike vector field \(N\) which commutes with the Killing vector field \(\partial_t\) and also if we choose a different foliation by timelike slices, i.e., a different time function\footnote{In the latter case one may have to alter the decay statement for the function \(P\), i.e., replace it with \(P(\tau) \to 0\) for \(\tau \to \tau^*\).}.

Another way to show that the energy of the trapped null geodesic \(\gamma_{r_0}\) is bounded away from zero and infinity is to choose a different suitable vector field \(N\). Recall that the vector fields \(\partial_t\) and \(\partial_\varphi\) are Killing, and that at each point in the domain of outer communications they also span a timelike direction. We can thus find a timelike vector field \(\tilde{N}\) that commutes with \(\partial_t\) and such that in a small \(r\)-neighbourhood of \(r_0\) the vector field \(\tilde{N}\) is given by \(\partial_t + k\partial_\varphi\) with \(k \in \R\) a constant. Thus, \(\tilde{N}\) is Killing in this small \(r\)-neighbourhood and hence the \(\tilde{N}\)-energy of \(\gamma_{r_0}\) is constant.

\subsubsection{Blue-shift near the Cauchy horizon of (sub)-extremal Kerr}
In this section we show that the results of Section \ref{sub-extremal} and \ref{BlueExtreme} also hold for (sub)-extremal Kerr. The proof is completely analogous: In the above defined \((v_+, r, \theta, \varphi_+)\) coordinates a family of ingoing null geodesics with uniformly bounded energy on \(\so\) near spacelike infinity \(\iota^0\) is given by
\begin{equation*}
\gamma_{v_+^0}(s) = ( v_+^0 , -s , \theta_0 , \varphi_0  ) \;,
\end{equation*}
where \(s\in (-\infty,0)\). The same pictures as in Sections \ref{sub-extremal} and \ref{BlueExtreme} apply, along with the same spacelike hypersurfaces \(\Sigma_0\) and \(\Sigma_1\). In order to obtain regular coordinates in a neighbourhood of the Cauchy horizon, we define, starting with \((t,r,\theta,\varphi)\) coordinates in region \(II\), outgoing `null' coordinates \((v_-, r, \theta, \varphi_-)\) by \(v_- = t - r^*\) and \(\varphi_- = \varphi - \bar{r}\). These coordinates cover the regions \(II\) and \(IV\) in the sub-extremal case and regions \(II\) and \(III\) in the extremal case. In these coordinates, the tangent vector of the null geodesic \(\gamma_{v_+^0}\) takes the form
\begin{equation}
\label{cc}
\dot{\gamma}_{v_+^0} = -\frac{\partial}{\partial r}\Big|_+ = 2\frac{r^2 + a^2}{\Delta} \frac{\partial}{\partial v_-}\Big|_- - \frac{\partial}{\partial r} \Big|_- + 2 \frac{a}{\Delta} \frac{\partial}{\partial \varphi_-}\Big|_- \;,
\end{equation}
which blows up at the Cauchy horizon. It is again easy to see that the inner product with a timelike vector field, which extends smoothly to a timelike vector field over the Cauchy horizon, necessarily blows up along \(\Sigma_1\) for \(v_+^0 \to \infty\). Thus, we obtain, after invoking Theorem \ref{symbiosis},

\begin{theorem}[Blue-shift near the Cauchy horizon in sub-extremal Kerr]
\label{Blue}
Let \(\so\) and \(\Sigma_1\) be spacelike slices in the sub-extremal Kerr spacetime as indicated in the second diagram in Section \ref{sub-extremal}. Then there exists a sequence \(\{u_i\}_{i\in \mathbb{N}}\) of solutions to the wave equation with initial energy \(E^{n_{\so}}_0(u_i) =1\) on \(\so\) such that the \(n_{\Sigma_1}\)-energy on \(\Sigma_1\) goes to infinty, i.e., \(E^{n_{\Sigma_1}}_1(u_i) \to \infty\) for \(i \to \infty\).

In particular, there is no energy boundedness statement of the form \eqref{energyboundimp}.
\end{theorem} 

As before, let us state the following

\begin{conjecture}
\label{Con2}
For generic compactly supported smooth initial data on \(\Sigma_0\), the \(n_{\Sigma_1}\)-energy along \(\Sigma_1\) of the corresponding solution to the wave equation is infinite.
\end{conjecture}

Let us conclude this section with a couple of remarks:
\begin{enumerate}[i)]
\item
Obviously, an analogous statement to Theorem \ref{Blue} is true for extremal Kerr, however, one has to introduce again a suitable globally hyperbolic subset in order to be able to apply Theorem \ref{symbiosis}. 
\item
The discussion in Section \ref{Dis} carries over to the Kerr case. In particular let us stress that Conjecture \ref{Con2} only concerns \emph{sub}-extremal Kerr black holes - the same statement for extremal Kerr black holes is expected to be false. However, as for Reissner-Nordstr\"om black holes, we conjecture a \(C^k\) instability (for some finite \(k\)) at the Cauchy horizon of extremal Kerr black holes. 
\item
We leave it as an exercise for the reader to convince him- or herself that analogous versions of the Theorems \ref{TrapKerr} and \ref{Blue} also hold true for the Kerr-Newman family.
\end{enumerate}

\section*{Acknowledgements}
I would like to thank my supervisor Mihalis Dafermos for numerous instructive and stimulating discussions. Moreover, I am grateful to Mihalis Dafermos, Gustav Holzegel and Stefanos Aretakis for many useful comments on a preliminary version of this paper. Furthermore, I would like to thank the Science and Technology Facilities Council (STFC), the European Research Council (ERC), and the German Academic Exchange Service (DAAD) (Doktorandenstipendium) for their financial support.
\vspace{1.5cm}

\section*{Appendix}

The purpose of the appendix is to compare the construction of localised solutions to the wave equation using the Gaussian beam approximation with the older method which employs the `naive' geometric optics approximation. In particular, we discuss Ralston's paper \cite{Ral69} from 1969, where he proves that trapping in the obstacle problem necessarily leads to a loss of derivative in a uniform LED statement.

\appendix

\section{A sketch of the construction of localised solutions to the wave equation using the geometric optics approximation}
\label{geomop}

In this short section we outline how one can construct localised solutions to the wave equation with the help of the `naive' geometric optics approximation. Although this approach is simpler than the Gaussian beam approximation we have presented, it alone is not strong enough to prove Theorem \ref{localised}, since the naive geometric optics approximation, in contrast to the Gaussian beam approximation, breaks down at caustics. However, as mentioned in Section \ref{parsimonious} one can extend the naive geometric optics approximation through caustics using Maslov's canonical operator.

As already mentioned at the beginning of Section \ref{GaussianBeams}, the naive geometric optics approximation also considers approximate solutions of the wave equation that are of the form \(u_\lambda = a \cdot e^{i\lambda \phi}\). But, here it suffices to consider real valued functions \(a\) and \(\phi\). Also recall that we can satisfy \eqref{firstcond}, i.e.,
\(||\Box u_\lambda||_{L^2(R_{[0,T]})} \leq C\), if we require
\begin{align}
d\phi \cdot d\phi &= 0 \qquad \qquad \textnormal{ (eikonal equation)} \label{eikonal}\\
2\grad \,\phi(a) + \Box \phi \cdot a &= 0  \qquad \qquad \textnormal{ (transport equation)}\;. \label{transport}
\end{align}
Recall that one can solve the eikonal equation \(H(x,p) = H(x,d\phi) = \frac{1}{2} g^{-1}(x)(d\phi,d\phi) =0\) using the method of characteristics. The characteristic equations are
\begin{equation}
\label{characteristics}
\begin{aligned}
\dot{x}^\mu &= \frac{\partial H}{\partial p_\mu} = g^{\mu \nu}p_\nu \\
\dot{p}_\nu &= -\frac{\partial H}{\partial x^\nu} = -\frac{1}{2} \big(\frac{\partial}{\partial x^\nu} g^{\mu \kappa}\big) p_\mu p_\kappa  \;.
\end{aligned}
\end{equation}
Given initial data \(\phi\big|_{\so}\) we choose \(\no \phi\) such that \(d\phi \cdot d\phi =0\) is satisfied on \(\so\). Moreover, we assume that \(\grad\,\phi\) is transversal to \(\so\). Then the integral curves of \eqref{characteristics} sweep out a \(4\)-dimensional submanifold of \(T^*M\) - and one can show that it is Lagrangian, i.e., it is locally the graph of a function \(\phi\) which solves the eikonal equation. This ensures that a solution \(\phi\) of \eqref{eikonal} exists locally. In order to understand the obstruction for a global solution to exist, first note that \eqref{characteristics} are just the equations for the geodesic flow in the cotangent bundle. In particular, the projections of the integral curves of \eqref{characteristics} to \(M\) are geodesics \(\gamma\) with tangent vector \(\dot{\gamma} = \grad \,\phi\). Moreover, using the eikonal equation, it follows that \(\phi\) is constant along those geodesics. Thus, if two of those geodesics cross (which is called a \emph{caustic}) the solution of the eikonal equation breaks down.

\begin{center}
\def\svgwidth{7cm}
\begingroup%
  \makeatletter%
  \providecommand\color[2][]{%
    \errmessage{(Inkscape) Color is used for the text in Inkscape, but the package 'color.sty' is not loaded}%
    \renewcommand\color[2][]{}%
  }%
  \providecommand\transparent[1]{%
    \errmessage{(Inkscape) Transparency is used (non-zero) for the text in Inkscape, but the package 'transparent.sty' is not loaded}%
    \renewcommand\transparent[1]{}%
  }%
  \providecommand\rotatebox[2]{#2}%
  \ifx\svgwidth\undefined%
    \setlength{\unitlength}{415.86606445bp}%
    \ifx\svgscale\undefined%
      \relax%
    \else%
      \setlength{\unitlength}{\unitlength * \real{\svgscale}}%
    \fi%
  \else%
    \setlength{\unitlength}{\svgwidth}%
  \fi%
  \global\let\svgwidth\undefined%
  \global\let\svgscale\undefined%
  \makeatother%
  \begin{picture}(1,0.78908597)%
    \put(0,0){\includegraphics[width=\unitlength]{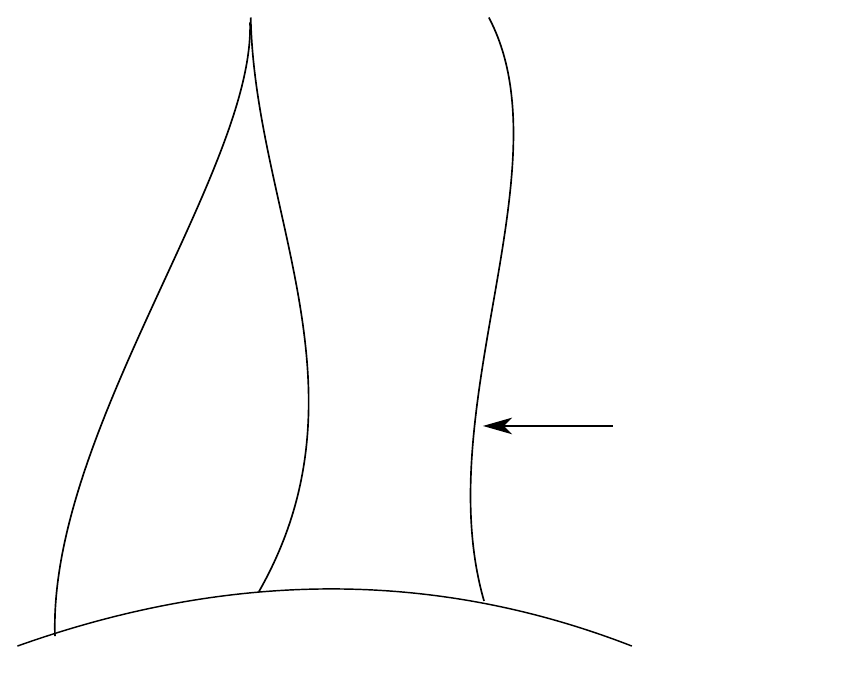}}%
    \put(0.75687876,0.03381123){\color[rgb]{0,0,0}\makebox(0,0)[lb]{\smash{\(\Sigma_0\)}}}%
    \put(0.60013315,0.48267371){\color[rgb]{0,0,0}\makebox(0,0)[lb]{\smash{\(\gamma\)}}}%
    \put(0.72073343,0.28578597){\color[rgb]{0,0,0}\makebox(0,0)[lb]{\smash{\(\phi\) is constant along \(\gamma\)}}}%
    \put(0.01407581,0.40585296){\color[rgb]{0,0,0}\makebox(0,0)[lb]{\smash{\(\phi = 3\)}}}%
    \put(0.32608997,0.60995245){\color[rgb]{0,0,0}\makebox(0,0)[lb]{\smash{\(\phi =5\)}}}%
  \end{picture}%
\endgroup%

\end{center}

Let us now consider the transport equation. Since \(\dot{\gamma} = \grad\, \phi\), \(a\) is transported along the geodesics determined by \(\phi\). Hence, the solution of \eqref{transport} has the same domain of existence as the solution of \eqref{eikonal}, and thus we see that the geometric optics approximation only breaks down at caustics.

In the context of Theorem \ref{localised}, i.e., for the purpose of the construction of localised solutions to the wave equation, recall that given a neighbourhood \(\N\) of a certain geodesic \(\gamma\) and a finite time \(T >0\), we aim for a solution \(a\) of \eqref{transport} such that \(a\) is supported in \(\N\) up to time \(T\). Therefore, we first prescribe initial data \(\phi\big|_{\so}\), \(\no \phi\) such that this particular geodesic is one of the integral curves of \eqref{characteristics}. Let us assume that there are no caustics up to time \(T\), i.e., we obtain a solution \(\phi\) of \eqref{eikonal} in \(R_{[0,T]}\). Secondly, notice that if \(a\) is initially zero at some point on \(\so\), then it vanishes on the geodesic it is transported along. Thus, by continuity we can choose \(\supp(a|_{\so})\) so small, centred around the base point of \(\gamma\), such that \(a|_{\st}\) is supported in \(\N \cap \st\) up to time \(T\), i.e., \eqref{thirdcond} is satisfied, at least up to time \(T\).

\begin{center}
\def\svgwidth{7cm}
\begingroup%
  \makeatletter%
  \providecommand\color[2][]{%
    \errmessage{(Inkscape) Color is used for the text in Inkscape, but the package 'color.sty' is not loaded}%
    \renewcommand\color[2][]{}%
  }%
  \providecommand\transparent[1]{%
    \errmessage{(Inkscape) Transparency is used (non-zero) for the text in Inkscape, but the package 'transparent.sty' is not loaded}%
    \renewcommand\transparent[1]{}%
  }%
  \providecommand\rotatebox[2]{#2}%
  \ifx\svgwidth\undefined%
    \setlength{\unitlength}{618.00424805bp}%
    \ifx\svgscale\undefined%
      \relax%
    \else%
      \setlength{\unitlength}{\unitlength * \real{\svgscale}}%
    \fi%
  \else%
    \setlength{\unitlength}{\svgwidth}%
  \fi%
  \global\let\svgwidth\undefined%
  \global\let\svgscale\undefined%
  \makeatother%
  \begin{picture}(1,0.84328847)%
    \put(0,0){\includegraphics[width=\unitlength]{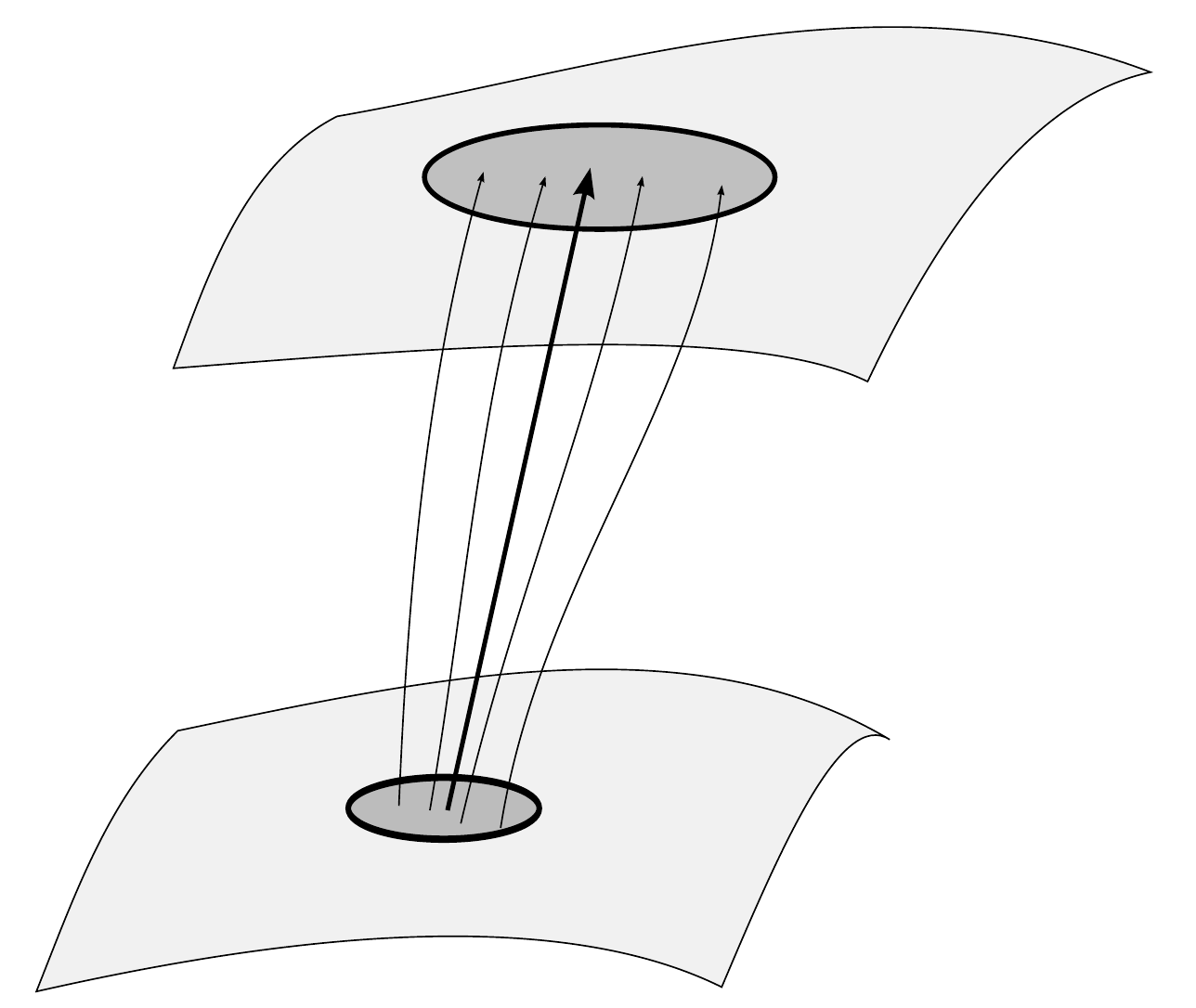}}%
    \put(0.00906143,0.11510525){\color[rgb]{0,0,0}\makebox(0,0)[lb]{\smash{\(\Sigma_0\)}}}%
    \put(0.29274742,0.79264663){\color[rgb]{0,0,0}\makebox(0,0)[lb]{\smash{\(\Sigma_T\)}}}%
    \put(0.4690183,0.15641878){\color[rgb]{0,0,0}\makebox(0,0)[lb]{\smash{\(\supp(a)\cap\Sigma_0\)}}}%
    \put(0.66181463,0.67972306){\color[rgb]{0,0,0}\makebox(0,0)[lb]{\smash{\(\supp(a)\cap\Sigma_T\)}}}%
    \put(0.45524717,0.4731556){\color[rgb]{0,0,0}\makebox(0,0)[lb]{\smash{\(\gamma\)}}}%
  \end{picture}%
\endgroup%

\end{center}
Finally, we notice that the initial energy of \(u_\lambda\) grows like \(\lambda^2\), i.e., \eqref{secondcond} is satisfied as well. This finishes then the construction of the approximate solution and one can now prove Theorem \ref{localised} \emph{under the additional assumption that no caustics form up to time \(T\)} in the same way as before, using \eqref{firstcond}, \eqref{secondcond} and \eqref{thirdcond}. Note that, although we cannot prove Theorem \ref{localised} without any further assumptions by just using the geometric optics approximation, this construction is already sufficient for obtaining solutions to the wave equation with localised energy along light rays in the Minkowski spacetime for instance, since there we can avoid the formation of caustics by a suitable choice of initial data for \(\phi\). However, in general spacetimes one cannot exclude the possibility of the formation of caustics.

\section{Discussion of Ralston's proof that trapping forms an obstruction to LED in the obstacle problem}
\label{Comparison}

The \emph{obstacle problem} is the study of the wave equation \[-\frac{\partial^2}{\partial t^2}u + \Big( \frac{\partial^2}{\partial x_1^2} + \frac{\partial^2}{\partial x_2^2} + \frac{\partial^2}{\partial x_3^2} \Big)u =0\] on \(\R \times \D\) with Dirichlet boundary conditions on \(\R \times \partial \D\), where \(\D \subseteq \R^3\) is an open set with smooth boundary and bounded complement.

Let us define an \emph{admissible light ray} to be a straight line in \(\D\) that is nowhere tangent to \(\partial\D\) and that is reflected off the boundary \(\D\) by the classical laws of ray optics.
Moreover, let \(\ell_R\) denote the supremum of the lengths of all admissible light rays that are contained in \(B_R(0)\), where \(R>0\).

\begin{center}
\def\svgwidth{5cm}
\begingroup%
  \makeatletter%
  \providecommand\color[2][]{%
    \errmessage{(Inkscape) Color is used for the text in Inkscape, but the package 'color.sty' is not loaded}%
    \renewcommand\color[2][]{}%
  }%
  \providecommand\transparent[1]{%
    \errmessage{(Inkscape) Transparency is used (non-zero) for the text in Inkscape, but the package 'transparent.sty' is not loaded}%
    \renewcommand\transparent[1]{}%
  }%
  \providecommand\rotatebox[2]{#2}%
  \ifx\svgwidth\undefined%
    \setlength{\unitlength}{542.03637695bp}%
    \ifx\svgscale\undefined%
      \relax%
    \else%
      \setlength{\unitlength}{\unitlength * \real{\svgscale}}%
    \fi%
  \else%
    \setlength{\unitlength}{\svgwidth}%
  \fi%
  \global\let\svgwidth\undefined%
  \global\let\svgscale\undefined%
  \makeatother%
  \begin{picture}(1,1.05969802)%
    \put(0,0){\includegraphics[width=\unitlength]{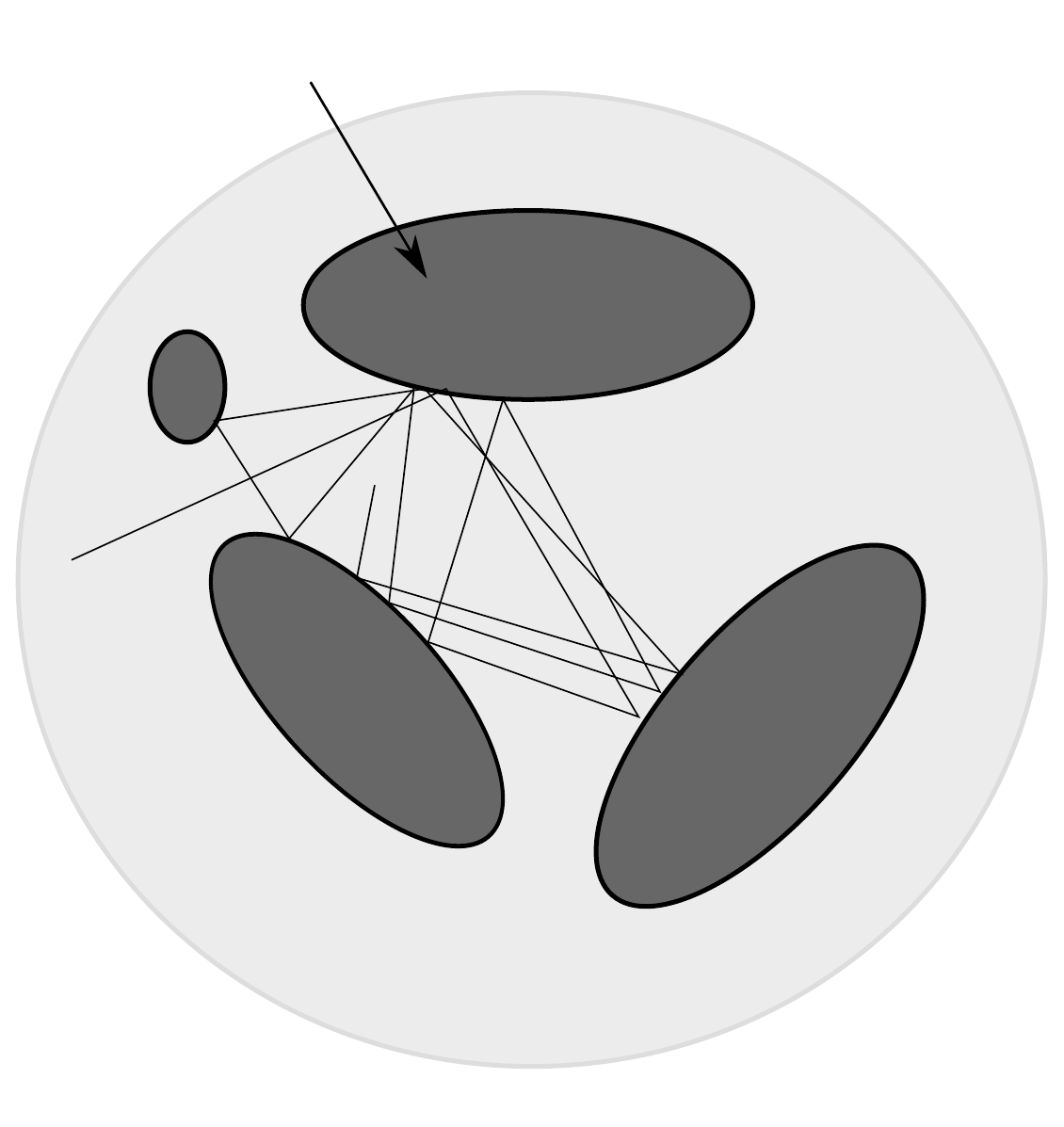}}%
    \put(0.74151346,0.68732019){\color[rgb]{0,0,0}\makebox(0,0)[lb]{\smash{\(\D\)}}}%
    \put(0.70126125,0.03657558){\color[rgb]{0,0,0}\makebox(0,0)[lb]{\smash{\(B_R(0)\)}}}%
    \put(0.08070574,1.00933816){\color[rgb]{0,0,0}\makebox(0,0)[lb]{\smash{reflecting obstacle}}}%
  \end{picture}%
\endgroup%

\end{center}

In \cite{Ral69}, Ralston proved the following
\begin{theorem}
\label{Ralstonthm}
If \(\ell_R=\infty\), then there is no uniform decay of the energy contained in \(B_R(0)\) with respect to the initial energy, i.e., there is no constant \(C>0\) and no function \(P: [0,\infty) \to (0,\infty)\) with \(P(t) \to 0\) for \(t \to \infty\) such that
\[\int_{B_R(0)} |\nabla u|^2(t,x) + |\partial_t u |^2(t,x) \;dx \leq P(t) \cdot C \cdot \int_\D |\nabla u|^2(0,x) + |\partial_t u |^2(0,x) \;dx \]
holds for all solutions \(u\) of the wave equation that vanish on \(\R \times \partial \D\) and whose initial data (prescribed on \(\{t=0\}\)) is contained in \(B_\rho(0)\) for some large, but fixed \(\rho\).
\end{theorem}
His work was motivated by a conjecture of Lax and Phillips from 1967, see \cite{LaxPhillips}, Chapter 5.3. 
They conjectured that an even stronger theorem holds, namely that the theorem above is true even without the assumption that the light rays contributing to \(\ell_R\) are nowhere tangent. However, the behaviour of such grazing rays is in general quite complicated and is still not completely understood.

In the following we will give a very brief sketch of his proof and discuss why it does not transfer directly to more general spacetimes. 
The idea, Ralston followed, to contradict the uniform rate of the local energy decay is to construct, using the geometric optics approximation, solutions with localised energy that follow one of those trapped light rays. One can implement reflections at \(\partial D\) into the construction of localised solutions via geometric optics without problems, see for example \cite{Taylor1}, Chapter 6.6.  However, one has to expect the formation of caustics and thus the breakdown of the approximate solution. Ralston addresses this difficulty by using the explicit representation formula for the wave equation in the vicinity of caustics.

Ralston starts by constructing the optical path: Given a light ray \(\gamma\) that starts at \(P_0\), he considers a small \(2\)-surface \(\Sigma_0\) through \(P_0\) such that \(\gamma\) points in the normal direction \(n_0\). This gives rise to a whole bunch of light rays that start off \(\Sigma_0\) in normal direction. 

\begin{center}
\def\svgwidth{7cm}
\begingroup%
  \makeatletter%
  \providecommand\color[2][]{%
    \errmessage{(Inkscape) Color is used for the text in Inkscape, but the package 'color.sty' is not loaded}%
    \renewcommand\color[2][]{}%
  }%
  \providecommand\transparent[1]{%
    \errmessage{(Inkscape) Transparency is used (non-zero) for the text in Inkscape, but the package 'transparent.sty' is not loaded}%
    \renewcommand\transparent[1]{}%
  }%
  \providecommand\rotatebox[2]{#2}%
  \ifx\svgwidth\undefined%
    \setlength{\unitlength}{557.41108398bp}%
    \ifx\svgscale\undefined%
      \relax%
    \else%
      \setlength{\unitlength}{\unitlength * \real{\svgscale}}%
    \fi%
  \else%
    \setlength{\unitlength}{\svgwidth}%
  \fi%
  \global\let\svgwidth\undefined%
  \global\let\svgscale\undefined%
  \makeatother%
  \begin{picture}(1,0.83175822)%
    \put(0,0){\includegraphics[width=\unitlength]{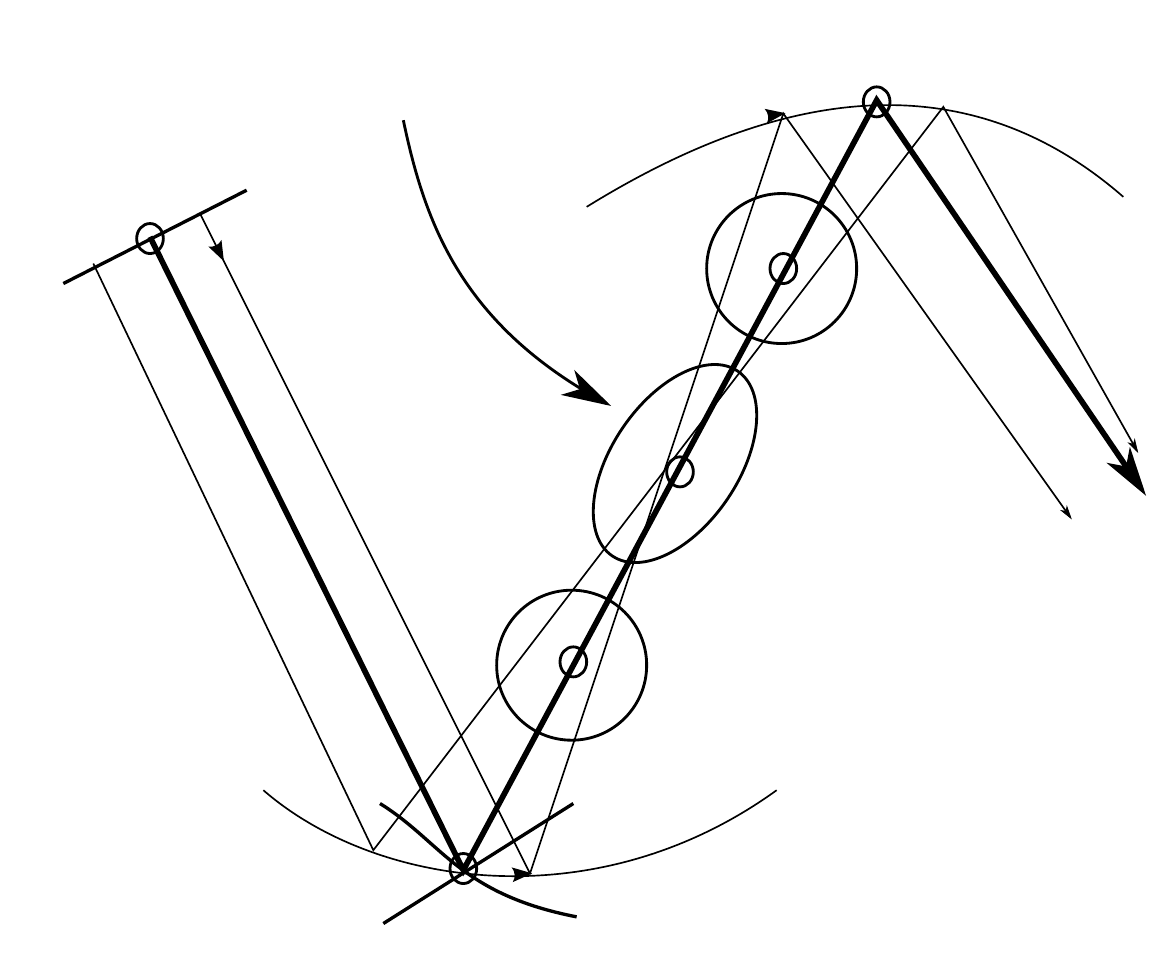}}%
    \put(0.07463074,0.647878){\color[rgb]{0,0,0}\makebox(0,0)[lb]{\smash{\(P_0\)}}}%
    \put(0.37028327,0.01925756){\color[rgb]{0,0,0}\makebox(0,0)[lb]{\smash{\(P_1\)}}}%
    \put(0.71760324,0.7827874){\color[rgb]{0,0,0}\makebox(0,0)[lb]{\smash{\(P_2\)}}}%
    \put(0.56834177,0.23166812){\color[rgb]{0,0,0}\makebox(0,0)[lb]{\smash{\(Q_1^-\)}}}%
    \put(0.72334407,0.49861652){\color[rgb]{0,0,0}\makebox(0,0)[lb]{\smash{\(Q_1^+\)}}}%
    \put(0.64871333,0.3838){\color[rgb]{0,0,0}\makebox(0,0)[lb]{\smash{\(Q_1\)}}}%
    \put(0.28991171,0.7368608){\color[rgb]{0,0,0}\makebox(0,0)[lb]{\smash{caustic}}}%
    \put(0.01435206,0.54454313){\color[rgb]{0,0,0}\makebox(0,0)[lb]{\smash{\(\Sigma_0\)}}}%
    \put(0.26694841,0.03073921){\color[rgb]{0,0,0}\makebox(0,0)[lb]{\smash{\(\Sigma_0'\)}}}%
    \put(0.50806309,0.02499838){\color[rgb]{0,0,0}\makebox(0,0)[lb]{\smash{\(\Sigma_1\)}}}%
    \put(0.25259634,0.40472823){\color[rgb]{0,0,0}\makebox(0,0)[lb]{\smash{\(\gamma\)}}}%
    \put(0.21042502,0.59566465){\color[rgb]{0,0,0}\makebox(0,0)[lb]{\smash{\(n_0\)}}}%
  \end{picture}%
\endgroup%

\end{center}

If the principal curvatures of \(\so\) at \(P_0\) are distinct from \(\frac{1}{l_0} := \frac{1}{|P_1 - P_0|}\) (which can be achieved, of course), then the normal translate \(\Sigma_0'\) of \(\so\) at \(P_1\) exists if we choose \(\so\) small enough. We then `reflect' \(\so'\) at the boundary \(\partial \D\) and obtain in this way the surface \(\Sigma_1\). This procedure is repeated, and by slight perturbations of the already constructed surfaces we can ensure the condition on the principal curvatures. Caustics are forming in a neighbourhood of \(Q_1\), which is at distance \(\frac{1}{\kappa_1}\), where \(\kappa_1\) is one of the principal curvatures of \(\Sigma_1\) at \(P_1\). Here, the normal translate of \(\Sigma_1\) fails to exist, even if we choose \(\Sigma_1\) very small. Ralston then considered two points, \(Q_1^-\) and \(Q_1^+\) on \(\gamma\), that are at distance \(\delta\) of \(Q_1\). The construction with the \(2\)-surfaces allows for an explicit construction of the phase function in the geometric optics approximation away from \(Q_1\). The phase is such that \(\grad_{\R^3} \phi\) points exactly along the light rays we have constructed. Thus, via geometric optics we can obtain a localised, approximate solution \(u_\lambda\) that propagates from a neighbourhood of \(P_0\) to a neighbourhood of \(Q_1^-\). 

To bridge the caustics, Ralston uses the explicit representation formula for solutions of the wave equation in \(\R^{3+1}\) with initial data \(u_\lambda(t=\tau_1)\) and \(\partial_t u_\lambda(t= \tau_1)\).\footnote{Actually, Ralston only considers the leading order in \(\lambda\) for the time derivative. This simplifies the computations, and works equally well, since later one takes large \(\lambda\) anyway.} Let us denote this solution with \(u_\lambda^f\). Since the initial data of \(u_\lambda^f\) is highly oscillatory, one can use the method of stationary phase to approximate \(u_\lambda^f(t=\tau_1 + 2\delta)\). Ralston does this in a uniform way and finds that \(u_\lambda^f(t=\tau_1 + 2\delta)\) is approximately localised around \(Q_1^+\) and  also has approximately the correct phase dependence to continue propagating along the preassigned optical path. Moreover, it is clear by the domain of dependence property that, if we choose \(\delta\) (and thus the \(2\)-surfaces \(\Sigma_i\)) small enough, \(u_\lambda^f\) will stay for the time \(t \in [\tau_1, \tau_1 + 2 \delta]\) in a preassigned small neighbourhood of \(\gamma\). 

At \(Q_1^+\) we go over to an approximation via geometric optics again, etc. This scheme yields a localised, approximate solution \(W_\lambda\) up to some finite time \(T\) which is patched together by the geometric optics approximations and the free space solutions. 
Hence, Ralston obtained \(||\Box W_\lambda||_{L^2(\D\times [0,T])} \leq C\), and as for the proof of Theorem \ref{localised} in Section \ref{underlying} we get our proper solution \(v\) to the initial boundary value problem with initial energy equal to one. Note that in this setting it is trivial to ensure that the energy of \(v\), that is localised in a neighbourhood of \(\gamma\), does not decay: equation \eqref{thmlocapp} states in particular that the energy of \(v\) that is outside the neighbourhood in question is smaller than \(\mu\). But the energy of \(v\) is constant and equal to one. Thus, the energy inside the neighbourhood of \(\gamma\) must be bigger than \(1-\mu\). In this way Ralston contradicts the uniform local energy decay statement and proves Theorem \ref{Ralstonthm}.
\newline
\newline
Since we are interested in the wave equation on general Lorentzian manifolds, we also have to expect the formation of caustics (see Appendix \ref{BD}). However, we do not have an explicit representation formula for solutions of the wave equation that would help us mimic Ralston's proof for the obstacle problem. Instead, one could use Maslov's canonical operator here. 
Let us also remark that the absence of a globally timelike Killing vector field allows for the phenomenon that the `trapped' energy decays or blows up. Hence, a theorem of the form \ref{main} is not needed for the obstacle problem, but it is essential for the general Lorentzian case.

\section{A breakdown criterion for solutions of the eikonal equation}
\label{BD}
We give a breakdown criterion for solutions of the eikonal equation for which a given null geodesic is a characteristic. 

\begin{theorem}
\label{breakdown}
Let \((M,g)\) be a Lorentzian manifold and \(\gamma :[0,a) \to M\) an affinely parametrised null geodesic, \(a \in (0,\infty]\). If \(\gamma\) has conjugate points then there exists no solution \(\phi : U \to \R\) of the eikonal equation \(d\phi \cdot d\phi =0\) with \(\grad \,\phi\big|_{\I\, \gamma} = \dot{\gamma}\), where \(U\) is a neighbourhood of \(\I\,\gamma\).
\end{theorem}

The theorem is motivated by the construction of localised solutions to the wave equation using the naive geometric optics approximation, where we need to find a solution of the eikonal equation for which a given null geodesic is a characteristic. It is well known that solutions of the eikonal equation break down whenever characteristics cross. However, by choosing the initial data (and thus the neighbouring characteristics) suitably one can try to avoid crossing characteristics. This is for example possible in the Minkowski spacetime. The theorem gives a sufficient condition for when no such choice is possible. 

Our proof is a minor adaptation of Riemannian methods to the Lorentzian null case, see for example \cite{EschSull76}, in particular their Proposition \(3\).
\newline
\newline
First we need some groundwork. We pull back the tangent bundle \(TM\) via \(\gamma\) and denote the subbundle of vectors that are orthogonal to \(\dot{\gamma}\) by \(N(\gamma)\). The vectors that are proportional to \(\dot{\gamma}\) give rise to a subbundle of \(N(\gamma)\), which we quotient out to obtain the quotient bundle \(\bar{N}(\gamma)\). It is easy to see that the metric \(g\) induces a positive definite metric \(\bar{g}\) on \(\bar{N}(\gamma)\) and that the bundle map \(R_\gamma : N(\gamma) \to N(\gamma)\), where \(R_\gamma(X) = R(X,\dot{\gamma})\dot{\gamma}\) and \(R\) is the Riemann curvature tensor, induces a bundle map \(\bar{R}_\gamma\) on \(\bar{N}(\gamma)\) and finally that the Levi-Civita connection \(\nabla\) induces a connection \(\bar{\nabla}\) for \(\bar{N}(\gamma)\).

\begin{definition}
\label{JacobiTensor}
\(\bar{J} \in \End\big(\bar{N}(\gamma)\big)\) is a \emph{Jacobi tensor class} iff\footnote{Here and in what follows we write \(\bar{D}_t\) for \(\bar{\nabla}_{\partial_t}\;.\)} \(\bar{D}^2_t \bar{J} + \bar{R}_\gamma \bar{J} =0\;.\)
\end{definition}

A Jacobi tensor class should be thought of as a variation field of \(\gamma\) that arises from a many-parameter variation by geodesics. It generalizes the notion of a Jacobi field (class), an infinitesimal \(1\)-parameter variation. Indeed, a solution \(\phi\) of the eikonal equation for which \(\gamma\) is a characteristic gives rise to a Jacobi tensor class \(\bar{J}\):

We denote the flow of \(\grad \, \phi\) by \(\Psi_t\) and define \(J \in \End \big(N(\gamma)\big)\) by
\begin{equation*}
J_t(X_t) := (\Psi_t)_*(X_0)\;,
\end{equation*}
where we extend \(X_t \in N(\gamma)_t\) by parallel propagation to a vector field \(X\) along \(\gamma\) whose value at \(0\) is \(X_0\). Note that \(J\) is well-defined, i.e., we have \(J_t(X_t) \in N(\gamma)\): Given \(X_0 \in T_{\gamma(0)}M\), extend it to a vector field \(\tilde{X}\) on \(M\) with \([\tilde{X}, \grad\, \phi] =0\), i.e., along \(\gamma\) we have \(\tilde{X}\big|_{\gamma(t)} = (\Psi_t)_*(X_0)\). Then
\begin{equation*}
0= \nabla_{\tilde{X}} (\grad\, \phi, \grad \, \phi) = 2 ( \nabla_{\tilde{X}}\grad\,\phi, \grad\,\phi) = 2 \nabla_{\grad\,\phi} (\tilde{X}, \grad\, \phi)\;,
\end{equation*}
from which it follows that \(\tilde{X}\big|_{\gamma(t)}\) is orthogonal to \(\grad \,\phi \big|_{\gamma(t)}\).
Moreover, \(J\) is a Jacobi tensor:\footnote{This notion is analogous to Definition \ref{JacobiTensor}, without taking the quotient.} Let \(X\) be a parallel section along \(\gamma\) and \(\tilde{X}\) an extension of \(X_0\) as above. Then
\begin{equation*}
(D_t J)(X) = D_t(JX) =D_t(\Psi_{t^*}X_0) = \nabla_{\grad\,\phi}\tilde{X} = \nabla_{\tilde{X}} \grad\,\phi = \nabla_{JX} \grad \,\phi \;.
\end{equation*}
Thus,
\begin{equation}
\label{firstderivative}
D_tJ = (\nabla \grad\,\phi) \circ J \;.
\end{equation}
Differentiating once more gives 
\begin{equation*}
(D_t^2 J)(X)= \nabla_{\grad\, \phi} (\nabla_{JX} \grad \, \phi) = R(\grad \, \phi, JX) \grad\,\phi = -R_\gamma \circ J (X) \;.
\end{equation*}
Using that \((\Psi_t)_* \big(\grad \,\phi\big|_{\gamma(0)}\big) = \grad \,\phi\big|_{\gamma(t)}\), it is now clear that \(J\) descends to a Jacobi tensor class \(\bar{J}\). Moreover, \(\bar{J}\) is non-singular, i.e., \(\bar{J}^{-1}\) exists. Since the metric \(\bar{g}\) is non-degenerate, we can form adjoints of sections of \(\End\big(\bar{N}(\gamma)\big)\), what we will denote by \(^*\). Note also that \((\bar{D}_t \bar{J}) \bar{J}^{-1}\) is self-adjoint. This follows from \eqref{firstderivative} and the fact that \(\nabla \nabla \phi\) is symmetric. We now prove the theorem.

\begin{proof}[Proof of Theorem \ref{breakdown}:]
Assume there exists such a solution \(\phi\) of the eikonal equation. Say the points \(\gamma(t_0)\) and \(\gamma(t_1)\) are conjugate, \(0 \leq t_0 < t_1 < a\), and \(\bar{J}\) is the Jacobi tensor class induced by  \(\phi\) as discussed above. Using the identification of \(\End\big(\bar{N}(\gamma)_t\big)\) with \(\End\big(\bar{N}(\gamma)_{t_0}\big)\) via parallel translation, we write
\begin{equation*}
\bar{K}(t) := \bar{J}(t) C \int_{t_0}^t \big(\bar{J}^* \bar{J}\big)^{-1}(\tau) \,d\tau\;,
\end{equation*}
where \(C= \bar{J}^{-1}(t_0) \bar{J}^*(t_0) \bar{J}(t_0)\). A straightforward computation shows that \(\bar{K}\) is a Jacobi tensor class with \(\bar{K}(t_0) =0\) and \(\bar{D}_t\bar{K}(t_0) = \id\). Moreover, \(\bar{K}(t)\) is non-singular for \(t>t_0\). 

On the other hand there exists a Jacobi field \(Y\) with \(Y(t_0)=0\) and \(Y(t_1)=0\). In particular, this implies that \(Y\) is a section of \(N(\gamma)\). The Jacobi field \(Y\) induces a non-trivial Jacobi field class \(\bar{Y}\) that vanishes at \(t_0\) and \(t_1\). However, a Jacobi field class is uniquely determined by its value and velocity at a point. Parallely propagating \(\bar{D}_t\bar{Y}\big|_{t_0}\) gives rise to a vector field class \(\bar{Z}\). \(\bar{K}\bar{Z}\) is then a Jacobi field class  that has the same value and velocity as \(\bar{Y}\) at \(t=t_0\), thus \(\bar{K}\bar{Z}= \bar{Y}\). This, however, contradicts \(\bar{K}\) being non-singular for \(t > t_0\).
\end{proof}

\bibliographystyle{acm}
\bibliography{Bibly}
\end{document}